\documentclass[12pt]{amsart}
\usepackage{graphics,graphicx} 
\usepackage{fancybox}
\usepackage{mathtools}

\usepackage[T1]{fontenc}

\usepackage[cp1250]{inputenc}
\usepackage[arrow, matrix, curve]{xy}
\usepackage{amsmath,amsthm,amssymb}
\usepackage{mathabx}
\usepackage{amscd}
\xyoption{frame}

\newenvironment{Proof of}[1]{\emph{Proof of #1.}}{$\qquad \square$\par}

\DeclareMathOperator{\hull}{hull}
\DeclareMathOperator{\dashind}{-Ind}

\DeclareMathOperator{\Prim}{Prim}
\DeclareMathOperator{\fin}{fin}
\DeclareMathOperator{\Prime}{Prime}
\DeclareMathOperator{\End}{End}
\DeclareMathOperator{\Pos}{Pos}

\DeclareMathOperator{\Irr}{Irr}

\DeclareMathOperator{\clsp}{\overline{span}}

\newcommand{\I}{\mathcal I}

\newcommand{\K}{\mathcal K}

\newcommand{\h}{\widehat h}

\newcommand{\VV}{\mathcal V}

\newcommand{\morp}{\varrho}

\newcommand{\LL}{\mathcal{L}}

\newcommand{\al}{\alpha}

\newcommand{\OO}{\mathcal O}

\newcommand{\B}{\mathcal B}

\newcommand{\SA}{\widehat{A}}
\newcommand{\SI}{\widehat{I}}
\newcommand{\SB}{\widehat{B}}

\newcommand{\C}{\mathbb C}

\newcommand{\J}{\mathcal J}
\newcommand{\Z}{\mathbb Z}
\newcommand{\N}{\mathbb N}
\newcommand{\T}{\mathbb T}
\newcommand{\TT}{\mathcal T}

\newtheorem{thm}{Theorem}[section]\newtheorem{lem}[thm]{Lemma} 
\newtheorem{prop}[thm]{Proposition} 
\newtheorem{cor}[thm]{Corollary}
\theoremstyle{definition} 
\newtheorem{defn}[thm]{Definition}
\newtheorem{ex}[thm]{Example}
\newtheorem{rem}[thm]{Remark}

\title[Pure infiniteness and ideal structure of cross-sectional $C^*$-algebras]{Pure infiniteness and ideal structure of $C^*$-algebras associated to Fell bundles}

\author{Bartosz Kosma  Kwa\'sniewski}
\address{Department of Mathematics and Computer Science, The University of Southern Denmark, 
Campusvej 55, DK--5230 Odense M, Denmark// Institute of Mathematics,  University  of Bia{\l}ystok\\
ul. Akademicka 2,  PL-15-267  Bia{\l}ystok,   Poland}
\email{bartoszk@math.uwb.edu.pl}

\author{Wojciech Szyma\'nski}
\address{Department of Mathematics and Computer Science, The University of Southern Denmark, 
Campusvej 55, DK--5230 Odense M, Denmark}
\email{szymanski@imada.sdu.dk}



\begin{document}
\begin{abstract}
We investigate structural properties of the reduced cross-sectional algebra $C^*_r(\B)$ of a Fell bundle $\B$ over a discrete group $G$.
Conditions allowing one to determine the ideal structure of $C^*_r(\B)$  are studied.  Notions of aperiodicity, paradoxicality and $\B$-infiniteness for the Fell bundle $\B$ are introduced and investigated by themselves and in relation to the partial dynamical system dual to $\B$. Several criteria of pure infiniteness of  $C^*_r(\B)$ are given. It is shown that they generalize and unify  corresponding results obtained in the context of crossed products,  by the following duos: Laca, Spielberg \cite{Laca-Spiel}; Jolissaint, Robertson \cite{Joli-Robert};  Sierakowski, R{\o}rdam  \cite{rordam_sier};  Giordano, Sierakowski \cite{gs} and 
Ortega, Pardo \cite{Ortega-Pardo}.

For exact, separable Fell bundles satisfying the residual intersection property primitive ideal space of $C^*_r(\B)$ is determined. The results of the paper are shown to be optimal when applied to graph $C^*$-algebras. Applications to a class of Exel-Larsen crossed products are presented. 
\end{abstract}
\maketitle

\setcounter{tocdepth}{1}

\section{Introduction}

Many of $C^*$-algebras studied in literature are equipped with a natural additional structure which can be used to study their properties. This structure can be exhibited by a  group co-action (or a group action if the underlying group is abelian) or  more generally  by a group grading of the 
$C^*$-algebra. It allows one to investigate the $C^*$-algebra by means of the associated Fell bundle of subspaces determining the grading.  
Fell bundles over discrete groups  proved to be  a  convenient framework for studying  crossed products corresponding to global or partial group actions, and were successfully applied to  diverse classes of $C^*$-algebras, \cite{exel-book}, \cite{gs}, \cite{AA}. 
Moreover,  the approach based on Fell bundles  has recently gained  an increased interest in an analysis of $C^*$-algebras associated to generalized graphs \cite{BSV}, Nica-Pimsner algebras \cite{CLSV}, and Cuntz-Pimsner algebras  \cite{kwa-szym}, \cite{AM} associated to product systems over semigroups. 
We remark that,  in contrast to most of applications in  \cite{exel-book}, \cite{gs}, \cite{AA}, in the latter case the core $C^*$-algebra corresponding to the unit in the group, as a rule, is non-commutative. 
The present paper is devoted to investigations of the ideal structure, pure infiniteness  and related features of the reduced cross-sectional algebras $C^*_r(\B)$ 
arising from a Fell bundle $\B=\{B_t\}_{t\in G}$ over a discrete group $G$ with the unit fiber $B_e$ being genuinely a non-commutative $C^*$-algebra. 
One of our primary aims is to give  convenient $C^*$-dynamical conditions on $\B$ that lead to a coherent treatment 
unifying various approaches to  pure infiniteness of crossed products by group actions \cite{Laca-Spiel}, \cite{Joli-Robert}, \cite{rordam_sier}, \cite{gs}, and that are applicable to $C^*$-algebras arising from semigroup structures. Actually, for a class of Fell bundles we consider, the $C^*$-algebra $C^*_r(\B)$ has the ideal property, and it is known that in the presence of this property pure infiniteness  \cite[Definition 4.1]{kr} is equivalent to strong pure infiniteness \cite[Definition 5.1]{kr2}. Additionally, if $C^*_r(\B)$ is separable  we provide a  description of the primitive spectrum of $C^*_r(\B)$. This together with known criteria for nuclearity of $C^*_r(\B)$, cf.  \cite[Proposition 25.10]{exel-book}, form  a full toolkit  for producing and analyzing graded $C^*$-algebras  that undergo Kirchberg's classification (up to stable isomorphism)  via ideal system equivariant  KK-theory \cite{kirchberg}.

In order to detect pure infiniteness of a non-simple $C^*$-algebra, one needs to understand its ideal structure. The general algebraic necessary and sufficient conditions assuring that the ideals in the ambient algebra are uniquely determined by their intersection with the core are known. These conditions are exactness and the residual intersection property. They were introduced in the context of crossed products in \cite{s}, then generalized to partial crossed products in \cite{gs} and  to cross-sectional algebras in \cite{AA}. We give a metric characterisation of the intersection property using a notion of topological grading, and we shed  light on the notion of exactness of a Fell bundle $\B=\{B_t\}_{t\in G}$ by characterising it in terms of graded and Fourier ideals in  $C^*_r(\B)$. 

An important dynamical condition implying  the (residual) intersection property  of $\B$ is (residual) topological freeness of a dual partial dynamical system $(\{\widehat{D}_t\}_{t\in G}, \{\h_t\}_{t\in G})$ defined on the spectrum $\widehat{B}_e$  of the core $B_e$. This result is well-known for crossed products, cf. \cite{Arch_Spiel}. Recently, it was generalized to cross-sectional algebras of saturated Fell bundles by the authors of the present paper \cite{kwa-szym}, and to general Fell bundles by Beatriz Abadie and Fernando Abadie \cite{AA}. 
 The system $(\{\widehat{D}_t\}_{t\in G}, \{\h_t\}_{t\in G})$ is very useful in investigation of the ideal structure of $C^*_r(\B)$. In particular, it factorizes to a partial dynamical system on the primitive spectrum $\Prim(B_e)$ of $B_e$, and we show that for exact, separable Fell bundles satisfying the residual intersection property the primitive ideal space of $C^*_r(\B)$ can be identified with the quasi-orbit space of this dual action on $\Prim(B_e)$. We show below that this result applied to  graph $C^*$-algebras $C^*(E)$  with their natural $\Z$-gradings  gives a new way of determining primitive ideal space  of $C^*(E)$ for an arbitrary graph $E$ satisfying Condition (K). The latter description was originally obtained in \cite{bhrs} by different methods.

In general, the aforementioned dual system is not  well suited for  determining pure infiniteness of $C^*_r(\B)$, as it gives no control on positive elements.    Therefore we introduce a concept of 
aperiodicity for  Fell bundles, which is related to the aperiodicity condition for $C^*$-correspondences 
introduced by Muhly and Solel in \cite{MS}. One should note that the origins of this notion go back 
to the work of Kishimoto \cite{kishimoto} and Olesen and Pedersen \cite{OlPe} where the close relationship between this condition and properties of the Connes spectrum were revealed. More recently, similar aperiodicity conditions were investigated 
in the context of partial actions by Giordano and Sierakowski in \cite{gs}. 
 The precise relationship between aperiodicity and  topological freeness is not clear, however we prove that, under the additional hypothesis that  the primitive ideal space of $B_e$ is Hausdorff, topological freeness of the partial  dynamical system on $\Prim(B_e)$ implies aperiodicity of $\B$.
We show that  a Fell bundle associated to a graph $E$ is aperiodic if and only if $E$ satisfies Condition (L).

Exploiting ideas of R\o rdam and Sierakowski \cite{rordam_sier}, modulo observations made in \cite{kwa-endo}, we prove that if a Fell bundle $\B$ is exact, residually aperiodic, and $B_e$ has the ideal property or contains finitely many $\B$-invariant ideals\footnote{In the initially submitted manuscript we considered only the case when $B_e$ has the ideal property}, then $C^*_r(\B)$  has the ideal property and  pure infiniteness of  $C^*_r(\B)$ is equivalent to proper infiniteness  of every non-zero positive element in $B_e$ (treated as an element in $C^*_r(\B)$). If additionally $B_e$ has real rank zero then  pure infiniteness of  $C^*_r(\B)$ is equivalent to proper infiniteness of every non-zero projection in $B_e$. One can find many  different dynamical conditions implying proper infiniteness of every non-zero positive element in $B_e$. For instance, in the context of group action this holds  for strong boundary actions \cite{Laca-Spiel}, $n$-filling actions \cite{Joli-Robert}, and paradoxical actions \cite{rordam_sier}, \cite{gs}. We note that $n$-filling actions generalize strong boundary actions and are necessarily minimal and paradoxical actions. However, paradoxical actions  considered in \cite{rordam_sier}, \cite{gs}, are acting on totally disconnected spaces, while actions studied in \cite{Laca-Spiel}, \cite{Joli-Robert} do not have this restriction.  The notion of a paradoxical set can be naturally generalized to the setting of Fell bundles, and we define  $\B$-paradoxical elements for an arbitrary Fell bundle $\B$.  However,  we found that in general $\B$-paradoxicality is hard to be checked  in practice. Therefore we also introduce a weaker notion of residually $\B$-infinite elements in $B_e$.  We prove that if, in addition to previously mentioned assumptions on $\B$, every non-zero positive element  in $B_e$ is Cuntz equivalent to a residually infinite element, then $C^*_r(\B)$ is purely infinite.
This result can be viewed as a strengthening and unification of all the aforementioned results, as we show that for $n$-filling actions considered in \cite{Joli-Robert} every non-zero positive element in $B_e$ is residually $\B$-infinite for the corresponding Fell bundle. Moreover, we prove that for a graph $C^*$-algebra and the associated Fell bundle our conditions for pure infiniteness are not only sufficient but also necessary.  

Apart from already mentioned applications to partial crossed products and graph $C^*$-algebras, we use  the results of the present paper to study semigroup corner systems $(A,G^+,\alpha,L)$ and their crossed products. These objects are important as they lie on the intersection of various approaches to semigroup crossed products. We explain below that $(A,G^+,\alpha,L)$  can be equivalently treated as an Exel-Larsen system \cite{Larsen}, a semigroup of endomorphisms $\alpha=\{\alpha_t\}_{t\in G^+}$, a semigroup of retractions (transfer operators) $L=\{L_t\}_{t\in G^+}$, or a group interaction $\VV=\{\VV_g\}_{g\in G}$ in the spirit of \cite{exel4}. The  semigroup $G^+$ we consider is a positive cone in a totally ordered abelian group $G$, and the maps act on an arbitrary $C^*$-algebra $A$. To any corner system $(A,G^+,\alpha,L)$ we associate a Fell bundle $\B$ and define the corresponding crossed product $A\rtimes_{\alpha,L}G^+$ to be the cross-sectional algebra $C^*(\B)$. Then we identify $A\rtimes_{\alpha,L}G^+$ as a universal $C^*$-algebra with respect to certain representations. Thus we see that in the unital case,  $A\rtimes_{\alpha,L}G^+$ coincide with the crossed product  constructed, using more direct methods, in \cite{Kwa-Leb}. We also conclude that $(A,G^+,\alpha,L)$ coincides with Exel-Larsen crossed product introduced in \cite{Larsen}.  We manage to formulate in a natural  way the  constructions and results for the Fell bundle $\B$  in terms of the systems $\alpha$, $L$ and $\VV$. This gives us  several completely new results, including  description of ideal structure,  the primitive ideal space of $A\rtimes_{\alpha,L}G^+$, and criteria for pure infiniteness of $A\rtimes_{\alpha,L}G^+$. In particular, in the case $G^+=\N$, our pure infiniteness criteria imply the result of Ortega and Pardo  \cite{Ortega-Pardo}, cf. Remark \ref{Eduard remark} below.

The paper is organized as follows.

After some preliminaries, in Section 3, we discuss notions of exactness, intersection property, and topological 
freeness for a Fell bundle, recently introduced in the context of bundles in \cite{AA}. In particular, we give
 convenient characterizations of exactness (Proposition \ref{characterization of exactness}) and the intersection property (Proposition 
\ref{intersection and uniqueness for Fell bundles}). 

In Section 4  we study the concept of 
aperiodicity for Fell bundles. We note that aperiodicity of a Fell bundle implies the intersection property (Corollary 
\ref{aperiodicity imply intersection property}) and indicate its relation to topological freeness (Proposition \ref{twierdzenie do sprawdzenia}).  In the main result of this section (Theorem 
\ref{pure infiniteness for general Fell bundles}) we give a characterization of pure infiniteness of the 
reduced cross-sectional algebra $C^*_r(\B)$ of an exact, residually aperiodic Fell bundle $\B$ whose unit fiber $B_e$  
 has the ideal property. 

Section 5 is devoted to investigation of conditions implying proper infiniteness of elements in the core $B_e$ of the cross sectional $C^*$-algebra $C^*_r(\B)$.  We introduce $\B$-paradoxical elements and  closely related residually $\B$-infinite elements for a  Fell bundle $\B$. We clarify the relationship between these notions and other conditions of this type studied in the literature. The main result of this section (Theorem 
\ref{pure infiniteness for paradoxical Fell bundles}) contains a criterion of pure infiniteness of the reduced 
cross-sectional algebra of a Fell bundle $\B$, phrased in terms of $\B$-infinite elements.

We describe the primitive ideal space of $C^*_r(\B)$ in Section 6. 
More specifically, the main result of this section (Theorem 
\ref{Primitive ideal space description}) identifies the primitive ideal space of the reduced cross-sectional 
algebra of a separable, exact Fell bundle $\B$ satisfying the residual intersection property both with the space 
of $\B$-primitive ideals of the trivial fiber (cf. Definition \ref{B-primitive}) and with the quasi-orbit space 
associated to the partial action dual to $\B$. 

In Section 7, we test the results of this paper against graph $C^*$-algebras $C^*(E)$ equipped with their natural grading over $\Z$. 
We show that aperiodicity of the associated Fell bundle is equivalent to Condition (L) on the graph $E$. Likewise, residual aperiodicity of that bundle is equivalent 
to Condition (K) on the graph.  We use our general results to get an alternative way of determining the primitive ideal space of $C^*(E)$ for $E$ satisfying Condition (K)  (Corollary \ref{primitives for graphs}). Finally, we show that our criterion of pure infiniteness is optimal in the sense 
that in the case of graph $C^*$-algebras it is not only sufficient but also necessary  (Theorem 
\ref{pure infiniteness of graph algebras}). 

In Section 8, we present various equivalent points of view on a semigroup corner system $(A,G^+,\alpha,L)$. We associate to $(A,G^+,\alpha,L)$  a Fell bundle $\B$   (Proposition \ref{Fell bundle of corner system}). This allows us to define the crossed product $A\rtimes_{\alpha,L}G^+$ as a cross sectional algebra of $\B$ (Definition \ref{definition of semigroup crossed product}). We describe $A\rtimes_{\alpha,L}G^+$  as a universal $C^*$-algebra for certain representations of $(A,G^+,\alpha,L)$  (Theorem \ref{Description  by generators and relations}) and we show it is isomorphic to Exel-Larsen crossed product (Corollary \ref{Exel-Larsen crossed product}). The main structural results on $A\rtimes_{\alpha,L}G^+$ are criteria of faithfulness of representations of  $A\rtimes_{\alpha,L}G^+$, description of ideal structure and primitive ideal space (Theorem \ref{first of main results of the section}),   and the criteria for pure infiniteness of $A\rtimes_{\alpha,L}G^+$ (Theorem \ref{pure infiniteness for paradoxical corner systems}).

\subsection{Acknowledgements}
The first named author was supported by a Marie Curie Intra European Fellowship within the 7th
European Community Framework Programme FP7-PEOPLE-2013-IEF; project `OperaDynaDual' number 621724 (2014-2016). The second named author was supported by the FNU Project Grant `Operator algebras, dynamical systems and quantum information theory' (2013--2015) and by Villum Fonden Research Grant  `Local and Global Structures of Groups and their Algebras'  (2014-2018).

\section{Preliminaries}

\subsection{$C^*$-algebras, positive elements and  ideals}
Let  $A$ be a $C^*$-algebra. By $1$ we  denote  the unit in the multiplier $C^*$-algebra $M(A)$.  All ideals in  $C^*$-algebras  are assumed to be closed and two-sided. All homomorphisms between $C^*$-algebras are by definition $*$-preserving.  For  actions $\gamma\colon A\times B\to C$ 
such as  multiplications, inner products, etc., we use the notation:
\begin{equation}\label{stupid convention for lazy people}
\gamma(A,B)=\clsp\{\gamma(a,b) :a\in A, b\in B\}.
\end{equation}
The set of positive elements in a $C^*$-algebra $A$ is denoted by $A^+$. In \cite{Cuntz2}, Cuntz introduced a preorder $\precsim$ on $A^+$, which nowadays is called \emph{Cuntz comparison}, cf., for instance, \cite{kr}. Namely, for any $a, b\in A^+$ we write $a \precsim b $ whenever there exists a sequence $\{x_k\}_{k=1}^\infty$ in $A$ with $x^*_k b x_k \to a$. We say two elements $a,b\in A^+$ are \emph{Cuntz equivalent} if both $a \precsim b$ and $b\precsim a$ holds.  We recall, see \cite[Definition 3.2]{kr}, that an element $a\in A^+$ is \emph{infinite} if there is $b\in A^+\setminus \{0\}$ such that $a\oplus b \precsim a\oplus 0$  in the matrix algebra $M_2(A)$.  An $a\in A^+\setminus\{0\}$ is \emph{properly infinite} if $a\oplus a \precsim a\oplus 0$. We have the following simple characterisations of these notions,  cf. \cite[Proposition 3.3(iv)]{kr}. We write $\approx_{\varepsilon}$ to indicate that $\|a- b\|< \varepsilon$, for $a,b\in A$.
\begin{lem} If  $a\in A^+\setminus\{0\}$ then 
\begin{equation}\label{infiniteness characterisation}
a\textrm{ is infinite }\Longleftrightarrow \,\, \exists_{b\in A^+\setminus\{0\}} \forall_{\varepsilon >0}\,\, \exists_{x,y\in aA}\quad x^*x\approx_{\varepsilon} a, \quad y^*y\approx_{\varepsilon} b, \quad x^*y\approx_{\varepsilon} 0,
\end{equation}
\begin{equation}\label{proper infiniteness}
a\textrm{ is properly infinite }\Longleftrightarrow \,\, \forall_{\varepsilon >0}\,\, \exists_{x,y\in aA}\quad x^*x\approx_{\varepsilon} a, \quad y^*y\approx_{\varepsilon} a, \quad x^*y\approx_{\varepsilon} 0.
\end{equation}
\end{lem}
\begin{proof} We  only show \eqref{infiniteness characterisation}. If $a$ is infinite, then there is $b\in A^+\setminus\{0\}$ such that for $\varepsilon >0$ there is a matrix 
$
z=\left(\begin{array}{c c} s  & t 
\\
 *  &* \end{array}\right)$ such that $z^* (a\oplus 0) z \approx_{\varepsilon} a\oplus b$. The last relation  implies that $s^*as\approx_{\varepsilon} a$,   $t^*at\approx_{\varepsilon} b$ and $s^*a t\approx_{\varepsilon} 0$. Hence putting $x:=a^{1/2}s$ and $y:=a^{1/2}t$ we get $x,y\in aA$ such that $x^*x\approx_{\varepsilon} a$, $y^*y\approx_{\varepsilon} b$ and $x^*y\approx_{\varepsilon} 0$. 
Now assume the condition on the right hand side of \eqref{infiniteness characterisation}. Then there is  $b\in A^+\setminus\{0\}$ and sequences $\{x_n\}\subseteq aA$, $\{y_n\}\subseteq aA$ such that $x_n^*x_n\to a$, $y_n^*y_n \to b$ and $x_n^*y_n \to 0$. Since $\{x_n\}, \{y_n\}\subseteq a^{1/2}A$ we can  find $s_n$ and $t_n$ in $A$ with $\|a^{1/2}s_n - x_n\|\to 0$ and $\|a^{1/2}t_n - y_n\|\to 0$. Then $s_n^*a s_n\to a$, $t_n^*at_n \to b$ and $s_n^*at_n \to 0$. Put $
z_n=\left(\begin{array}{c c} s_n  &t_n 
\\
0  & 0 \end{array}\right)$. Then $z_n^* (a \oplus 0) z_n \to a\oplus b$, showing that $a\oplus b \precsim a \oplus 0$.
\end{proof}
We will often exploit \cite[Proposition 3.14]{kr}  which says that $a\in A^+\setminus\{0\}$ is properly infinite if and only if $a+I$ in $A/I$ is either zero or infinite for every ideal $I$ in $A$.

 In view of \cite[Theorem 4.16]{kr},  pure infiniteness for (not necessarily simple) $C^*$-algebras as introduced in \cite{kr}, can be expressed as follows:     a $C^*$-algebra $A$ is \emph{purely infinite} if and only if every $a\in A^+\setminus\{0\}$ is properly infinite. We say that a $C^*$-algebra $A$ has \emph{the ideal property} \cite{Pasnicu}, \cite{Pas-Ror},  if  every ideal in $A$ is generated (as an ideal) by its projections. By \cite[Proposition 2.14]{Pas-Ror}, in the presence of ideal property pure infiniteness is equivalent to strong pure infiniteness \cite[Definition 5.1]{kr2}.

Let $A$ be a  $C^*$-algebra. We denote by $\I(A)$ the set of ideals in $A$ equipped with the
\emph{Fell topology} for which a sub-base of open sets is given by the sets of the form
$$
U_I := \{J \in \I(A): J \nsupseteq  I \},\qquad  I \in \I(A).
$$
 If $A$ and $B$ are two $C^*$-algebras, then   $h: \I(A)\to \I(B)$ is a homeomorphism if and only if it is a bijection  which preserves inclusion of ideals.

We denote by $\Irr(A)$  the set of all irreducible representations of $A$, and let $\Prim(A):=\{\ker\pi: \pi \in \Irr(A)\}$ be the set of primitive ideals in $A$.  
Fell topology restricted to $\Prim(A)$  is the usual Jacobson topology. We have a one-to-one correspondence between  closed sets in $\Prim(A)$ and ideals in $A$ given by:
$
\hull(I):=\{P\in \Prim(A): P \supseteq I\}$ and  $I=\bigcap_{P\in \hull(I)} P$, for all $I\in \I(A)$.
For any ideal $I\in \I(A)$ we have mutually inverse maps $P_I\mapsto P:=\{a\in A: aI\subseteq I\}$ and  $ P \mapsto P_I:=P\cap I$ that allow us to identify $\Prim(I)$ with the open set $\Prim(A)\setminus\hull(I)$:
\begin{equation}\label{identification equation}
\Prim(I)=\{P\in \Prim(A): P \nsupseteq I\}.
\end{equation}
The above identification extends  to hereditary subalgebras. Namely, for any hereditary $C^*$-subalgebra $B$ of $A$ the map $ P \mapsto P\cap B$ allows us to assume the identification $\Prim(B)=\{P\in \Prim(A): P \nsupseteq B\}$.

For any $\pi\in \Irr(A)$ we denote by $[\pi]$ the  unitary equivalence class of $\pi$. Then  $[\pi] \to \ker \pi$ is a well defined surjection from the spectrum  
$\SA:=\{[\pi]: \pi\in \Irr(A)  \}$ of $A$ onto  $\Prim(A)$, which induces Jacobson topology on $\SA$. Identification \eqref{identification equation} lifts to the following identification on the level of spectra:
\begin{equation}\label{identification of ideals}
\SI=\{[\pi] \in \SA: \ker\pi   \nsupseteq  I\}, \qquad I\in \I(A).
\end{equation}
More generally,  for any hereditary $C^*$-subalgebra $B$ of $A$  identification $\Prim(B)=\{P\in \Prim(A): P \nsupseteq B\}$ lifts to the identification $\widehat{B}=\{[\pi] \in \SA: \ker\pi   \nsupseteq  B\}$.

We recall that  $I\in \I(A)$  is a \emph{prime ideal} if for any pair of  $J_1,J_2\in \I(A)$ with $J_1\cap J_2 \subseteq I$  either $J_1\subseteq I$ or $J_2\subseteq I$. We denote by $\Prime(A)$  the set of prime ideals in $A$ and equip it with Fell topology. It is well known that  $\Prim(A)\subseteq \Prime(A)$ and if $A$ is separable, then   actually $\Prim(A)=\Prime(A)$. Let us note that using the identification \eqref{identification of ideals}  the inclusion $\Prim(A)\subseteq \Prime(A)$ actually  means that 
$$
\widehat{I\cap J}=\widehat{I}\cap \widehat{J}, \qquad \textrm{for all } I,J \in  \I(A).
$$

\subsection{Hilbert bimodules and induced representations} 
Let $A$ and $B$ be $C^*$-algebras. Following \cite[1.8]{BMS}, by an   \emph{$A$-$B$-Hilbert bimodule} we mean  a linear space $X$ which is both a left Hilbert $A$-module and  right Hilbert $B$-module with  the corresponding inner products  ${_A\langle} \cdot , \cdot \rangle:X\times X \to A$ and ${\langle} \cdot , \cdot \rangle_B:X\times X \to B$ satisfying:
$$
x\langle y , z \rangle_B={_A\langle} x , y \rangle z, \qquad x,y,z\in X.
 $$ 
Note that then   $X$ establishes  a Morita-Rieffel equivalence between the  ideals ${_A\langle} X , X \rangle \in \I(A)$  and $\langle X , X \rangle_B \in \I(B)$. We recall \cite{Rieffel}, see \cite[Definition 3.1]{morita}, that a \emph{Morita-Rieffel} (or \emph{imprimitivity}) $A$-$B$-bimodule is an   $A$-$B$-Hilbert bimodule $X$ such that ${_A\langle} X , X \rangle=A$ and $\langle X , X \rangle_B=B$.  For the Morita-Rieffel bimodule $X$ the formula
$$
\I(B) \ni I  \to {_A\langle} XI , X \rangle \in \I(A)
$$
defines a homeomorphism $\I(B)\cong  \I(A)$ called \emph{Rieffel correspondence} \cite[Proposition 3.24]{morita}. This correspondence restricts to the homeomorphism $h_X:\Prim(B)\to \Prim(A)$ called \emph{Rieffel homeomorphism} \cite[Proposition 3.3]{morita}. The latter has a lift to a homeomorphism $\h_X:\SB\to \SA$ also called \emph{Rieffel homeomorphism}.

More specifically, let $X$ be an $A$-$B$-Hilbert bimodule and let $\pi:B\to \B(H_\pi)$ be a representation. 
We define  a Hilbert space $X\otimes_\pi H_\pi$ to be a Hausdorff completion  of the  tensor product vector space $X\otimes H$ with the semi-inner-product given by
$$
\langle x_1\otimes_\pi h_1, x_2\otimes_\pi h_2 \rangle_{\C} = \langle h_1,\pi(\langle x_1, x_2 \rangle_{A})h_2\rangle_{\C}.
$$
Then the  formula
$$
X\dashind^A_B(\pi)(a)  (x\otimes_\pi h) = (a x)\otimes_\pi h, \qquad  a\in A,
$$
defines a representation $X\dashind^A_B(\pi):A\to \B(H_\pi)$ called \emph{induced representation}. If $X$ is a Morita-Rieffel bimodule then the formula
$$
\h_X([\pi])=[X\dashind^A_B(\pi)], \qquad \pi \in \Irr(B),
$$
defines the Rieffel homeomorphism $\h_X:\SB\to \SA$,  see \cite[Corollaries 3.32, 3.33]{morita}. In particular, we have  $h_X(\ker\pi)=\ker\left(X\dashind^A_B(\pi)\right)$ for any $\pi \in \Irr(B)$. 

Let $X$ be an $A$-$A$-Hilbert bimodule. In this case, we will  also call $X$ a \emph{Hilbert bimodule over $A$}. An ideal $I\in \I(A)$ is said to be \emph{$X$-invariant} if 
$
IX=XI.
$
For an $X$-invariant ideal the quotient space $X/XI$ is naturally a Hilbert bimodule over $A/I$.  In the sequel we will need the following simple fact, which is probably well-known to experts, but we lack a good reference.
\begin{lem}\label{coincidence of restrictions and quotients}
Let $X$ be a Hilbert bimodule over a $C^*$-algebra $A$ and let $I$ be an $X$-invariant ideal in $A$. Then  for any  representation $\pi:A/I\to \B(H_\pi)$ we have the following unitary equivalence of  representations of $A$:
$$
\big((X/XI)\dashind \pi\big)\circ q\cong X\dashind (\pi\circ q)
$$
where $q:A\to A/I$ is the quotient map.
\end{lem}
\begin{proof}
Note that for any $x_i\in X$ and $h_i\in H_\pi$, $i=1,...,n$, we have
\begin{align*}
\left\|\sum_{i=1}^n (x_i +XI)\otimes_\pi h_i\right\|^2&=\sum_{i,j=1}^n \langle h_i,\pi\big(\langle x_i+XI,x_j+XI\rangle_{A/AI}\big) h_j \rangle
\\
&=\sum_{i,j=1}^n \langle h_i,\pi\big(q(\langle x_i,x_j\rangle_A)\big) h_j \rangle = \|\sum_{i=1}^n x_i\otimes_{(\pi\circ q)} h_i\|^2.
\end{align*}
Accordingly, the mapping      $(x +XI)\otimes_\pi h \mapsto x \otimes_{\pi\circ q} h$, $x\in X$, $h\in H_\pi$, extends by linearity and continuity to a unitary  operator $V:X/XI\otimes_\pi H_\pi \to X\otimes_{\pi\circ q} H_\pi$. Unitary $V$  intertwines $\big((X/XI)\dashind \pi\big)\circ q$ and $X\dashind (\pi\circ q)$ because for any $a\in A$, $x\in X$, $h\in H_\pi$ we have
\begin{align*}
V \big((X/XI)\dashind \pi\big)(q(a))(x +XI)\otimes_\pi h&=  V (ax +XI)\otimes_\pi h 
\\
&=ax \otimes_{\pi\circ q} h 
\\
&= X\dashind (\pi\circ q)(a) V (x +XI)\otimes_\pi h.
\end{align*}

\end{proof}
\subsection{Partial actions}
We recall that a \emph{partial action of a discrete group $G$ on a $C^*$-algebra $A$} is a pair $\alpha= (\{D_t\}_{t \in G}, \{\alpha_t\}_{t \in G}) $, where for each ${t \in G}$,   $\alpha_t:D_{t^{-1}}\to D_t$ is an isomorphism between ideals of $A$ such that
$$
\alpha_e=id_A \quad\textrm{and}\quad  \alpha_{st} \textrm{ extends } \alpha_s\circ \alpha_t \textrm{ for } s,t \in G.
$$
The second property above  is equivalent to the following relations: $
\alpha_t(D_{t^{-1}}\cap D_{s})\subseteq  D_{ts}$ and $\alpha_s(\alpha_t(a))=\alpha_{st}(a)$ for   $a\in D_{t^{-1}}\cap D_{t^{-1}s^{-1}}$, $s,t \in G$.
The triple $(A,G,\alpha)$ is called  {\em partial $C^*$-dynamical system}. There are two $C^*$-algebras naturally associated to such systems: the full crossed product $A\rtimes_\alpha G$ and the reduced crossed product $A\rtimes_{\alpha,r} G$ (they can be defined in terms of $C^*$-algebras associated to Fell bundles, see Example \ref{delpartialsystem} below). When $D_{t}=A$ for every $t\in G$, we talk about global actions and global  $C^*$-dynamical systems.

Any partial action $\alpha$ on a commutative $C^*$-algebra  $A=C_0(\Omega)$, where $\Omega$ is a locally compact Hausdorff space, is given by 
\begin{equation}\label{partial action commutative case}
\alpha_t(f)(x) := f(\theta_{t^{-1}} (x)),\qquad  f\in C_0(\Omega_{t^{-1}} ),
\end{equation}
where $D_{t}=C_0(\Omega_t)$, $t\in G$, and  $(\{\Omega_t\}_{t\in G}, \{\theta_t\}_{t\in G}) $ is a partial action of $G$  on   $\Omega$. In general,  a \emph{partial action of $G$  on  a topological space} $\Omega$  is a pair $\theta=(\{\Omega_t\}_{t\in G}, \{\theta_t\}_{t\in G}) $ where $\Omega_t$'s  are open subsets of $\Omega$ 
and   $\theta_t:\Omega_{t^{-1}}\to \Omega_{t}$ are homeomorphisms such that
$$
\theta_e=id_\Omega \quad\textrm{and}\quad  \theta_{st} \textrm{ extends } \theta_s\circ \theta_t \textrm{ for } s,t \in G.
$$
The triple $(\Omega, G,\theta)$ is called  \emph{partial (topological) dynamical system}. In case every $\Omega_{t}=\Omega$, we say $\theta$ is a global action.

 For global actions on commutative $C^*$-algebras, it is a part of $C^*$-folklore, for the extended discussion see, for instance, \cite{Arch_Spiel} or \cite{kwa}, that   simplicity of the associated reduced partial crossed products is equivalent to minimality and topological freeness of the dual action. This result was adapted to partial actions in \cite{ELQ}. Let us  recall the relevant definitions:

Let $\theta=(\{\Omega_t\}_{t\in G}, \{\theta_t\}_{t\in G}) $ be a partial action on a (not necessarily Hausdorff) topological space $\Omega$.  A subset $V$ of $\Omega$ is $\theta$-\emph{invariant}  if 
$
\theta_t(V\cap \Omega_{t^{-1}}) \subseteq V$  for every $t\in G$ (then we actually have $\theta_t(V\cap \Omega_{t^{-1}})= V\cap \Omega_t$, for all  $t\in G$, and thus $\Omega\setminus V$ is also $\theta$-invariant). The restriction $\theta_V:=(\{\Omega_t\cap V\}_{t\in G}, \{\theta_t|_{\Omega_{t^{-1}}\cap V }\}_{t\in G})$ of $\theta$ to a  $\theta$-invariant  set $V\subseteq \Omega$ is a again a partial dynamical system, and   $\theta$ is called \emph{minimal} if there are no non-trivial closed $\theta$-invariant subsets of $\Omega$. 
The partial action $\theta$ is \emph{topologically free} if for  finite  set $F\subseteq G\setminus\{e\}$ the set $\bigcup_{t \in F} \{x\in \Omega_{t^{-1}}: \theta_{t}(x)=x\}$ has empty interior in $\Omega_e$.  We say, cf. \cite[Page 230]{s}, \cite[Definition 3.4]{gs}, that  $\theta$ is \emph{residually topologically free} if the restriction of $\theta$ to any closed  $\theta$-invariant set is topologically free.

Dynamical conditions implying pure infiniteness of reduced crossed products for global actions on totally disconnected spaces  were introduced in \cite{rordam_sier}, and adapted to the case of partial actions in  \cite{gs}. A crucial notion is that of paradoxical set,  cf. \cite[Definition 4.2]{rordam_sier}, \cite[Definition 4.3]{gs}:
\begin{defn}\label{paradoxical definition of Sierakowski and Rordam}
Let $(\{\Omega_t\}_{t\in G}, \{\theta_t\}_{t\in G}) $ be a partial action on a topological space $\Omega$. A non-empty open set $V\subseteq \Omega$ is called \emph{$G$-paradoxical} if there are open sets $V_1,...,V_{n+m}$ and elements  $t_1,...,t_{n+m}\in G$, such that 
\begin{itemize}
\item[(1)] $V=\bigcup_{i=1}^n V_i=\bigcup_{i=n+1}^{n+m} V_i,$
\item[(2)] $V_{i}\subseteq \Omega_{t_i^{-1}}$ and  $\theta_{t_i}(V_i)\subseteq V$ for all $i=1,...,n+m$,
\item[(3)] $\theta_{t_i}(V_{t_i})\cap \theta_{t_j}(V_{t_j})=\emptyset$ for all $i\neq j$.
\end{itemize}
\end{defn}
The notion of quasi-orbit space adapted to partial actions, cf. \cite[Page 5740]{gs}, is defined as follows:
\begin{defn}
 Let $(\{\Omega_g\}_{g\in G}, \{\theta_g\}_{g\in G})$ be a  partial action of a group $G$ on a topological  space $\Omega$. We let the \emph{orbit} of a point $x\in \Omega$ to be the set
 $$
 Gx:=\bigcup_{t\in G \atop x\in \Omega_{t^{-1}}} \{\theta_t(x)\}.
 $$
We define the \emph{quasi-orbit} $\OO(x)$ of $x$ to be the equivalence class of $x$ under the equivalence relation  on $\Omega$ given by 
$$
x\sim y \,\,\Longleftrightarrow\,\, \overline{Gx}=\overline{Gy}.
$$ 
We denote by $\OO(\Omega)$ the \emph{quasi-orbit space} $\Omega/\sim$ endowed with the quotient topology.
\end{defn}

\subsection{Fell bundles and graded $C^*$-algebras}

Let $G$ be a discrete group. A \emph{Fell bundle} over $G$ can be defined as a collection $\B=\{B_t\}_{t\in G}$ of closed subspaces of a $C^*$-algebra $C$ such that $B_t^*=B_{t^{-1}}$ and $B_t B_s\subseteq B_{ts}$ for all $s,t \in G$ (see \cite[Definition 16.1]{exel-book} for the  axiomatic description). These relations in particular imply that
\begin{equation}\label{ternary relation}
B_t B_{t^{-1}} B_t =B_t, \qquad t \in G,
\end{equation}
and $B_t B_{t^{-1}}$ is an ideal in the \emph{core} $C^*$-algebra $B_e$, cf. \cite[Lemma 16.12]{exel-book}. Moreover, for any $t\in G$, 
$
B_t \textrm{ is naturally a Hilbert bimodule over }B_e 
$ with right and left inner products given by $\langle x,y\rangle_{B_e}:=x^*y$ and ${_{B_e}\langle} x,y\rangle:=xy^*$, $x,y\in B_t$. Then $\langle B_t,B_t\rangle_{B_e}=B_t B_{t^{-1}}$.

 If $\B=\{B_t\}_{t\in G}$ is a Fell bundle, then the direct sum 
$\bigoplus_{t\in G} B_t$ is naturally equipped with the structure of a $*$-algebra which admits a $C^*$-norm. In general, there are many $C^*$-norms on 
$\bigoplus_{t\in G} B_t$. There is always  a maximal such norm, 
and it satisfies the inequality 
\begin{equation}\label{topological grading inequality}
\|a_e\|\leq \|\sum_{t \in G} a_t\|, \qquad \textrm{ for all } \,\,\sum_{t \in G} a_t\in \bigoplus_{t\in G} B_t,\,\,a_t\in B_t, \,\, t\in G,
\end{equation}
 see \cite[Lemma 1.3]{quigg}, \cite[Proposition 2.9]{Exel}, or \cite[Lemma 17.8]{exel-book}. 
The completion of $\bigoplus_{t\in G} B_t$  in  the maximal $C^*$-norm is denoted by $C^*(\B)$. It is called \emph{cross sectional algebra}
 of $\B$. It follows from 
\cite[Theorem 3.3]{Exel} that there is also  a minimal $C^*$-norm on $\bigoplus_{t\in G} B_t$ satisfying \eqref{topological grading inequality} and a completion of $\bigoplus_{t\in G} B_t$ in this minimal $C^*$-norm is naturally isomorphic to the \emph{reduced cross sectional algebra} $C_r^*(\B)$, as introduced in 
\cite[Definition 2.3]{Exel} or in \cite[Definition 3.5]{quigg} (both definitions are known to be equivalent). A Fell bundle  $\B=\{B_t\}_{t\in G}$ is said to be \emph{amenable} \cite[Definition 20.1]{exel-book} if the algebras $C_r^*(\B)$ and $C^*(\B)$ coincide; in other words,
  if there exists a unique $C^*$-norm on 
$\bigoplus_{t\in G} B_t$ satisfying \eqref{topological grading inequality}. It is known that if the group $G$ is amenable (or more generally 
if $\B=\{B_t\}_{t\in G}$ has the approximation property, see  \cite[Definition 20.4]{exel-book}) then the Fell bundle $\B=\{B_t\}_{t\in G}$ is automatically amenable, see  \cite[Theorem 20.7]{exel-book}.

Let $\B=\{B_t\}_{t\in G}$ be a Fell bundle. Any $C^*$-algebra $B$ which is a closure of $\bigoplus_{t\in G} B_t$ is called $\B$-\emph{graded} (or simply \emph{graded}). If  additionally the norm in $B$ satisfies   \eqref{topological grading inequality}, then   $B$ is called \emph{topologically graded} \cite[Definition 19.2]{exel-book}, \cite[Definition 3.4]{Exel}. For any topologically $\B$-graded  $C^*$-algebra $B$ the canonical projections 
$$
F_t:\bigoplus_{s\in G} B_s\to B_t, \qquad t\in G,
$$
extend to contractive linear maps on $B$, cf. \cite[Corollary 19.6]{exel-book}. They are called \emph{Fourier coefficient operators} in \cite[page 197]{exel-book}. In particular, $F_e:B\to B_e$ is a conditional expectation onto the core $C^*$-algebra $B_e$. This conditional expectation is faithful  if and only if $B=C^*_r(\B)$, see \cite[Proposition 2.12]{Exel}. 

Perhaps the most significant example of a Fell bundle is the one coming from partial actions. In particular, every separable Fell bundle whose unit fiber is stable arises in this way, see \cite[Theorem 27.11]{exel-book}.

\begin{ex}[Fell bundle associated to a partial action]
\label{delpartialsystem}
The \emph{Fell bundle $\B_\alpha=\{B_t\}_{t \in G}$ associated to a partial action} $\alpha= (\{D_t\}_{t \in G}, \{\alpha_t\}_{t \in G})$ on a $C^*$-algebra $A$ is defined as follows: $B_t:=\{a_t\delta_t:  a_t \in D_t\}$  is isomorphic as a Banach space to $D_t$ ($\delta_t$ is just an abstract marker), and multiplication and star operation are given by 
$$(a_t\delta_t)(a_s\delta_s)=\alpha_t(\alpha_{t^{-1}}(a_t)a_s)\delta_{ts}, \ \ \ (a_t\delta_t)^*=\alpha_{t^{-1}}(a_t^*)\delta_{t^{-1}}.
$$ 
In particular, cf. \cite[Proposition 16.28]{exel-book} the \emph{full crossed product} and the \emph{reduced crossed product} can be  defined as follows
$$
A\rtimes_{\alpha}G:= C^*(\B),\qquad A\rtimes_{\alpha,r}G:= C^*_r(\B).
$$     
In the sequel, we will identify $B_e$ with $D_e=A$, so we will write $a$ for $a\delta_e$.
\end{ex}

\subsection{Ideals in Fell bundles and graded $C^*$-algebras}
An \emph{ideal} in a Fell bundle $\B=\{B_t\}_{t\in G}$ is a collection $\J = \{J_t\}_{t\in G}$,
consisting  of closed subspaces $J_t \subseteq B_t$, such that 
  $B_s J_t \subseteq J_{st}$ and
  $J_s B_t \subseteq J_{st}$, for all $s,t\in G$, see \cite[Definition 2.1]{Exel1}.  Then it follows, see \cite{Exel1}, that $\J$ is
self-adjoint in the sense that $(J_t)^*=J_{t^{-1}}$,  so in particular
$\J$ is a Fell bundle in its own right (thus our definition agrees with \cite[Definition 21.10]{exel-book}). Moreover, the family $\B/\J :=
\{B_t/J_t\}_{t\in G}$ is  equipped with a natural Fell bundle structure  and as such is called \emph{quotient Fell bundle}, cf. \cite[Definition 21.14]{exel-book}.
In view of \cite[Proposition 2.2]{Exel1}, see also  \cite[Proposition 21.15]{exel-book}, we have the following  natural exact sequence 
\begin{equation}\label{sequence which is always exact}
  0 \longrightarrow C^*(\J) \stackrel{\iota}{\longrightarrow} C^*(\B) \stackrel{\kappa}{\longrightarrow} C^*(\B/\J) \longrightarrow 0, 
  \end{equation}
which by \cite[Lemma 4.2]{Exel1} induces the following (not necessarily exact!) sequence
\begin{equation}\label{sequence to be exact}
  0 \longrightarrow C^*_r(\J) \stackrel{\iota_r}{\longrightarrow} C^*_r(\B) \stackrel{\kappa_r}{\longrightarrow} C^*_r(\B/\J) \longrightarrow 0 
  \end{equation}
	where $\iota_r$ is injective and $\kappa_r$ surjective, but in general  $\iota_r(C^*_r(\J))\subsetneq \ker\kappa_r$.
	
	Ideals in Fell bundles and graded algebras are related to each other in the following way. 
	If $J$ is an ideal in a graded $C^*$-algebra $B=\overline{\bigoplus_{t\in G} B_t}$, then it is easy to see that $\J:=\{J\cap B_t\}_{t\in G}$ is an ideal in the Fell bundle  $\B=\{B_t\}_{t\in G}$. Moreover, by \cite[Proposition 23.1]{exel-book} we have the equivalence
\begin{equation}\label{induced ideals are graded}
J \textrm{ is generated as an ideal by }J\cap B_e\, \,\Longleftrightarrow \, \, J=\overline{\bigoplus_{t\in G} J\cap B_t}.   
\end{equation}
The ideals in $B=\overline{\bigoplus_{t\in G} B_t}$ satisfying equivalent conditions in \eqref{induced ideals are graded}, are called \emph{induced} \cite[Definition 3.10]{Exel} or \emph{graded} \cite[Definition 23.2]{exel-book}. In the present general context we prefer the second name. 

In topologically graded algebras there is another important class of ideals. Recall that $B=\overline{\bigoplus_{t\in G} B_t}$ is topologically graded if and only there are Fourier coefficient operators $F_t:B\to B_t\subseteq B$, $t\in G$. In this case an ideal $J$ in $B$ is called \emph{Fourier} if 
$$
F_t(J)\subseteq J, \qquad \textrm{ for all } t \in G,
$$
see \cite[Definition 23.8]{exel-book}. The following fundamental relationship between the general, graded and Fourier ideals  in topologically graded $C^*$-algebras was already established in \cite[Theorem 3.9]{Exel}, see also \cite[Proposition 23.4]{exel-book}.
\begin{prop}\label{remark on Exel's thm}
If $J$ is an ideal in a topologically graded $C^*$-algebra $B$,  then 
$$
\{b\in B: F_e(b^*b)\in J\}=\{b\in B: F_t(b)\in J, \,\, t\in G\}
$$
and this set  is a Fourier ideal in $B$ that contains the graded ideal generated by $J\cap B_e$. In particular, 
 if $J$ is graded, then it is Fourier.
\end{prop}

For the sake of discussion let us denote the Fourier and the graded ideal in the above proposition respectively by $J_F$ and $J_G$. Then  we have  two  inclusions $J_G\subseteq J$ and $J_G\subseteq J_F$, and in general this is all we can say. More precisely, $J_G\subsetneq J_F$ if and only if $J_F$ is a Fourier ideal which is not graded. There is  always such an ideal if $B\neq C^*_r(\B)$  (consider the kernel of the canonical epimorphism from $B=\overline{\bigoplus_{t\in G} B_t}$ onto  $C^*_r(\B)$), and  even if $B=C^*_r(\B)$ one can construct such an ideal when the underlying group $G$ is not exact, see \cite[page 199]{exel-book}. On the other hand, considering the Fell bundle arising from the $C^*$-dynamical system $(\C,id, \Z)$ for any non-trivial ideal $J$ in $\C\rtimes_{id} \Z\cong C(\T)$ we get $J\cap B_e=\{0\}$, and therefore $J\nsubseteq J_F=J_G=\{0\}$. This  indicates that the equality  $J_G= J_F$ is related with a notion of `exactness' while inclusion $J\subseteq J_F$ has to do with an `intersection property'.  We will make  these notions precise and study them in more detail in the forthcoming section.

\section{Exactness, the intersection property and topological freeness}
In this section, we exploit notions of exactness, the intersection property and topological freeness for a  Fell bundle $\B$, introduced recently in \cite{AA}. 
As shown in \cite{AA}, these properties allow one to  parametrize  ideals in  $C^*_r(\B)$  by   ideals in the core $C^*$-algebra $B_e$. 
The relevant ideals in $B_e$  are defined as follows.
\begin{defn}[Definition 3.5 in \cite{AA}] Let  $\B=\{B_t\}_{t\in G}$ be a Fell bundle.  
We say that an ideal $I$ in  $B_e$ is  \emph{$\B$-invariant} if  $B_t I B_{t^{-1}}\subseteq I$ for every $t\in G$.
 We denote the set of all $\B$-invariant ideals in $B_e$  by $\I^\B(B_e)$ and equip it with the Fell topology inherited from  $\I(B_e)$.
\end{defn}

The relationship between various types of ideals is explained in the following: 

\begin{prop}\label{proposition on induced ideals}
Let $B=\overline{\bigoplus_{t\in G}B_t}$ be a graded $C^*$-algebra. Relations  
$$
 J=\overline{\bigoplus_{t\in G}J_t}, \quad \qquad J_t =J\cap B_t=B_tI=IB_t,\,\quad t\in G,
$$
establish natural bijective  correspondences between the following objects:
\begin{itemize}
\item[(i)] graded ideals $J$ in $B$,
\item[(ii)]  ideals  $\J=\{J_t\}_{t\in G}$ in the Fell bundle $\B=\{B_t\}_{t\in G}$,
\item[(iii)] $\B$-invariant ideals $I$ in the core $C^*$-algebra $B_e$.
\end{itemize}
\end{prop}
\begin{proof}
The correspondence between objects in items (i) and (ii) was  in essence already discussed and follows easily from \eqref{induced ideals are graded}. The correspondence between objects in items (ii) and (iii) is proved in \cite[Proposition 3.6]{AA}. 	
\end{proof}

\begin{cor} For any Fell bundle $\B=\{B_t\}_{t\in G}$,  we have  a  surjective  map
\begin{equation}\label{mapping to become a homeomorphism}
\I(C^*_r(\B)) \ni  J \to J \cap B_e \in \I^\B(B_e),
 \end{equation}
which becomes a homeomorphism when restricted to graded ideals in $C^*_r(\B)$. 
\end{cor}
\begin{proof} Clearly, the mapping \eqref{mapping to become a homeomorphism} is well defined and preserves inclusions. Thus the assertion follows from Proposition \ref{proposition on induced ideals}.
\end{proof}
 
Note that Proposition \ref{proposition on induced ideals} (and injectivity of $\iota_r$ in \eqref{sequence to be exact}) implies that for any Fell bundle  $B=\{\B_g\}_{g\in G}$ and  any graded ideal $J$  in $C^*_r(\B)$, we have 
$$
J\cong C^*_r(\J)
$$
where $\J=\{J_t\},\,\, J_t=J\cap B_t, \,\,t\in G$. In particular, we get  a similar isomorphism for the quotient $C^*_r(\B)/J$ provided the  sequence  \eqref{sequence to be exact} is exact. The following definition  generalizes the notion of exactness  for group (partial) actions  introduced in \cite[Definition 1.5]{s}, \cite[Definition 3.1(ii)]{gs}. 
  \begin{defn}[Definition 3.14 in \cite{AA}]
We say that a Fell bundle $\B=\{B_g\}_{g\in G}$ is \emph{exact} if the sequence  \eqref{sequence to be exact} is exact for every  ideal $\J$ in $\B$. 
  \end{defn}
\begin{rem}
In view of  \cite[Proposition 2.2]{Exel1} a discrete group $G$ is exact if and only if any Fell bundle over $G$  is exact.
\end{rem}

\begin{cor}\label{corollary on quotients} If  $B=\{\B_g\}_{g\in G}$ is an exact Fell bundle and
 $J$ is a graded  ideal in $C^*_r(\B)$, then
$$
 C^*_r(\B)/J\cong C^*_r(\B/\J) 
$$
where $\J=\{J_t\},\,\, J_t=J\cap B_t, \,\,t\in G$. 
\end{cor}
\begin{proof}
Apply the correspondence between objects in (i) and (ii) in Proposition \ref{proposition on induced ideals} and exactness of the sequence \eqref{sequence to be exact}. 
\end{proof}
We have the following characterization of  exactness of Fell bundles in terms of the structure of Fourier ideals in $C^*_r(\B)$.
\begin{prop}\label{characterization of exactness}
A Fell bundle $\B=\{B_g\}_{g\in G}$ is exact if and only if every Fourier ideal in $C^*_r(\B)$ is graded.  
 \end{prop}
 \begin{proof}
Let $\J=\{J_t\}_{t\in G}$ be an ideal in $\B$ and put $J^{(1)}=\iota_r(C^*_r(\J))$  and $J^{(2)}=\ker(\kappa_r)$ where $\iota_r$ and $\kappa_r$ are mappings appearing in the  sequence \eqref{sequence to be exact}. It follows  from the construction of $\iota_r$ and $\kappa_r$ that
\begin{equation}\label{inclusion for exactness}
J^{(1)}=\overline{\bigoplus_{t\in G}J_t} \subseteq J^{(2)}=\{b\in B: F_t(b)\in J_t, \,\, t\in G\}.
\end{equation}
Hence $J^{(2)}$ is Fourier, and since $J_t=J_tJ_t^*J_t\subseteq J_eJ_t$, for all $t\in G$, one concludes that $J^{(1)}$ is an induced ideal. In particular, \eqref{inclusion for exactness} implies that
$
J_t=J^{(i)}\cap B_t=F_t(J^{(i)})$ for all $t \in G, \,\, i=1,2$.

Suppose now $J$ is a Fourier ideal in $C^*_r(\B)$. For each $t\in G$, let  $J_t=J\cap B_t$ so that $\J=\{J_t\}_{t\in G}$ is an ideal in $\B$.  Consider the ideals $J^{(1)}$ and $J^{(2)}$ associated to $\J=\{J_t\}_{t\in G}$ as above. By the Fourier property of  $J$, we conclude from  \eqref{inclusion for exactness} that
$$
J^{(1)} \subseteq J \subseteq J^{(2)} .
$$
Thus if \eqref{sequence to be exact} is exact then $J=J^{(1)} =J^{(2)} $ is a graded ideal. If $J$ is not graded then $J_1\neq J$ and \eqref{sequence to be exact} is not exact. 
\end{proof}
Standard arguments show that exactness passes to ideals and quotients. 
\begin{lem}\label{permanence of exactness}
Let $\J$ be an ideal in an exact Fell bundle $\B$. Then  both $\J$  and $\B/\J$  are exact.
\end{lem}
\begin{proof}
If $\I$ is an ideal in $\J$, then treating it as an ideal in $\B$ we get the exact sequence
$$
 0 \longrightarrow C^*_r(\I) \stackrel{\iota_r}{\longrightarrow} C^*_r(\B) \stackrel{\kappa_r}{\longrightarrow} C^*_r(\B/\I) \longrightarrow 0 
$$
which restricts to the exact sequence 
$$
 0 \longrightarrow C^*_r(\I) \stackrel{\iota_r}{\longrightarrow} C^*_r(\J) \stackrel{\kappa_r}{\longrightarrow} C^*_r(\J/\I) \longrightarrow 0.
$$
Hence $\J$ is exact.  Now suppose that $\widetilde{\I}$ is an ideal in $\B/\J$ and let $\I$ be the preimage   of $\widetilde{\I}$  under the quotient map. Then $\I$ is an ideal in $\B$, containing $\J$,
and we have a natural isomorphism $(\B/\J)/\widetilde{\I}\cong \B/\I$. 
In particular, we have a commutative diagram
$$
\begin{xy}
\xymatrix{
0 \ar[r] &  C^*_r(\I) \ar[r] \ar[d] & C^*_r(\B) \ar[r] \ar[d]&  C^*_r(\B/\I) \ar[r] \ar[d]_{\cong} &  0 
  \\
	0 \ar[r] &  C^*_r(\widetilde{\I}) \ar[r] & C^*_r(\B/\J) \ar[r] &  C^*_r((\B/\J)/\widetilde{\I}) \ar[r] &  0 
			} 
  \end{xy}, 
$$
where the upper row is exact and the two left most vertical maps are surjective (they are induced by the quotient map $\B\to   \B/\J$). Using this, one readily gets that the lower row is also exact. Hence $\B/\J$ is exact.
\end{proof}

As one would expect, amenability of a Fell bundle implies exactness. 
\begin{lem}\label{permanence of amenability}
Let $\J$ be an ideal in a Fell bundle $\B$. Then 
$$
 \B \textrm{ is amenable } \,\,  \Longleftrightarrow \,\,  \J \textrm{ and } \B/\J \textrm{ are amenable}.
$$
\end{lem}
\begin{proof}
Denote by $\Lambda_\B:C^*(\B)\to C^*_r(\B)$ the canonical epimorphism. By \cite[Proposition 3.1]{Exel} we have  $\ker(\Lambda_\B)=\{a \in  C^*(\B): E_\B(a^*a)=0\}$ where $E_\B: C^*(\B) \to B_e$ is the canonical conditional expectation. Note that, for any Fell bundles $\B$ and $\B'$ and any homomorphism $\Phi:C^*(\B)\to C^*(\B')$ that preserves the gradings we have $\Phi(E_\B(a^*a))=E_{\B'}(\Phi(a)^*\Phi(a))$, $a\in C^*(\B)$,  and therefore $\Phi(\ker(\Lambda_\B))\subseteq \ker(\Lambda_{\B'})$.
Thus the exact  sequence \eqref{sequence which is always exact} restricts to the sequence 
\begin{equation}\label{radical exact sequence}
 0 \longrightarrow \ker(\Lambda_\J) \stackrel{\iota}{\longrightarrow} \ker(\Lambda_\B) \stackrel{\kappa}{\longrightarrow} \ker(\Lambda_{\B/\J}) \longrightarrow 0.
\end{equation}
It is not hard to see that  \eqref{radical exact sequence} is also exact. Indeed, $\iota$ is injective and clearly we have  $\iota(\ker(\Lambda_\J))=\ker(\Lambda_\B)\cap \iota(C^*(\J))$. Hence the isomorphism $C^*(\B)/\iota(C^*(\J))\cong C^*({\B/\J})$, induced by the epimorphism  $\kappa:C^*(\B)\to  C^*({\B/\J}))$, `restricts' to the isomorphism  $\ker(\Lambda_\B)/\iota(\ker(\Lambda_\J))\cong \ker(\Lambda_{\B/\J})$. Thus  \eqref{radical exact sequence} is  exact.

Now, since $\B$ is amenable if and only if $\ker(\Lambda_\B)=\{0\}$, the assertion follows from exactness of the sequence \eqref{radical exact sequence}.   
\end{proof}
\begin{cor}
Every amenable Fell bundle is exact. 
\end{cor}
\begin{proof}
In view of Lemma \ref{permanence of amenability}, if $\B$ is amenable,  the sequences   \eqref{sequence which is always exact} and \eqref{sequence to be exact} coincide.  
\end{proof}

The notion of (residual) intersection property, in the context of crossed products was introduced in  \cite[Definition 1.9]{s}, see also \cite[Definition 3.1(iii)]{gs}. 
\begin{defn}[Definition 3.14 in \cite{AA}]\label{residual definition for Fell bundles} 
We say that a Fell bundle $\B=\{B_t\}_{t\in G}$ has the \emph{intersection property} if every non-zero ideal in $C^*_r(\B)$ has a non-zero intersection with $B_e$. We say that a Fell bundle $\B=\{B_t\}_{t\in G}$ has the \emph{residual intersection property} if $\B/\J$ has the intersection property for every ideal $\J$ in $\B$. 
\end{defn}
The following theorem is essence reformulation of \cite[Theorem 3.19]{AA}.  It is   a generalization of  \cite[Theorem 3.2]{gs}, cf. also \cite[Theorem 1.13]{s}. 
\begin{thm}\label{Sierakowski's ;) theorem}
 Let $\B=\{B_t\}_{t\in G}$ be a Fell bundle. The following statements are equivalent:
\begin{itemize}
\item[(i)] The map \eqref{mapping to become a homeomorphism} establishes a homeomorphism $\I(C^*_r(\B))\cong \I^\B(B_e)$.
\item[(ii)] All ideals in $C^*_r(\B)$ are graded.
\item[(iii)]  $\B$ is exact and has the residual intersection property.
\end{itemize}
\end{thm}
\begin{proof} Equivalence (i)$\Longleftrightarrow$(ii) is clear. Equivalence (i)$\Longleftrightarrow$(iii) follows from \cite[Theorem 3.19]{AA} (note that for every ideal $\J\in \I^\B(B_e)$ the $C^*$-algebra $C^*_r(\J)$ embeds into $C^*_r(\B)$ as a graded ideal).
 \end{proof}
\begin{cor}\label{simplicity for Fell bundles}
 Let $\B=\{B_t\}_{t\in G}$ be a  Fell bundle with the intersection property. Then  $C^*_r(\B)$ is simple if and only if there are no non-trivial  $\B$-invariant ideals in $B_e$.
 \end{cor}
\begin{proof}
The `only if' part is clear. For the `if' part note  that if there are no non-trivial  $\B$-invariant ideals in $B_e$, then $\B$ is trivially an exact Fell bundle. Thus it suffices to apply Theorem \ref{Sierakowski's ;) theorem}. 
\end{proof}

\begin{cor}\label{core for the ideal property}
 Let $\B=\{B_t\}_{t\in G}$ be an exact  Fell bundle with the residual intersection property. If $B_e$ has the ideal property then $C^*_r(\B)$ has the ideal property.
 \end{cor}
\begin{proof}
By Theorem \ref{Sierakowski's ;) theorem},   every  ideal $J\in \I( C^*_r(\B))$ is generated by  $I=B_e\cap J$. Denoting by $P(J)$ and $P(I)$ respectively the sets of projections in $J$ and $I$, we see that the ideal generated by $P(J)$ in $C^*_r(\B)$ contains 
$$
C^*_r(\B)P(I)C^*_r(\B)=C^*_r(\B)B_eP(I)B_eC^*_r(\B)=C^*_r(\B)IC^*_r(\B)=J.   
$$
Hence $C^*_r(\B)P(J)C^*_r(\B)=J$, which shows the ideal property for $C^*_r(\B)$.
\end{proof}
 We have  the following characterization of the intersection property in terms of graded $C^*$-algebras.

\begin{prop}\label{intersection and uniqueness for Fell bundles}
Let  $\B=\{B_t\}_{t\in G}$ be  a Fell bundle. The following conditions are equivalent:
\begin{itemize}
\item[(i)] $\B$ has the intersection property,
\item[(ii)] every graded $C^*$-algebra $B=\overline{\bigoplus_{t\in G}B_t}$ is automatically topologically graded, 
\item[(iii)] for every graded $B=\overline{\bigoplus_{t\in G}B_t}$ and any positive element $b=\oplus_{t\in G} b_t$  in $\bigoplus_{t\in G}B_t$ we have 
$$
\|b_e\|\leq \|b\|_{B}.
$$
\end{itemize}  
\end{prop}
\begin{proof} (i)$\Rightarrow$(ii) 
Suppose  $\B=\{B_t\}_{t\in G}$ has the intersection property and  $B=\overline{\bigoplus_{t\in G}B_t}$ is a graded  $C^*$-algebra. We have two surjective homomorphisms $\Psi:C^*(\B)\to B$ and $\Lambda_{\B}:C^*(\B)\to C^*_r(\B)$ which are identities on $\bigoplus_{t\in G}B_t$. The image of $J:=\ker\Psi$ under $\Lambda_{\B}$ is an ideal in $C^*_r(\B)$ whose intersection with $B_e$ is zero. Thus (by the intersection property) $J\subseteq \ker\Lambda_{\B}$. Therefore  $\Lambda_{\B}$ factors through to the epimorphism $\Phi$ from $B$, identified with $C^*(\B)/J$, onto $C^*_r(\B)$ which is injective on $\bigoplus_{t\in G}B_t$. Since $C^*_r(\B)$ is topologically graded it follows that  $B=\overline{\bigoplus_{t\in G}B_t}$ is  topologically graded. Indeed, if $b=\oplus_{t\in G} b_t$ is in  $\bigoplus_{t\in G}B_t\subseteq B$, then
$
\|b_e\|=\|\Phi(b_e)\|\leq \|\Phi(b)\|\leq \|b\|.
$

(ii)$\Rightarrow$(iii) is trivial. To show (iii)$\Rightarrow$(i), assume on the contrary that there is a non-zero ideal $J$ in   $C^*_r(\B)$ such that $J\cap B_e=\{0\}$. 
Then $J$ has trivial intersection with all the spaces $B_t$, $t\in G$. Hence $B:=C^*_r(\B)/J$ is  graded by the Fell bundle $\{B_t\}_{t\in G}$ and the quotient map  $q: C^*_r(\B)\to B=C^*_r(\B)/J$ is injective on  $\bigoplus_{t\in G}B_t$.
Since the conditional expectation $E:C^*_r(\B) \to B_e$ is faithful and $J\neq\{0\}$, there is a positive element $a\in J$ with $\|E(a)\|=1$. As $\bigoplus_{t\in G}B_t$ is a dense $*$-algebra in $C^*_r(\B)$, we may find a positive element $b=\oplus_{t\in G} b_t$  in $\bigoplus_{t\in G}B_t$ such that $\|a-b\|_{C^*_r(\B)}<1/3$.
Then $\|b_e-E(a)\|_{C^*_r(\B)}=\|E(b- a)\|_{C^*_r(\B)}<1/3$, which implies $\|b_e\|>2/3$, and $\|q(b)\|_{B}=\|q(b-a)\|_B<1/3$. Thus if we assume (iii), we get
$$
2/3<\|b_e\|=\|q(b_e)\|\leq \|q(b)\|_B<1/3,
$$
a contradiction.
\end{proof}
A useful condition that implies the intersection property is topological freeness of a dual system. A dynamical system dual to a saturated Fell bundle was considered in \cite[Corollary 6.5]{kwa-szym} (it is a special case of a dual semigroup associated to a product system in \cite[Definition 4.8]{kwa-szym}.   A partial dynamical system dual to a semi-saturated Fell bundle over $\Z$ was studied in \cite{kwa}. A partial dynamical system dual to an arbitrary Fell bundle over a discrete group was defined in \cite[Definition 2.3]{AA}. Let us now recall the relevant constructions and facts.

Let  $\B=\{B_t\}_{t\in G}$ be a Fell bundle over a discrete group $G$, and put 
\begin{equation}\label{domain ideals}
D_t:=B_t B_{t^{-1}} \qquad\textrm{ for all } t\in G.
\end{equation}
Then  $D_t$ is an ideal in  $B_e$ and  we may treat $B_t$ as a Morita-Rieffel $D_{t}$-$D_{t^{-1}}$-bimodule. Thus we get a partial homeomorphism $\h_t:\widehat{D}_{t^{-1}} \to \widehat{D}_t$ of $\widehat{B_e}$, where
\begin{equation}\label{formula defining partial action}
\h_t:=[B_t\dashind^{D_{t^{-1}}}_{D_t}] \qquad  \textrm{ for all } t\in G.
\end{equation}
Recall that  $\h_t$ is a lift of a  partial homeomorphism $
h_t : \Prim(D_{t^{-1}}) \to \Prim(D_t)$ of $\Prim(B_e)$, and the latter  is  a restriction of the Rieffel homeomorphism $h_t:\I(D_{t^{-1}}) \to \I(D_{t})$, which in this case is given by the formula
\begin{equation}\label{Rieffel homeomorphism}
 \I(D_{t^{-1}})\ni I \longmapsto h_t(I) =B_t I B_{t^{-1}}\in \I(D_{t}),
\end{equation}
 cf. also \cite[Remark 2.3]{kwa}.

\begin{prop}\label{proposition on dual partial systems}
Formulas \eqref{domain ideals} and \eqref{formula defining partial action} define a partial action $(\{\widehat{D}_t\}_{t\in G}, \{\h_t\}_{t\in G})$ of $G$ on $\widehat{B_e}$.
It is a lift of a partial action $(\{\Prim(D_t)\}_{t\in G}, \{h_t\}_{t\in G})$ of $G$ on $\Prim(B_e)$.
\end{prop}
\begin{proof}
The first part follows from \cite[Proposition 2.2]{AA}.
Now, since the maps $\h_t:\widehat{D}_{t^{-1}} \to \widehat{D}_t$ and $
h_t : \Prim(D_{t^{-1}}) \to \Prim(D_t)$ are intertwined by the surjection $\widehat{B_e} \ni [\pi] \to \ker\pi \in \Prim(B_e)$, one readily concludes that $(\{\Prim(D_t)\}_{t\in G}, \{h_t\}_{t\in G})$ is also a  partial action. 
\end{proof}
\begin{rem}\label{Remark on dual systems} If $\B$ is the Fell bundle of a partial action $\alpha=(\{D_t\}_{t \in G}, \{\alpha_t\}_{t \in G})$ on a $C^*$-algebra $A$ then the ideals $D_t$, $t\in G$, coincide with those given by \eqref{domain ideals}. Modifying slightly the proof of \cite[Lemma 6.7]{kwa-szym} or \cite[Proposition 2.18]{kwa-interact}, see also the proof of Lemma \ref{lemma for duals to interactions}  below, one sees that the dual partial dynamical system  $(\{\widehat{D}_g\}_{g\in G}, \{\h_g\}_{g\in G})$ is given by the formula
$$
\h_t([\pi])=[\pi\circ \alpha_{t^{-1}}], \qquad [\pi]\in \widehat{D}_{t^{-1}},\,\, t\in G.
$$
The Rieffel homeomorphism is given by $
h_t(I)=\alpha_t(I)\in \I (D_{t})$ for $I \in \I (D_{t^{-1}})$.  In particular, if $A=C_0(\Omega)$ is commutative and $(\{\Omega_g\}_{g\in G}, \{\theta_g\}_{g\in G})$ is the system that determines $\alpha$ by \eqref{partial action commutative case}, then $(\{\Omega_g\}_{g\in G}, \{\theta_g\}_{g\in G})$ can be identified with $(\{\widehat{D}_g\}_{g\in G}, \{\h_g\}_{g\in G})$, cf. \cite[Subsection 4.1]{AA}.
\end{rem}
\begin{defn}\label{definition of a partial dual system}
We call both of the systems    $(\{\widehat{D}_t\}_{t\in G}, \{\h_t\}_{t\in G})$  and $(\{\Prim(D_t)\}_{t\in G}, \{h_t\}_{t\in G})$ described above  \emph{partial dynamical systems dual to the Fell bundle} $\B$. 
\end{defn}
\begin{rem}
The authors of \cite{AA} consider only the system $(\{\widehat{D}_t\}_{t\in G}, \{\h_t\}_{t\in G})$ and call  it, \cite[Definition 2.3]{AA}, the partial action associated to $\B$. 
\end{rem}
The  following theorem can be proved by a straightforward  adaptation of the argument leading to the main result of \cite{kwa}.
In a slightly different way, it was proved in \cite{AA}.

\begin{thm}[Corollary 3.4.(ii) in \cite{AA}]\label{uniqueness theorem for fell bundles}
If the partial dynamical system  $(\{\widehat{D}_g\}_{g\in G}, \{\h_g\}_{g\in G})$ dual to a Fell bundle $\B$ is topologically free,  then $\B$ has the intersection property.
\end{thm}
\begin{rem} In view of Remark \ref{Remark on dual systems} and Proposition \ref{intersection and uniqueness for Fell bundles}, the above theorem is a generalization of similar results obtained earlier for  saturated Fell bundles \cite[Corollary 6.5(i)]{kwa-szym}, classical crossed products \cite[Theorem 1]{Arch_Spiel},  partial crossed products \cite[Theorem 2.4]{Leb} and semigroup crossed products of corner endomorphisms \cite[Theorem 6.5]{kwa-demos}, cf. Section \ref{Crossed products by a semigroup of corner systems}. Actually, Theorem \ref{uniqueness theorem for fell bundles} is much stronger than the last mentioned two results which concern  topological freeness of the system $(\{\Prim(D_g)\}_{g\in G}, \{h_g\}_{g\in G})$. Clearly, the latter  implies topological freeness of $(\{\widehat{D}_g\}_{g\in G}, \{\h_g\}_{g\in G})$, but  the converse  fails drastically already for the Cuntz algebra $\OO_n$, cf.  example after \cite[Proposition 3.16]{kwa-interact}, or Proposition \ref{aperiodicity for graph algebras} below. 
\end{rem}

In order to  get a description of all ideals in $C^*_r(\B)$, we need the following lemma.
\begin{lem}\label{lemma on invariance} If $I$ is an ideal in $B_e$, then  the following conditions are equivalent:
\begin{itemize}
\item[(i)] $I$ is $\B$-invariant,
\item[(ii)] $\widehat{I}\subseteq \widehat{B}_e$ is invariant under the partial action  $(\{\widehat{D}_t\}_{t\in G}, \{\h_t\}_{t\in G})$ dual to $\B$,
\item[(iii)] $\hull(I)$ is invariant under the partial action  $(\{\Prim(D_t)\}_{t\in G}, \{h_t\}_{t\in G})$.
\end{itemize}
\end{lem}
\begin{proof} The equivalence 
(i)$\Leftrightarrow$(ii) was proved in \cite[Proposition 3.10]{AA}.  Since $\widehat{I}$ is invariant if and only if its complement $\widehat{B}_e\setminus \widehat{I}$ is invariant and $(\{\widehat{D}_t\}_{t\in G}, \{\h_t\}_{t\in G})$ is a lift of $(\{\Prim(D_t)\}_{t\in G}, \{h_t\}_{t\in G})$, we conclude that $\hull(I)=\{\ker\pi : [\pi]\in \widehat{B}_e\setminus \widehat{I}\}$ is invariant
if and only if  $\widehat{I}$ is invariant. Hence (ii)$\Leftrightarrow$(iii).
\end{proof}

\begin{cor}\label{residual topological freeness corollary}
Suppose that the partial dynamical system  $(\{\widehat{D}_t\}_{t\in G}, \{\h_t\}_{t\in G})$ dual to $\B$ is residually topologically free.
Then $\B$ has the residual intersection property. In particular, 
\begin{itemize}
\item[(i)] if $\B$ is exact then 
$
J \to \widehat{J \cap B_e
}$
 is a lattice isomorphism from  $\I(C^*_r(\B))$ onto the set of all open invariant subsets in $\widehat{B_e}$.

\item[(ii)] $C^*_r(\B)$ is simple if and only if  $(\{\widehat{D}_t\}_{t\in G}, \{\h_t\}_{t\in G})$ is minimal (and then $\B$ is exact).
\end{itemize}

\end{cor}
\begin{proof}
Let $\J$ be an ideal in $\B$. Let $t\in G$ and treat $B_t$ as a Hilbert bimodule over $B_e$. By  Proposition \ref{proposition on induced ideals}, we see that  $I:=J_e$ is a $B_t$-invariant ideal and $B_t/J_t=B_t/B_tI$ is the quotient Hilbert bimodule. By Lemma \ref{lemma on invariance}, the closed set $Y:=\widehat{A}\setminus \widehat{I}$ is invariant under $(\{\widehat{D}_t\}_{t\in G}, \{\h_t\}_{t\in G})$. Using  Lemma \ref{coincidence of restrictions and quotients}, we conclude that the restricted partial dynamical system $(\{\widehat{D}_g\cap Y\}_{g\in G}, \{\h_g|_Y\}_{g\in G})$ can be naturally identified with the partial dynamical system dual to the quotient bundle $\B/\J$. Thus by Theorem \ref{uniqueness theorem for fell bundles}, $\B/\J$ has the intersection property. Accordingly, $\B$ has the residual intersection property. Now, part (i)  follows from  Theorem \ref{Sierakowski's ;) theorem} and Lemma \ref{lemma on invariance}. Part (ii) is a consequence of Corollary \ref{simplicity for Fell bundles}.
\end{proof}
\begin{rem} Items (i) and (ii) in Corollary \ref{residual topological freeness corollary} are in essence a content of the second part of \cite[Corollary 3.20]{AA} and of \cite[Corollary 3.12]{AA}, respectively. The first part of the assertion  in Corollary \ref{residual topological freeness corollary} was stated without a proof in \cite[Corollary 3.16]{AA}.
\end{rem}



\section{Aperiodicity for Fell bundles and  criteria of pure infiniteness}
 Muhly and Solel introduced   a notion  of aperiodicity for $C^*$-correspondences, \cite[Definition 5.1]{MS}, which was in turn inspired by the results of Kishimoto, see \cite[Lemma 1.1]{kishimoto}, and Olesen and Pedersen, see \cite[Theorems 6.6 and 10.4]{OlPe}. In the context of   partial group actions, a similar condition was exploited in \cite{gs}. We formulate it for Fell bundles as follows.
\begin{defn}
A Fell bundle $\B=\{B_g\}_{g\in G}$ is \emph{aperiodic} if for  each $t\in G\setminus\{e\}$, each $b_t\in B_t$ and every non-zero hereditary
subalgebra $D$ of $B_e$,
\begin{equation}\label{aperiodicity condition}
\inf \{\|ab_ta\| : a\in D^+,\,\, \|a\|=1\}=0.
\end{equation}
\end{defn} 
In other words, $\B$ is aperiodic if for each $t\in G\setminus\{e\}$ the Hilbert $B_e$-bimodule $B_t$ is aperiodic in the sense of  \cite[Definition 5.1]{MS}.
In particular, reinterpreting \cite[Lemma 5.2]{MS} we get the following lemma.
\begin{lem}\label{Muhly Solel lemma} Let $\B=\{B_g\}_{g\in G}$ be a Fell bundle and $B=\overline{\bigoplus_{g\in G} B_g}$ a $\B$-graded $C^*$-algebra.
The Fell bundle $\B$ is aperiodic if and only if for every element $b=\oplus_{g\in G} b_g$ in $\bigoplus_{g\in G} B_g$, with $b_e > 0$, and every $\varepsilon >0$ there exists $a$ in the hereditary subalgebra $\overline{b_eB_e b_e}$ of $B_e$, with $a\geq 0$ and $\|a\|=1$, such that 
$$
 \|ab_ea-aba\|_B <\varepsilon, \qquad \|ab_ea\| > \|b_e\|-\varepsilon.
 $$
 \end{lem}
\begin{proof}
By Lemma  5.2 of \cite{MS}, or more precisely very straightforward generalization of its proof,   $\B=\{B_g\}_{g\in G}$ is aperiodic if and only if for every element $b=\oplus_{g\in G} b_g$ in $\bigoplus_{g\in G} B_g$, with $b_e> 0$, and every $\epsilon >0$ there is  $a\in\overline{b_eB_e b_e}$, with $a\geq 0$ and $\|a\|=1$, such that 
$$
\|ab_ea\|> \|b_e\| - \epsilon\,\,\,\, \textrm{ and }\,\,\,\, \|ab_ta\|< \epsilon \,\,\, \textrm{ for }\,\, t\in G\setminus\{e\}.
$$
The assertion follows from the above if one puts  $\epsilon=\varepsilon/n$ where $n$ is the number of elements in the set $\{t\in G: b_t\neq 0\}$ (if $n=0$ the assertion is trivial).
    \end{proof}
    \begin{cor}\label{aperiodicity imply intersection property}
    If the Fell bundle $\B$ is aperiodic then it has the intersection property.
    \end{cor}
    \begin{proof} By Lemma \ref{Muhly Solel lemma} condition (iii) in Proposition \ref{intersection and uniqueness for Fell bundles} is satisfied.
    \end{proof}
		 \begin{cor}\label{aperiodicity imply compression}
    If the Fell bundle $\B$ is aperiodic then for any $b\in C^*_r(\B)^+\setminus\{0\}$ there is $a\in C^*_r(\B)^+\setminus\{0\}$ such that $a \precsim b$.
    \end{cor}
    \begin{proof} Let $b\in C^*_r(\B)^+\setminus\{0\}$. Lemma \ref{Muhly Solel lemma} implies that there exists a positive contraction $h $ in $B_e$ such that 
$$
 \|hE(b)h -hbh\|\leq 1/4, \qquad \|hE(b)h\| \geq \|E(b)\|-1/4=3/4,
$$
where $E$ is the conditional expectation from $C^*_r(\B)$ onto $B_e$. Putting $a:=(hE(b)h -1/2)_+\in B_e^+$ we see that $a\neq 0$. We conclude that $a \precsim b$, exactly as in the proof of \cite[Lemma 3.2]{rordam_sier}.
    \end{proof}
  Corollary \ref{aperiodicity imply intersection property} and Theorem \ref{uniqueness theorem for fell bundles} indicate that notions of  aperiodicity and topological freeness are closely related. The general relationship is rather mysterious, cf. Remark \ref{remark on aperiodicity and freeness} below. Nevertheless, we have the following:
		\begin{prop}\label{twierdzenie do sprawdzenia}
 Suppose that the unit fiber $B_e$ in the Fell bundle $\B=\{B_g\}_{g\in G}$ has a Hausdorff primitive ideal space. If the  partial action $(\{\Prim(D_t)\}_{t\in G}, \{h_t\}_{t\in G})$ dual to $\B$ is topologically free then $\B$ is aperiodic. 
	\end{prop}
\begin{proof}
Take any $b_t\in B_t$, where $t\in G\setminus\{e\}$,  and any hereditary
subalgebra $D$ of $B_e$. Let $U=\{x\in \Prim(B_e): x\nsupseteq D\}$ be the  open subset of $\Prim(B_e)$ corresponding to $D$.  If $D\cap D_{t^{-1}}=\{0\}$ then $DD_{t^{-1}}=\{0\}$ and since $B_t=B_tD_{t^{-1}}$ we get  $ab_ta=0$ for every $a\in D$. Hence we may assume that $D\cap D_{t^{-1}}\neq \{0\}$. Then $U\cap \Prim(D_{t^{-1}})$ is a non-empty open subset of $\Prim(B_e)$. By topological freeness there exists $x\in U\cap \Prim(D_{t^{-1}})$ such that $h_t(x)\neq x$. Since $\Prim(B_e)$ is Hausdorff we can actually find an open set $V\subseteq U\cap \Prim(D_{t^{-1}})$ such that   $h_t(V)\cap  V=\emptyset$.

Now we exploit the `$C_0(X)$-picture' of $B_e$. For each $x\in \Prim(B_e)$ and $a\in B_e$ we denote by $a(x)$ the image of  $A$ in the quotient $A/x$. It is a consequence of the Dauns-Hofmann theorem, see for instance \cite[Theorem A.34]{morita}, that the formula $(f\cdot a)=f(x)a(x)$ defines a module action of $C_0(\Prim(B_e))$ on $B_e$ via central elements in $M(B_e)$. In particular, since $D$ is hereditary, we have  $f\cdot a \in D$ for any  $a\in D$ and $f\in C_0(\Prim(B_e))$.
Using this fact, we may find an element $a\in D^+$, $\|a\|=1$, such that  $a(x)=0$ if $x\notin V$. The latter property means that  $a\in \bigcap_{x\notin V}x$. Thus we have 
$$
b_taab_t^*\in  b_t \left(\bigcap_{x\notin V}  x\right) b_t^*\subseteq \bigcap_{x\notin V,\,\,  x\in  \Prim(D_{t})}  h_t(x)\subseteq   \bigcap_{x\notin h_t(V)}  x.
$$
Since $h_t(V)$ and $V$ are disjoint,   we get $\bigcap_{x\notin V}x \cap \bigcap_{x\notin h_t(V)}  x=\{0\}$. Therefore  
$$
\|ab_ta\|^2=\|(ab_ta) (ab_ta)^*\|= \|a (b_taab_t^*) a\|=0.
$$
\end{proof}
\begin{rem}\label{remark on aperiodicity and freeness} If $\B$ is the Fell bundle associated to a partial action  $\alpha= (\{D_t\}_{t \in G}, \{\alpha_t\}_{t \in G}) $
 on a commutative $C^*$-algebra $A$, then  both aperiodicity of $\B$ and topological freeness of the dual action are equivalent to the intersection property, see \cite[Proposition A.7]{gs}.  If additionally $\alpha$ is a  global action, these notions are also known to be equivalent to pointwise proper outerness or pointwise spectral non-triviality, see respectively  \cite[Proposition A.7]{gs}  and \cite[Lemma 1.8]{pp}.  For global actions  on a separable (not necessarily commutative) $C^*$-algebra $A$, \cite[Theorem 6.6 and Lemma 7.1]{OlPe} imply that topological freeness of the dual system on $\SA$ is equivalent  to  aperiodicity of the associated bundle, and also to pointwise proper outerness. In particular, all the aforementioned notions are closely related to the Connes spectrum and the  Borcher's spectrum, cf.  \cite{OlPe}, \cite{pp}.
\end{rem}
Before we proceed we need the following definition.
\begin{defn}
We say that a Fell bundle $\B=\{B_g\}_{g\in G}$ is \emph{residually aperiodic} if $\B/\J$ is aperiodic for any ideal $\J$ in $\B$.
\end{defn} 
\begin{cor}\label{core for aperiodicity and ideal structure}
If $\B=\{B_g\}_{g\in G}$ is residually aperiodic then it has the residual intersection property, and thus if additionally $\B$ is exact then  $J \to J \cap B_e$ is a homeomorphism  from $\I(C^*_r(\B))$ onto $\I^\B(B_e)$.
\end{cor}
\begin{proof}
Apply Corollary \ref{aperiodicity imply intersection property} and Theorem \ref{Sierakowski's ;) theorem}.
\end{proof}
The following theorem is  a generalization of \cite[Theorem 3.3]{rordam_sier}, \cite[Theorem 4.2]{gs}, and  \cite[Proposition 2.46]{kwa-endo}, proved respectively for crossed products by group actions, partial actions, and single endomorphisms (see  Section \ref{Crossed products by a semigroup of corner systems} below, for more information on  latter crossed products). In order to prove it we need a lemma which is interesting in its own right.
\begin{lem}\label{compact ideals versus ideals generated by projections}
Suppose that $A$ is a purely infinite $C^*$-algebra with finitely many ideals. Then $A$ has the ideal property.
\end{lem}
\begin{proof} It suffices to prove that every ideal $J$ in $A$ is generated (as an ideal) by a projection. To this end, we first note that $J$ is singly generated. Indeed, let $\{J_k\}_{k=1}^n$ be a family of all ideals in  $J$ with the property that every $J_k$ is singly generated. Clearly,  the ideal  generated by the ideals $\{J_k\}_{k=1}^n$ is  equal to $J$. To see that $J$ is also singly generated, for every $k$, pick an  element $a_k \in A^+$, that generates $J_k$. Put $a=\sum_{k=1}^n a_k$ and denote by $I$ the ideal in $A$ generated by $a$. Clearly, $I\subseteq J$. Conversely, for any $k$ we have $a_k \leq  a$ and therefore $a_k \in I$ (because ideals are hereditary $C^*$-subalgebras). It follows that $J=I=AaA$ is generated by $a$. 
This implies that $J$ is also generated by $|a|:=(a^*a)^{1/2}\in  A^+\setminus \{0\}$. Indeed, writing $a=u|a|$ where $u$ is the partial isometry in $A^{**}$, and denoting by $\{\mu_\lambda\}$ an approximate unit in $\overline{a^*Aa}$, we get that $u \mu_\lambda$ is in $A$ and $u \mu_\lambda|a|$ converges to $a$. Thus $a\in \overline{A|a|A}$ and $J=\overline{A|a|A}$.  Now, since $A$ is purely infinite, $|a|$ is  a properly infinite element and  the proof of implication (i)$\Rightarrow$(ii) in \cite[Proposition 2.7]{Pas-Ror} produces from $|a|$  a projection $p\in J$ that generates $J$. 
\end{proof}

\begin{thm}\label{pure infiniteness for general Fell bundles} Suppose that $\B=\{B_g\}_{g\in G}$ is an exact, residually aperiodic Fell bundle. Assume that either  $B_e$ has the ideal property or that $B_e$ contains finitely many $\B$-invariant ideals. Then the following statements are equivalent:
\begin{itemize}
\item[(i)] Every non-zero positive element in $B_e$ is properly infinite in $C^*_r(\B)$.
\item[(ii)] $C^*_r(\B)$ is purely infinite.
\item[(iii)] $C^*_r(\B)$ is purely infinite   and  has the ideal property.
\item[(iv)] Every non-zero hereditary $C^*$-subalgebra in any quotient $C^*_r(\B)$ contains an infinite projection. 
\end{itemize}
If $B_e$ is of real rank zero, then each of the above conditions is equivalent to
\begin{itemize}
\item[(i')] Every non-zero projection in $B_e$ is properly infinite in $C^*_r(\B)$.
\end{itemize}

\end{thm}
\begin{proof}
Implications (iv)$\Leftrightarrow$(iii)$\Rightarrow$(ii)$\Rightarrow$(i) are general facts, see respectively \cite[Propositions 2.11]{Pas-Ror}, \cite[Proposition 4.7]{kr} and \cite[Theorem 4.16]{kr}. If $A$ is if real rank zero the equivalence (i)$\Leftrightarrow$(i') is ensured by \cite[Lemma 2.44]{kwa-endo}. Thus it suffices to show that (i) implies (iii) or (iv). Let us then assume that every   element in $B_e^+\setminus\{0\}$ is properly infinite in $C^*_r(\B)$. We note that, in view of Corollary \ref{core for aperiodicity and ideal structure},  the equivalent conditions in Theorem \ref{Sierakowski's ;) theorem} hold.

Suppose first that $B_e$ has the ideal property. We will show (iv). Let  $J$ be an ideal in $C^*_r(\B)$  and $D$ be a non-zero hereditary $C^*$-subalgebra  in the quotient $C^*_r(\B)/J$. We need to show that $D$ contains an infinite projection.
By Corollary \ref{corollary on quotients}, we have $
C^*_r(\B)/J\cong C^*_r(\B/\J)$,  where $\J=\{J_t\},\, J_t=J\cap B_t, \,\,t\in G$. 
 Fix a non-zero positive element $b$ in $D$. By Corollary \ref{aperiodicity imply compression}  there exists a non-zero positive element $a$ in $B_e/J_e$ such that $a \precsim b$ relative to $ C^*_r(\B/\J)$. Note that $a$ is properly infinite in  $C^*_r(\B/\J)$ as a non-zero homomorphic image of a properly infinite positive element in  $C^*_r(\B)$, by the assumption in (i). Since $B_e$ has the ideal property we can find  a projection $q\in B_e$ that belongs to the ideal in $B_e$  generated by the preimage of $a$ in $B_e$ but not to $J_e$. Then $q+J_e$ is a non-zero projection that belongs to the ideal in $B_e/J_e$ generated by $a$, whence $q+J_e  \precsim a \precsim b$, by \cite[Proposition 3.5(ii)]{kr}. From the comment after \cite[Proposition 2.6]{kr} we can find $z\in C^*_r(\B/\J)$ such that $q+J_e=z^*bz$. With $v:=b^{\frac{1}{2}}z$ it follows that $v^*v=q+J_e$, whence $p:=vv^*=b^{\frac{1}{2}}zz^*b^{\frac{1}{2}}$ is a projection in $B$,
 which is equivalent to $q$. By the assumption in (i), $q$ and hence also $p$ is properly infinite.

Suppose now that $B_e$ has finitely many, say $n$, $\B$-invariant ideals.  Since they are in one-to-one correspondence with ideals in  $C^*_r(\B)$ (recall Corollary \ref{core for aperiodicity and ideal structure})  Lemma \ref{compact ideals versus ideals generated by projections} implies that the conditions (ii) and (iii) are equivalent.  We will prove (ii). The proof goes by induction on $n$.

 Assume first that $n=2$ so that $C^*_r(\B)$ is simple. Take any $b\in C^*_r(\B)^+\setminus\{0\}$. By Corollary \ref{corollary on quotients} there is $a\in B_e^+\setminus\{0\}$ such that $a\precsim b$. Then $b\in \overline{C^*_r(\B)aC^*_r(\B)}=C^*_r(\B)$ and as $a$ is properly infinite we get $b\precsim a$ by \cite[Proposition 3.5]{kr}. Hence $b$ is properly infinite as it is Cuntz equivalent to $a$. Thus  $C^*_r(\B)$  purely infinite. 

Now suppose that our claim holds for any $k<n$. Let $J$ be any non-trivial ideal in  $C^*_r(\B)$. By Corollary \ref{corollary on quotients}, we have $J\cong C^*_r(\J)$ and $C^*_r(\B)/J\cong C^*_r(\B/\J)$,  where $\J=\{J_t\},\, J_t=J\cap B_t, \,\,t\in G$.  By Lemma \ref{permanence of exactness}, exactness passes to ideals and quotients, and clearly the same holds for residual aperiodicity. 
Thus both  $\J$ and $\B/\J$ satisfy the assumptions of the assertion and the corresponding unit fibers have less than  $n$ invariant ideals. Moreover, both  $\J$ and $\B/\J$ satisfy condition (i). Indeed, for $\B/\J$ it is clear, as proper infiniteness passes to quotients, and for $\J$ it follows from the fact that proper infiniteness of $a\in J_e^+\setminus \{0\}$ in $C^*_r(\B)$ imply proper infiniteness of $a$ in $J$, by \cite[Proposition 3.3]{kr}. Concluding, by induction hypotheses, both  $J$ and $C^*_r(\B)/J$ are purely infinite, and since pure infiniteness is closed under extensions \cite[Theorem 4.19]{kr} we get that  $C^*_r(\B)$ is purely infinite.
\end{proof}
\begin{rem}\label{remark on strongly pure infiniteness}
We recall, see  \cite[Propositions 2.11, 2.14]{Pas-Ror}, that in the presence of the ideal property  pure infiniteness of a $C^*$-algebra is equivalent to strong pure infiniteness, weak  pure infiniteness, and many other notions of infiniteness appearing in the literature. Thus the list of equivalent conditions in Theorem \ref{pure infiniteness for general Fell bundles} could  be considerably extended.
\end{rem}

\section{Paradoxicality, residual infiniteness and pure infiniteness}

 Now, we give and study a   noncommutative, algebraic version of  the notion of paradoxical sets, cf. Definition \ref{paradoxical definition of Sierakowski and Rordam},  phrased in terms of Fell bundles. We also introduce a notion of residually $\B$-infinite elements, which we think is a good alternative for $\B$-paradoxical elements. In particular,  the former elements seem to be more  convenient to work with in practice, cf. Remark \ref{Remark below} below.

\begin{defn}\label{paradoxical definition for B}
Let $\B=\{B_g\}_{g\in G}$ be a Fell bundle and let $a\in B_e^+\setminus\{0\}$. We say that:
\begin{itemize} 
\item[(i)]
$a$ is \emph{$\B$-paradoxical} if for every $\varepsilon >0$ there are elements  $a_i\in aB_{t_i}$, where $t_i\in G$ for $i=1,...,n+m$, such that
\begin{equation}\label{relations for paradoxicality}
a \approx_\varepsilon \sum_{i=1}^n a_i^*a_i, \quad a\approx_\varepsilon\sum_{i=n+1}^{n+m}a_i^*a_i,\quad \textrm{and} \quad \|a_i^* a_j\| < \frac{\varepsilon}{\max\{n^2, m^2\}}\,\, \textrm{ for } i\neq j.
\end{equation}
If the elements $a_i$, $i=1,...,n+m$, above can be chosen so that  
\begin{equation}\label{relations for projections}
 a = \sum_{i=1}^n a_i^*a_i=\sum_{i=n+1}^{n+m}a_i^*a_i \quad \textrm{ and }\quad  a_i^* a_j=0\,\,\textrm{ for }\,\, i\neq j,
\end{equation}
we call $a$ \emph{strictly $\B$-paradoxical}.
\item[(ii)] $a$ is \emph{$\B$-infinite} if there is $b\in B_e^+\setminus\{0\}$ such that for every $\varepsilon >0$ there are elements  $a_i\in aB_{t_i}$ where $t_i\in G$ for $i=1,...,n+m$,  such that 
$$ 
a \approx_\varepsilon \sum_{i=1}^n a_i^*a_i, \quad b\approx_\varepsilon\sum_{i=n+1}^{n+m}a_i^*a_i,\quad \textrm{and} \quad \|a_i^* a_j\| < \frac{\varepsilon}{\max\{n^2, m^2\}}\,\, \textrm{ for } i\neq j.
$$
We say  $a$ is \emph{strictly $\B$-infinite} if there is a non-zero positive element $b\in aB_ea$ and  elements  $a_i\in aB_{t_i}$ where $t_i\in G$ for $i=1,...,n$,  such that
\begin{equation}\label{relations for infiniteness}
a =\sum_{i=1}^n a_i^*a_i, \qquad  a_i^* a_j=0 \quad \textrm{ for } i\neq j,\quad \textrm{and} \quad a_i^*b=0\quad  \textrm{ for } i=1,...,n.
 \end{equation}
\item[(iii)]  $a$ is \emph{residually  $\B$-infinite} if for every ideal $\J=\{J_g\}_{g\in G}$ in $\B$ the element $a+J_e$ is either zero in $B_e/J_e$ or it is $\B/\J$-infinite. We say that $a$ is \emph{residually strictly $\B$-infinite} if for every ideal $\J=\{J_g\}_{g\in G}$ in $\B$ the element $a+J_e$ is either zero in $B_e/J_e$ or it is strictly $\B/\J$-infinite
\end{itemize}
\end{defn}
\begin{rem}
If $a$ is strictly $\B$-infinite it is $\B$-infinite; take $m=1$, $t_{n+1}=e$ and $a_{n+1}=\sqrt{b}$. Moreover, if there are elements $a_i\in aB_{t_i}$ where $t_i\in G$ for $i=1,...,n+m$,  such that 
$$
a =\sum_{i=1}^n a_i^*a_i, \qquad  \sum_{i=n+1}^{n+m}a_i^*a_i\neq 0 \quad \textrm{ and }\quad  a_i^* a_j=0 \quad \textrm{for }i\neq j,
$$ 
then putting $b=\sum_{i=n+1}^{n+m}a_ia_i^*$, we see that $a$ is  strictly $\B$-infinite. In particular, it follows that  every strictly $\B$-paradoxical element is strictly $\B$-infinite and actually  residually strictly $\B$-infinite. Also it is readily seen that every  $\B$-paradoxical element is   residually  $\B$-infinite. Whether the converse holds is an open problem.
\end{rem}
The following Proposition \ref{paradoxality imply infiniteness} provides a motivation for this definition. It also implies that both paradoxicality and residual $\B$-infiniteness can be viewed  as  generalizations of proper infiniteness.
\begin{prop}\label{paradoxality imply infiniteness}
Let $B=\overline{\bigoplus_{g\in G}B_g}$ be a $\B$-graded $C^*$-algebra and let $a\in B_e^+\setminus\{0\}$.
\begin{itemize} 
\item[(i)]  If $a$ is  $\B$-paradoxical then  $a$   is properly infinite in  $B$.
\item[(ii)] If $a$ is  $\B$-infinite then $a$  is  infinite in $B$.
\item[(iii)] If $a$ is residually $\B$-infinite then for any graded ideal $J$ in $B$ the image of $a$ in $B/J$ is either zero or infinite.
\end{itemize}
\end{prop}
\begin{proof}
(i). Let $\varepsilon >0$ and choose    elements $a_i\in aB_{t_i}$, $t_i\in G$, $i=1,...,n+m$, witnessing $\varepsilon$-paradoxicality of $a$, that is assume \eqref{relations for paradoxicality} holds. Putting
$
x=\sum_{i=1}^n a_i$  and $y=\sum_{i=n+1}^{n+m} a_i$ we immediately get that $x, y \in aB$. Using that  $ \|a_i^* a_j\| < \varepsilon/\max\{n^2, m^2\}$  for  $i\neq j$ we get 
 $$
\left\|\sum_{i,j=1 \atop i\neq j}^n a_i^*a_j\right\| < \varepsilon, \qquad \left\|\sum_{i,j=n+1 \atop i\neq j}^{n+m} a_i^*a_j\right\| < \varepsilon, \qquad \left\|\sum_{i=1}^n \sum_{j=n+1}^{n+m} a_i^*a_j\right\| < \varepsilon.
$$
The above inequalities imply respectively that $x^*x\approx_{\varepsilon} a$,  $y^*y\approx_{\varepsilon} a$ and $x^*y\approx_{\varepsilon} 0$. Hence $a$ is properly infinite in $B$ by \eqref{proper infiniteness}.

(ii). Follow the above argument, where  instead of \eqref{proper infiniteness} use \eqref{infiniteness characterisation}.

(iii). By Proposition \ref{proposition on induced ideals} graded ideals $J$ in $B$ are in one-to-one correspondence with ideals $\J=\{J_g\}_{g\in G}$ in $\B$. For any such pair the $C^*$-algebra $B/J$ is $\B/\J$-graded. Hence the assertion follows from part (ii).
\end{proof}
\begin{cor}\label{corollary for residual infinite}
Suppose that $\B=\{B_g\}_{g\in G}$ is an exact Fell bundle with the residual intersection property. Then  any residually $\B$-infinite  $a\in B_e^+\setminus\{0\}$ is properly infinite in $C^*_r(\B)$. 
\end{cor}
\begin{proof} By Theorem \ref{Sierakowski's ;) theorem} every ideal in $C^*_r(\B)$ is graded. Hence the assertion follows from \cite[Proposition 3.14]{kr} and Proposition \ref{paradoxality imply infiniteness}(iii).
\end{proof}

\begin{cor}
Let $A$  be a $C^*$-algebra and let $a\in A^+\setminus\{0\}$. The following statements are equivalent 
\begin{itemize} 
\item[(i)]  $a$ is properly infinite in  $A$.
\item[(ii)] $a$ is $\B$-paradoxical for every Fell bundle $\B=\{B_g\}_{g\in G}$ with $B_e=A$.
\item[(iii)] $a$ is residually $\B$-infinite for every Fell bundle $\B=\{B_g\}_{g\in G}$ with $B_e=A$.
\end{itemize}
\end{cor}
\begin{proof}
(i)$\Rightarrow$(ii). Let $a \in A^{+}\setminus\{0\}$ be properly infinite. By \eqref{proper infiniteness} for any $\varepsilon >0$ there are $x,y\in aAa$ with $x^*x\approx_\varepsilon a$, $y^*y\approx_\varepsilon a$, and $x^*y  \approx_\varepsilon 0$. Thus for any Fell bundle $\B=\{B_g\}_{g\in G}$ with $B_e=A$ condition \eqref{relations for paradoxicality} is satisfied with $n=m=1$, $t_1=t_2=e$ and $a_1=x$, $a_2=y$.

(ii)$\Rightarrow$(iii). It is clear, since  $\B$-paradoxicality implies  residual $\B$-infiniteness.

(iii)$\Rightarrow$(i). Apply Corollary \ref{corollary for residual infinite} to the Fell bundle $\B=\{B_g\}_{g\in G}$ with $B_e=A$ and $B_g=\{0\}$ for $g\in G\setminus\{e\}$.
\end{proof}

A set-theoretic counterpart of the notion of a $\B$-infinite element is defined as follows.

\begin{defn}\label{residual infiniteness definition}
Let $(\{\Omega_t\}_{t\in G}, \{\theta_t\}_{t\in G}) $ be a partial action on a topological space $\Omega$. An  open set $V\subseteq \Omega$ is called \emph{$G$-infinite} if there are  open sets $V_1,...,V_{n}$ and elements  $t_1,...,t_{n}\in G$, $n\geq 1$, such that 
\begin{itemize}
\item[(1)] $V=\bigcup_{i=1}^n V_i$,
\item[(2)] $V_{i}\subseteq \Omega_{t_i^{-1}}$ for all $i=1,...,n$ and  $\overline{\bigcup_{i=1}^n \theta_{t_i}(V_i)}\subsetneq V$,
\item[(3)] $\theta_{t_i}(V_{t_i})\cap \theta_{t_j}(V_{t_j})=\emptyset$ for all $i\neq j$.
\end{itemize}
Non-empty open set $V\subseteq \Omega$ is called \emph{residually $G$-infinite} if for every closed invariant subset $Y\subseteq \Omega$ the set $V\cap Y$ is either empty or $G$-infinite for the partial action  $(\{\Omega_t\cap Y\}_{t\in G}, \{\theta_t|_{\Omega_{t^{-1}}\cap Y }\}_{t\in G})$. 
\end{defn}
\begin{rem}
Clearly, for strong boundary actions of discrete groups on  compact spaces considered by Laca and Spielberg in \cite{Laca-Spiel} every open set is residually $G$-infinite (note that strong boundary actions are necessarily minimal). Actually, we will show below a much more general fact. We will prove that for any $n$-filling action, notion introduced in \cite{Joli-Robert}, on a unital $C^*$-algebra $A$ without finite-dimensional corners, every element in $A^{+}\setminus\{0\}$ is residually  $\B$-infinite, for the corresponding Fell bundle $\B$.
\end{rem}

The following two propositions shed more light on the relationship between Definitions \ref{paradoxical definition of Sierakowski and Rordam} and \ref{paradoxical definition for B}.
\begin{prop}\label{paradoxicality of partial action lemma}
Let $\B=\{B_g\}_{g\in G}$ be the Fell bundle of a partial action $\alpha$ on $C_0(\Omega)$ which is induced  by a partial action $(\{\Omega_g\}_{g\in G}, \{\theta_g\}_{g\in G}) $ of $G$ on a locally compact Hausdorff space $\Omega$, see \eqref{partial action commutative case}. Let $a \in B_e^+=C_0(\Omega)^+$ be a non-zero element and put $V:=\{x\in \Omega: a(x) >0\}$.     
\begin{itemize}
\item[(i)]  $V$  is $G$-paradoxical if and only if $a$ is strictly $\B$-paradoxical.
\item[(ii)] $V$ is $G$-infinite if and only if $a$ is strictly $\B$-infinite.
\end{itemize}

\end{prop}
\begin{proof}
\emph{`Only if' parts.}  Let $V_1,...,V_{n}$ be open sets and $t_1,...,t_{n}\in G$ be such that $V=\bigcup_{i=1}^n V_i$,  
$V_{i}\subseteq \Omega_{t_i^{-1}}$ and $\theta_{t_i}(V_i)\subseteq V$ for all $i=1,...,n$, and $\theta_{t_i}(V_{t_i})\cap \theta_{t_j}(V_{t_j})=\emptyset$ for all $i\neq j$.  Let $\{h_i\}_{i=1}^{n}$ be a partition of unity for $V$ relative to the open cover $\{V_i\}_{i=1}^{n}$. Let $i$ be fixed for a while.  Since $ah_i\in C_0(\Omega)$ is supported on $V_{i}\subseteq \Omega_{t_i^{-1}}$ we can treat $(ah_i)\circ \theta_{t_i^{-1}}$ as a continuous function on $\Omega$ supported on $\theta_{t_i}(V_i)\subseteq V$. By definition of $V$  we actually get  $(ah_i)\circ \theta_{t_i^{-1}} \in aC_0(\Omega)$. Hence
$$
a_i:=\left((ah_i)\circ \theta_{t_i^{-1}} \right)^{1/2}\,\delta_{t_i}
$$
is an element of  $a_i\in a B_{t_i}$, and we claim that 
\begin{equation}\label{strict relations short}
a=\sum_{i=1}^n a_i^*a_i \quad \textrm{ and }\quad  a_i^* a_j=0\,\,\textrm{ for }\,\, i\neq j.
\end{equation}
Indeed, for any $i,j$ we have 
$
a_i^* a_j=\left( (ah_i)\circ \theta_{t_i^{-1}} \cdot(ah_j)\circ \theta_{t_j^{-1}}  \right)^{1/2}\circ \theta_{t_i} \, \delta_{t_i^{-1}t_j}.
$
For $i\neq j$ the functions $h_i\circ \theta_{t_i}^{-1}$ and $h_j\circ \theta_{t_j}^{-1}$ are supported on disjoint sets $\theta_{t_i}(V_{t_i})$ and $\theta_{t_j}(V_{t_j})$. Hence
$
a_i^* a_j=0.
$
On the other hand,  $a_i^* a_i= ah_i$ and consequently
$
\sum_{i=1}^n a_i^*a_i= \sum_{i=1}^n  ah_i=a$.

Now, if $\overline{\bigcup_{i=1}^n \theta_{t_i}(V_i)}\subsetneq V$, that is if $V$ is $G$-infinite, then taking  $b\in B_e$ to be any non-zero positive function supported on the non-empty open set $V\setminus \overline{\bigcup_{i=1}^n \theta_{t_i}(V_i)}$ we infer that $a$ is strictly $\B$-infinite.
\\
If  $V_{n+1},...,V_{n+m}$ are open sets and  $t_{n+1},...,t_{n+m}\in G$ are such that $V=\bigcup_{i=n+1}^{n+m} V_i$,  $V_{i}\subseteq \Omega_{t_i^{-1}}$ and $\theta_{t_i}(V_i)\subseteq V$ for all $i=n+1,...,n+m$, and $\theta_{t_i}(V_{t_i})\cap \theta_{t_j}(V_{t_j})=\emptyset$ for all $i\neq j$, $i,j=1,...,n+m$,  then constructing elements $a_i\in a B_{t_i}$ for $i=n+1,...,n+m$ exactly as we did for $i=1,...,n$ we get relations  \eqref{relations for projections}. Hence if $V$ is $G$-paradoxical, then $a$ is strictly $\B$-paradoxical.

\emph{`If' parts.}
Let $a_i\in aB_{t_i}$,   $t_i\in G$, $i=1,...,n$, be such that \eqref{strict relations short} holds. For any  $i=1,...,n$ we have   $a_i=b_i\delta_{t_i}$ where $b_i\in aD_{t_i}=C_0(\Omega_{t_i}\cap V)$. We put
$$
V_i:=\{x\in \Omega: a_i^*a_i(x) >0\}=\{x\in \Omega: |b_i|^2(\theta_{t_i}(x)) \neq 0\}=\theta_{t_i}^{-1}(\{x\in \Omega: b_i(x)\neq 0\}).
$$
Thus   $V_i$ is an open subset of $\Omega_{t_i^{-1}}$ and $\theta_{t_i}(V_i) \subseteq V$. Moreover,
$$
V=\{ x\in \Omega: a^*a(x) >0\}=\{ x\in \Omega: \sum_{i=1}^n a_i^*a_i(x) >0\}=\bigcup_{i=1}^n V_i,
$$
and 
$$
\Big(a_i^*a_j= 0\Big) \,\, \Longleftrightarrow \,\, \Big(b_ib_j=0\Big)\,\, \Longrightarrow\,\, \Big(\theta_{t_i}(V_{t_i})\cap \theta_{t_j}(V_{t_j})=\emptyset\Big).
$$
Using this one readily sees that if $a$ is strictly $\B$-paradoxical then $V$ is $G$-paradoxical. Moreover, if there is a non-zero positive $b\in aB_ea$ such that $a_i^*b=0$ for  $i=1,...,n$, then  $W:=\{x\in \Omega: b(x) >0\}$ is a non-empty open subset of $V$ and 
$$
\Big(a_i^*b= 0 \Big)\,\, \Longleftrightarrow \,\, \Big(b_ib=0\Big)\,\, \Longrightarrow\,\, \Big(\theta_{t_i}(V_{t_i})\cap W=\emptyset\Big).
$$
Hence $\overline{\bigcup_{i=1}^n \theta_{t_i}(V_i)}\subsetneq V$ and $V$ is $G$-infinite.
\end{proof}
It is not clear whether a version of the following proposition for residually infinite elements hold.
\begin{prop} Retain the assumptions of Proposition \ref{paradoxicality of partial action lemma}.
 If $\Omega$ is totally disconnected and every open compact subset of $\Omega$ is $G$-paradoxical, then every element in $B_e^{+}\setminus\{0\}$ is $\B$-paradoxical.
\end{prop}
\begin{proof}
Let $a\in C_0(\Omega)^+\setminus\{0\}=B_e^+\setminus\{0\}$. For any $\varepsilon >0$ there exists an open compact set $V$ such that
$$
\{x\in \Omega: a(x) \geq \varepsilon \}\subseteq V  \subseteq \{x\in \Omega: a(x) >0  \}.
$$ 
Let $a_V\in C_0(\Omega)$ denote the function such that $a_V(x)=a(x)$ for $x\in V$ and $a_V(x)=0$ outside $V$.  Using paradoxicality of $V$, by Proposition  \ref{paradoxicality of partial action lemma} there are elements $a_i\in a_V B_{t_i} \subseteq a B_{t_i} $,   $t_i\in G$, $i=1,...,n+m$, such that 
$
 a_V = \sum_{i=1}^n a_i^*a_i=\sum_{i=n+1}^{n+m}a_i^*a_i$ and $a_i^* a_j=0$ for  $i\neq j$.

Since $a_V\approx_\varepsilon a$ we see that  relations \eqref{relations for paradoxicality} hold.
\end{proof}
We now turn to the promised relationship between residual $\B$-infiniteness and $n$-filling actions introduced by Jolissaint and Robertson \cite{Joli-Robert}.
\begin{defn}[Definition 0.1 in \cite{Joli-Robert}] A global action $\alpha= (A, \{\alpha_t\}_{t \in G}) $ on a unital $C^*$-algebra $A$ is called $n$-filling,  for $n\geq 2$, if, for all elements $b_1,...,b_n\in A^{+}$ of norm one, and for all $\varepsilon >0$, there exist $g_1,...,g_n\in G$ such that $\sum_{i=1}^n\alpha_{g_i}(b_i) \geq 1-\varepsilon$.  
\end{defn}
\begin{rem} A $2$-filling action on a commutative $C^*$-algebra is equivalent to what is
called a \emph{strong boundary action} in \cite{Laca-Spiel}, see \cite{Joli-Robert}.  
\end{rem}
\begin{lem} Suppose that $\alpha$ is an $n$-filling action on a unital $C^*$-algebra $A$ such that for every nonzero projection $e\in A$ the algebra $eAe$ is infinite dimensional. Let $\B$ be the  Fell bundle corresponding to $\alpha$. Then the  action $\alpha$ is minimal and any element $a\in A^+\setminus\{0\}$ is strictly (residually) $\B$-infinite.
\end{lem}
\begin{proof}
That $\alpha$ does not admit non-trivial invariant ideals  is clear. Let $a\in A^+\setminus\{0\}$ and  $0<\varepsilon <1$ be smaller than $1$. We may suppose that $\|a\|=1$. Arguing in a similar way as in the proof of \cite[Lemma 1.4]{Joli-Robert}, we see that there are  normalized positive elements $b_1,...,b_{n+1}\in aA$ such that $b_i b_j=0$ for $i\neq j$. By $n$-filling property, there are elements $g_1,...,g_n\in G$ such that $h:=\sum_{i=1}^n\alpha_{g_i}(b_i) \geq 1-\varepsilon$. Hence $h$ is positive and invertible. Put
$$
a_i:=\sqrt{b_i }\alpha_{g_i^{-1}}(\sqrt{h^{-1}})\alpha_{g_i^{-1}}(\sqrt{a})u_{g_i}, \quad i=1,...,n,\quad \textrm{ and }\quad  b:=b_n+1.
$$
Then clearly $a_i^*a_j=0$ for $i\neq j$ and $a_i^*b=0$ for all $i,j=1,...,n$. Moreover,
$$
\sum_{i=1}^n a_i^*a_i=\sum_{i=1}^n \sqrt{a}  \sqrt{h^{-1}}\alpha_{g_i}(b_i )\sqrt{h^{-1}}\sqrt{a}=\sqrt{a}  \sqrt{h^{-1}}h\sqrt{h^{-1}}\sqrt{a}=a.
$$
 Hence $a\in A^+\setminus\{0\}$ is strictly  $\B$-infinite, and as $\alpha$ is minimal, actually strictly residually $\B$-infinite.
\end{proof}

We state the main result of this section using the notion of residual $\B$-infiniteness. As noted above this is  a weaker condition than  dynamical conditions implying pure infiniteness that appear in \cite{Laca-Spiel}, \cite{Joli-Robert},  \cite{rordam_sier} or  \cite{gs}. 
\begin{thm}\label{pure infiniteness for paradoxical Fell bundles} Suppose that $\B=\{B_g\}_{g\in G}$ is an exact, residually aperiodic Fell bundle and one of the following conditions holds 
\begin{itemize}
\item[(i)] $B_e$ contains finitely many $\B$-invariant ideals and every element in $B_e^+\setminus\{0\}$ is  Cuntz equivalent to a residually $\B$-infinite element,
\item[(ii)] $B_e$ has the ideal property and every element in $B_e^+\setminus\{0\}$ is  Cuntz equivalent to a residually $\B$-infinite element,
\item[(iii)] $B_e$ is of real rank zero  and every non-zero projection in $B_e$ is Cuntz equivalent to a residually $\B$-infinite element.
\end{itemize}
Then  $C^*_r(\B)$ has an ideal property and is purely infinite. 
\end{thm}
\begin{proof}
Note that in each of the cases (i)-(iii) Theorem \ref{pure infiniteness for general Fell bundles} applies. By Corollary \ref{corollary for residual infinite} every residually $\B$-infinite element in $B_e^+\setminus\{0\}$ is properly infinite in $C^*_r(\B)$. Since an element equivalent to a properly infinite one is properly infinite, the assertion holds. 
 \end{proof}

As a consequence of Theorem \ref{pure infiniteness for paradoxical Fell bundles} we get the following strengthening and unification of the following results   \cite[Theorem 5]{Laca-Spiel}, \cite[Theorem 1.2]{Joli-Robert}(in the commutative case), \cite[Corollary 4.4]{rordam_sier} and \cite[Theorem 4.4]{gs}.

\begin{cor} Suppose that $\alpha$  is an exact  partial action on $C_0(\Omega)$  induced  by residually topologically free partial action $(\{\Omega_g\}_{g\in G}, \{\theta_g\}_{g\in G}) $ of $G$ on a locally compact space $\Omega$. Assume that one of the following conditions holds 
\begin{itemize}
\item[(i)] $\Omega$ contains finitely many $\theta$-invariant closed sets and  and every non-empty  open set is residually $G$-infinite,
\item[(ii)] $\Omega$ is totally disconnected space and every non-empty compact and open set is residually $G$-infinite.
\end{itemize} 
 Then  $A\rtimes_{\alpha,r}G$ has the ideal property and is purely infinite.
\end{cor}
\begin{proof} Apply Theorem \ref{pure infiniteness for paradoxical Fell bundles} and Proposition \ref{paradoxicality of partial action lemma}(ii).
\end{proof}

In connection with the Theorem \ref{pure infiniteness for paradoxical Fell bundles}, it is useful to make the following  observation.
\begin{lem}\label{sums and equivalence for infiniteness}
A sum of orthogonal  (residually) strictly  $\B$-infinite elements is  (residually) strictly  $\B$-infinite. Any projection in $B_e$ which is Murray-von Neumann equivalent to a (residually) strictly  $\B$-infinite projection in $B_e$ is  (residually) strictly  $\B$-infinite. 
\end{lem}
\begin{proof} Suppose that $a^{(1)}, a^{(2)}$ are strictly  $\B$-infinite elements and $a^{(1)}a^{(2)}=0$. For each $j=1,2$, let  $b^{(j)}\in a^{(j)}B_ea^{(j)}\setminus\{0\}$  and   $a_i^{(j)}\in a^{(j)}B_{t_i}$, $t_i\in G$, $i=1,...,n_{j}$, be elements satisfying the counterpart of \eqref{relations for infiniteness}. Putting  $a=\sum_{j=1}^{2} a^{(j)}$,   $b=\sum_{j=1}^{2} b^{(j)}$,   $a_i:=a_i^{(1)}$ for $i=1,...,n_{1}$,   and $a_{n_{1}+i}:=a_i^{(2)}$ for $i=1,...,n_{2}$ we see that $b\in aB_ea$ and   $a_i\in aB_{t_i}$ for $i=1,..., n:=n_{1}+ n_{2}$ satisfy \eqref{relations for infiniteness}. Hence $a$ is strictly  $\B$-infinite. Since for any quotient map $q$, $q(a)=0$ if and only if $q(a^{(1)})=q(a^{(2)})=0$, we see that if $a^{(1)}, a^{(2)}$ are  residually strictly  $\B$-infinite  then $a$ is residually strictly  $\B$-infinite.

Suppose that $p=w^*w$ and $p'=ww^*$ are projections in $B_e$, $w\in B_e$, and let  $b\in (pB_ep)\setminus\{0\}$  and   $a_i\in pB_{t_i}$, $t_i\in G$, $i=1,...,n$, satisfy  \eqref{relations for infiniteness}, with  $p$ in place of $a$.  Putting $b':=wbw^*$ and $a_i':=wa_iw^*$, $i=1,...,n$ we get  $b'\in (p'B_ep')\setminus\{0\}$  and   $a_i'\in p'B_{t_i}$ such that $p' =\sum_{i=1}^n (a_i')^{*}a_i'$,  $(a_i')^{*} a_j'=0$   for $i\neq j$,  and  $(a_i')^{*}b'=0$  for  $i=1,...,n$. Hence $p'$ is strictly  $\B$-infinite. Since for any quotient map $q$, $q(p)=0$   if and only if $q(p')=0$, it follows that   $p$ is residually strictly  $\B$-infinite  if and only if  $p'$ is residually strictly  $\B$-infinite.
\end{proof}

\section{Primitive ideal space of the reduced cross-sectional algebra}
In this section, we describe the space of prime ideals in $C^*_r(\B)$, which in the separable case will lead us to a description of the primitive spectrum of $C^*_r(\B)$, via a quasi-orbit space for the dual partial action introduced in Definition  \ref{definition of a partial dual system}.

Throughout we fix a Fell bundle $\B=\{B_t\}_{t\in G}$. The following notions and results are generalizations of classical facts for crossed products \cite{green}, see also \cite{EchLac}. 

\begin{defn}
We  say that  a $\B$-invariant ideal $I\in \I^\B(B_e)$  is \emph{$\B$-prime}  if  for any pair of $\B$-invariant ideals $J_1,J_2\in\I^\B(B_e)$ with $J_1\cap J_2 \subseteq I$ we have either $J_1\subseteq I$ or $J_2\subseteq I$. 
We equip the  set of $\B$-prime ideals with Fell topology and denote it by  $\Prime^\B(B_e)$. 
\end{defn}

We get the following generalization of \cite[Proposition 2.3]{EchLac}.
\begin{prop}\label{starting propostion}
Let $\B$ be an exact Fell bundle. Suppose  that $\B$ has the residual intersection property, which holds for instance if one of the following conditions is satisfied:
\begin{itemize}
\item[(i)] the system $(\{\widehat{D}_g\}_{g\in G}, \{\h_g\}_{g\in G})$  dual to $\B$ is residually topologically free,
\item[(ii)] $\B$ is residually aperiodic.
\end{itemize}
Then the map $\I(C^*_r(\B)) \ni  J \to J \cap B_e \in \I^\B(B_e)$ restricts to a homeomorphism 
$$
\Prime(C^*_r(\B))\cong \Prime^\B(B_e).
$$
\end{prop}
\begin{proof}
By Theorem \ref{Sierakowski's ;) theorem}, see also Corollaries \ref{residual topological freeness corollary} and \ref{core for aperiodicity and ideal structure}, the map $\I(C^*_r(\B)) \ni  J \to J \cap B_e \in \I^\B(B_e)$ is a homeomorphism. It restricts to a homeomorphism from $\Prime(C^*_r(\B))$ onto $\Prime^\B(B_e)$.
\end{proof}

  For any ideal $J\in \I(B_e)$ we have a natural continuous embedding of $\I(J)$ in $\I(B_e)$ given by
$$
\I(J) \ni I \longmapsto I_J:=\{a\in B_e: a\cdot J\subseteq I\} \in \I(B_e).
$$
This map has a  left  inverse given by $\I(B_e) \ni K \longmapsto I:=K\cap J \in \I(J)$. Note that we use the mapping $\Prim(D_{t})\ni   P \longmapsto P_{D_{t}}\in \Prim(B_e)$ to identify the space $\Prim(D_{t})$ with an open subset of  $\Prim(B_e)$, $t\in G$. Thus formally,  the dual partial action on $\Prim(B_e)$, given by restriction of the Rieffel homeomorphism \eqref{Rieffel homeomorphism},  should be defined by $h_t(P_{D_{t^{-1}}}):=h_t(P)_{D_{t}}$ where $P\in \Prim(D_{t^{-1}})$, $t\in G$.

\begin{defn}\label{B-primitive}
For any $P \in \Prim(B_e)$ we put 
$$
P_\B:=\bigcap_{t\in G} h_t(P\cap D_{t^{-1}})_{D_t}=\bigcap_{t\in G }\{a\in B_e: a D_{t}\subseteq B_t P B_{t^{-1}}\}
$$
and call it a \emph{$\B$-primitive} ideal in $B_e$. We denote the space of $\B$-primitive ideals by
$$
\Prim^\B(B_e):=\{P_\B: P \in \Prim(B_e)\}
$$ 
and endow it with the Fell-topology inherited from $\I(B_e)$.
\end{defn}
\begin{lem}\label{G-kernel lemma}
For any $P\in \Prim(B_e)$,  $P_\B$ is the largest $\B$-invariant ideal contained in $P$. 
\end{lem}
\begin{proof}
Let $ GP:=\bigcup_{t\in G, P\in \Prim(D_{t^{-1}})} \{h_t(P)_{D_t}\}$ be the orbit of $P\in \Prim(B_e)$ under the partial action $(\{\Prim(D_t)\}_{t\in G}$, $ \{h_t\}_{t\in G})$. Since $h_t(D_{t^{-1}})_{D_t}=A$ we see that 
\begin{equation}\label{orbit description of B-primitive guys}
P_\B=\bigcap_{t\in G} h_t(P\cap D_{t^{-1}})_{D_t}=\bigcap_{t\in G  \atop P \nsupseteq D_{t^{-1}}} h_t(P\cap D_{t^{-1}})_{D_t} =\bigcap_{Q\in GP }Q.
\end{equation}
Thus 
$$
B_t P_\B B_{t^{-1}}=h_t(P_\B \cap D_{t^{-1}}) \subseteq \bigcap_{Q\in GP }h_t(Q \cap D_{t^{-1}})\subseteq\bigcap_{Q\in GP }Q=P_\B,
$$ 
where in the last inclusion we used the fact that $h_{ts}$ extends $h_{t}\circ h_{s}$ (as elements of the partial action on $\Prim(B_e)$). Hence $P_\B$ is $\B$-invariant and clearly $P_\B\subseteq P$. Now suppose that  $I$ is $\B$-invariant ideal contained in $P$. Then
$$
I D_t=D_t ID_t= B_tB_{t^{-1}}IB_tB_{t^{-1}} \subseteq B_t IB_{t^{-1}}\subseteq B_t P B_{t^{-1}}, \qquad t\in G.
$$
Thus $I\subseteq P_\B$.
\end{proof}

\begin{lem}\label{open projection map}
The mapping $\Prim(B_e) \ni P\longmapsto P_\B \in \Prim^\B(B_e)$ is continuous and open.
\end{lem}
\begin{proof} 
Let $J$ be an ideal in $\B_e$. Since intersection of $\B$-invariant ideals is a $\B$-invariant ideal, there exists   the smallest $\B$-invariant ideal in $B_e$ containing $J$. We denote it by $J^\B$. Using Lemma \ref{G-kernel lemma} twice  we get 
\begin{equation}\label{equivalent inclusions}
 P_\B \nsupseteq J \,\, \Longleftrightarrow \,\,   P_\B \nsupseteq J^\B \,\, \Longleftrightarrow \,\, P \nsupseteq  J^\B,
 \end{equation}
 for any $P\in \Prim(B_e)$. Thus  any open set $U\subseteq \Prim^\B(B_e)$ is of the form 
\begin{equation}\label{form of U}
U=\{P_\B: P_\B \nsupseteq J, \, P\in \Prim(B_e)\}=\{P_\B: P \nsupseteq J^\B, \, P\in \Prim(B_e)\},
\end{equation}
where $J\in \I(B_e)$. It follows that  the map  $P\longmapsto P_\B$ is continuous. To verify openness, choose an  open  set $V=\{P\in \Prim(B_e): P \nsupseteq J\}$, $J\in \I(B_e)$, and  let $U=\{P_\B: P_\B \nsupseteq J, \, P\in \Prim(B_e)\}$. 
Clearly, if $P\in V$ then $P_\B\in U$.  On the other hand, if $P_\B \in U$ then $h_t(P\cap D_{t^{-1}})_{D_t} \nsupseteq J$ for a certain $t\in G$ such that  $P\nsupseteq  D_{t^{-1}}$.  Since $h_t(P\cap D_{t^{-1}})_{D_t}\in V$ and $P_\B=(h_t(P\cap D_{t^{-1}})_{D_t})_\B$, cf. \eqref{orbit description of B-primitive guys},  it follows that the image of $V$ is equal to $U$.

\end{proof}

We can use the above lemma to identify the  space $\Prim^\B(B_e)$ with the quasi-orbit space   for the  partial action $(\{\Prim(D_g)\}_{g\in G}$, $ \{h_g\}_{g\in G})$   on $\Prim(B_e)$ dual to  $\B$.

\begin{lem}\label{lemma on orbit spaces}
Let $\OO(\Prim(B_e))$ be the quasi-orbit space associated to the partial action $(\{\Prim(D_t)\}_{t\in G}$, $ \{h_t\}_{t\in G})$ of $G$  on $\Prim(B_e)$, cf. Proposition \ref{proposition on dual partial systems}. Then the map
\begin{equation}\label{quasi-orbit homeomorphism}
\OO(\Prim (B_e))\ni \OO(P) \longmapsto P_\B \in \Prim^\B(B_e), \qquad P\in \Prim(B_e),
\end{equation}
is a homeomorphism. In particular,  the closure of a point in $\OO(\Prim B_e)$ is given by
$$
\overline{\{\OO(P)\}}=\{\OO(Q): Q \supseteq P_\B\}.
$$
\end{lem}
\begin{proof}
Let $P, Q\in \Prim(B_e)$. By \eqref{orbit description of B-primitive guys}, we get $\overline{GP}=\{K\in \Prim(B_e): K \supseteq P_\B\}$.  Therefore
$$
\overline{GP}=\overline{GQ} \,\,\, \Longleftrightarrow\,\,\, P_\B=Q_\B.
$$
Thus \eqref{quasi-orbit homeomorphism} is a well-defined  bijection, and we see that the quasi-orbit space $\OO(\Prim (B_e))$ may be identified with the quotient space $\Prim(B_e)/\sim$ arising from the open continuous map $\Prim(B_e) \ni P\longmapsto P_\B \in \Prim^\B(B_e)$, see Lemma \ref{open projection map}. Hence \eqref{quasi-orbit homeomorphism} is a homeomorphism. For the last part of the assertion, note that the closure  of $\{P_\B\}$ in $\Prim^\B(B_e)$ is $\{Q_\B: Q_\B \supseteq P_\B\}$. 
\end{proof}
\begin{lem}\label{lemma on orbit spaces2}
Let $\B=\{B_t\}_{t\in G}$ be a Fell bundle with $B_e$ separable. Then  
$$
\Prime^\B(B_e)=\Prim^\B(B_e)\cong \OO(\Prim(B_e)),
$$
where $\OO(\Prim(B_e))$ is the quasi-orbit space for  $(\{\Prim(D_t)\}_{t\in G}$, $ \{h_t\}_{t\in G})$.
\end{lem}

\begin{proof} Since $B_e$ is separable we have $\Prim(B_e)=\Prime(B_e)$.
Let $P\in \Prim(B_e)$. If $J_1, J_2\in \I^\B(B_e)$ are such that $J_1\cap J_2\subseteq P_\B$, then  $J_1\cap J_2 \subseteq P$, and hence $J_i\subseteq P$ for some $i\in \{1,2\}$. But then $J_i\subseteq P_\B$ by Lemma \ref{G-kernel lemma}. Hence $P_\B$ is $\B$-prime.

Conversely, let $I\in \Prime^B(B_e)$. We need to show that there is a $P\in \Prim(B_e)$ with $I=P_\B$. For this it suffices to show that $\hull(I)=\{Q\in \Prim(B_e): Q \supseteq I\}$ coincides with the closure $\overline{GP}=\{Q\in \Prim(B_e): Q \supseteq P_\B\}$ of the orbit of some $P\in \Prim(B_e)$. 
To this end, note that  $\hull(I)$ is closed and invariant, cf. Lemma \ref{lemma on invariance}. 
Thus  $P\in \hull(I)$ implies $\overline{GP}\subseteq \hull(I)$.
 Accordingly, the equality $\overline{GP}= \hull(I)$ is satisfied if and only if the set 
$$
C:=\{\OO(Q): Q\in \hull(I)\}
$$ is equal to 
the closure $\{\OO(Q): Q\supseteq P_\B\}$ of the point $\OO(P)$ in $\OO(\Prim(B_e))$. 
It is known that $\Prim(B_e)$ is  a totally Baire space and that an image of a totally Baire space under an open map is a  totally Baire space, cf. discussion preceding \cite[Lemma on page 222]{green}.   Thus  $\OO(\Prim(B_e))$ is a totally Baire space by Lemma \ref{open projection map}. Also it is second countable because $B_e$ is separable. Thus,  by virtue of  \cite[Lemma on page 222]{green}, to conclude that $\overline{GP}= \hull(I)$ it suffices to show that $C$ is  not a union of two proper closed subsets. So assume that $C_1$, $C_2$ are closed subsets of $\OO(\Prim(B_e))$ with $C=C_1\cup C_2$. Let $F_1$, $F_2$ denote their inverse images in $\Prim(B_e)$ and let $J_i:=\bigcap_{Q\in F_i}Q$ for $i=1,2$. Then  $\hull(I)=F_1\cup F_2$ which implies $J_1\cap J_2=I$. Since $I$ is $\B$-prime we have $J_i\subseteq I$ for some $i\in\{1,2\}$. This implies that $F_i=\hull(J_i)=\{Q\in \Prim(B_e): Q \supseteq J_i\}$ contains  $\hull(I)$, and this completes the proof. 
\end{proof}
We recall \cite[Definition 27.5]{exel-book} that a Fell bundle $\B=\{B_t\}_{t\in G}$ is \emph{separable}  if $G$ is countable and each space $B_t$, $t\in G$, is separable. Of course, $\B$ is separable if and only if $C^*_r(\B)$ (or any other  $\B$-graded $C^*$-algebra) is separable.
The following result can be viewed as a generalization of \cite[Corollaries 2.6 and 2.7]{EchLac} and  \cite[Corollary 3.8]{gs} proved respectively for global actions on $C^*$-algebras and for  partial actions on topological spaces.
\begin{thm}\label{Primitive ideal space description}
Retain the assumptions of Proposition \ref{starting propostion} and additionally suppose that $\B$ is separable. Then we have natural homeomorphisms
$$
\Prim(C^*_r(\B))\cong \Prime^\B(B_e)=\Prim^\B(B_e)\cong \OO(\Prim B_e).
$$ 
\end{thm}
\begin{proof} Since $\B$ is separable,  $C^*_r(\B)$ is separable  and $\Prim(C^*_r(\B))=\Prime(C^*_r(\B))$. Now it suffices to combine Proposition \ref{starting propostion} and Lemmas \ref{lemma on orbit spaces} and \ref{lemma on orbit spaces2}.
\end{proof}

It is natural to treat  the space $\Prim^\B(B_e)$ (or equivalently $\OO(\Prim B_e)$, see Lemma \ref{lemma on orbit spaces}) as a `primitive spectrum' of the Fell bundle $\B$. For instance, by Proposition \ref{proposition on induced ideals}, ideals  in $\B$ correspond to elements in $\I^\B(B_e)$  and we have the following simple fact.
\begin{lem}\label{intersection of prime ideals lemma}
For every $I\in \I^\B(B_e)$ we have $\displaystyle{I=\bigcap_{P\in \Prim^\B(B_e) \atop I\subseteq P} P}$.
\end{lem}
\begin{proof}
Since $I=\bigcap_{P\in \Prim(B_e)} P$, it suffices to note, using Lemma \ref{G-kernel lemma}, that  $I\subseteq P$  implies $I\subseteq P_\B$ for every $I\in \I^\B(B_e)$ and $P\in \Prim(B_e)$.
\end{proof}

\section{Graph algebras}\label{Graph algebras}
Throughout this section, we fix a directed graph $E = (E^0,E^1, r, s)$. Hence   $E^0$ and $E^1$ are countable sets  and $r,s:E^{1}\to E^0$ are range and source maps. For graphs and their $C^*$-algebras we use the notation and conventions of \cite{Raeburn}. Thus $E^\infty$ is the set of infinite paths and  $E^n$, $n>0$, stands for the set of finite paths $\mu=\mu_1...\mu_n$, satisfying $s(\mu_i)=r(\mu_{i+1})$ for all $i$, where   $|\mu|=n$ is the length of $\mu$.  If additionally $r(\mu_1)=s(\mu_n)$ we say that $\mu$ is a \emph{cycle}. The \emph{cycle $\mu$  have an entrance} if there is an edge $e$ such that $r(e)=r(\mu_k)$ and $e \neq \mu_k$, for some  $k=1,...,n$. We say that $E$ \emph{satisfies Condition (L)} if every cycle  in  $E$ has an entrance. A graph is said to \emph{satisfy Condition (K)} if for every vertex $v\in E^0$, either there are no cycles based at $v$, or  there are two distinct cycles $\alpha$ and $\mu$ such that $v=s(\alpha)=s(\mu)$ and neither $\alpha$ nor $\mu$ is an initial subpath of the other. 

The $C^*$-algebra $C^*(E)$ is generated by a universal Cuntz-Krieger $E$-family consisting of partial isometries $\{s_e: e\in E^1\}$ and mutually  orthogonal projections $\{p_v: v\in E^1\}$ such that  $s_e^*s_e=p_{s(e)}$, $s_e s_e^*\leq p_{r(e)}$   and $p_v=\sum_{r(e)=v} s_e s_e^*$ whenever the sum is finite (i.e. whenever $v$ is a  finite receiver). It follows that 
$$
C^*(E)=\clsp\{s_\mu s_\nu^*: \mu, \nu \in E^*\},
$$ 
where $E^*=\bigcup_{n\in \N} E^n$, $s_\mu=s_{\mu_1} s_{\mu_2}....s_{\mu_n}$ for $\mu=\mu_1...\mu_n\in E^n$, $n>0$, and $s_\mu=p_{\mu}$ for $\mu\in E^0$. We extend the maps $r$ and $s$ onto $E^*$ in the obvious way. The algebra $C^*(E)$ is graded by the subspaces
$$
B_{n}:=\clsp\{s_\mu s_\nu^*: \mu, \nu \in E^*, |\nu|-|\mu|=n \}, \qquad n \in \Z. 
$$
The only Fell bundle considered in this section  will be $\B=\{B_n\}_{n \in \Z}$ defined above. The gauge-uniqueness theorem  readily implies that $C^*(E)=C^*(\B)$. In particular, graded ideals in $C^*(E)$ coincide with gauge-invariant ideals in $C^*(E)$ and their description is known. We now use  it to describe the corresponding  $\B$-invariant ideals in $B_0$, cf. Proposition \ref{proposition on induced ideals}. 

For every $v$, $w\in E^0$ we write $v\geq w$ if there is a path from $w$ to $v$. A subset $H$ of $E^0$ is \emph{hereditary} if $v\in H$ and $v\geq w$ 
imply $w\in H$. A subset $H$ of $E^0$ is \emph{saturated} if every vertex $v$ which satisfies $0<|r^{-1}(v)|<\infty$ and $r(e)=v\, \Rightarrow\, s(e)\in H$ itself belongs to $H$. Given a saturated hereditary subset $H\subseteq E^0$, define
$$
H_\infty^{\fin}:=\{v\in E^0\setminus H: |r^{-1}(v)|=\infty\textrm{ and } 0<|r^{-1}(v)\cap s^{-1}(E^0\setminus H)|<\infty\}.
$$
If $v\in H_\infty^{\fin}$ we write 
$$
p_{v, H}:=\sum_{r(e)=v, s(e)\notin H} s_es_e^*.
$$
For any $B\subseteq H_\infty^{\fin}$ we put  
$$
I_{H,B}:=\clsp\{s_\alpha p_v s^*_\beta, s_\mu (p_w-p_{w, H})s^*_\nu: v\in H, w\in  B, \alpha,\beta,\mu,\nu \in E^{n}, \, n\in \N\}.
$$ 
Clearly, $I_{H,B}$ is an ideal in $B_0$ (in fact, it is a $\B$-invariant ideal).
Let us consider the set  
$$
\mathcal{H}_E:=\{(H,B): H \textrm{ is saturated and hereditary, } B \subseteq H_\infty^{\fin}\}.
$$
We equip it with a partial order (actually a lattice) structure by writing 
$$
(H,B)\leq   (H',B') \,\,\,\, \stackrel{def}{\Longleftrightarrow} \,\,\,\, H\subseteq H'  \,\,\textrm{ and }  \,\,  B\subseteq H'\cup B'.
$$
The following description follows readily from the analysis in \cite{bhrs}. 
\begin{prop}\label{description of B-invariant ideals for graph algebras}
The mapping  $\mathcal{H}_E \ni (H,B) \mapsto I_{H,B}\in \I^\B(B_e)$ is a well-defined order preserving bijection (hence a lattice isomorphism). 
\end{prop}
\begin{proof}
In view of \cite[Theorem 3.6]{bhrs} we see that the mapping under consideration is a bijection. Using  \cite[Corollary 3.10]{bhrs} we get that inclusion $I_{H,B}\subseteq I_{H',B'}$ holds if and only if $(H,B)\leq  (H',B')$.
\end{proof}
\begin{rem}\label{meet for J-pairs} It follows that $\mathcal{H}_E$ is a lattice and the meet operation is given by the formula 
$$
(H,B)\wedge   (H',B') =(H\cap H', (H\cap B') \cup (B\cap H') \cup (B\cap B'))
,
$$ cf. \cite[Proposition 3.9]{bhrs}. 
\end{rem}
Having the above description of $\B$-invariant ideals in hand, we can show that the conditions introduced in the present paper are natural,  can be verified in practice, and in general cannot be weakened.  Let us start with the notion of residual aperiodicity. 
\begin{prop}\label{aperiodicity for graph algebras}  The following statements are equivalent:
\begin{itemize}
\item[(i)] $\B$ is aperiodic.
\item[(ii)]  $\B$ has the intersection property.
\item[(iii)] $E$ satisfies Condition (L). 
\end{itemize}
If $E$ is finite, i.e. both $E^0$ and $E^1$ are finite, then  the above conditions are equivalent to
\begin{itemize}
\item[(iv)] the  partial dynamical system $(\{\widehat{D}_n\}_{n\in Z}, \{\h_n\}_{n\in Z})$  on $\widehat{B}_0$ dual to $\B$ is topologically free.
\end{itemize}
\end{prop}
\begin{proof} By Corollary \ref{aperiodicity imply intersection property}, we have  (i)$\Rightarrow$ (ii). Implication  (ii)$\Rightarrow$ (iii) belongs to the folklore in the field, see, for instance,  the proof of \cite[Lemma 2.1]{hs}.

(iii)$\Rightarrow$ (i).
Let  $b=\sum_{\alpha, \beta \in E^*}\lambda_{\alpha, \beta} s_{\alpha}s_{\beta}^*$ be an element of  $B_n$, where $n\neq 0$, and let $A$ be a  hereditary $C^*$-subalgebra of $B_0$. A moment of thought yields that to show  condition  \eqref{aperiodicity condition} it suffices to consider the case when  $n>0$, the number of non-zero coefficients $\lambda_{\alpha, \beta}\in\C$ is finite and $A$ is a corner of the form
$$
(s_\mu s_\mu^*) B_0 (s_\mu s_\mu^*)=\clsp\{s_{\mu\eta} s_{\mu\nu}^*: \eta, \nu \in E^*, \, | \eta|=|\nu|\}
$$
where $\mu\in E^n$. In particular, we may further assume that all $\alpha$'s and $\beta$'s appearing in the sum $b=\sum_{\alpha, \beta \in E^*}\lambda_{\alpha, \beta} s_{\alpha}s_{\beta}^*$ start with the path $\mu$, that is there is a finite set 
$$
F\subseteq \{(\alpha,\beta) \in \bigcup_{k \geq |\mu|} E^k\times E^{k+n}: s(\alpha)=s(\beta),\,\, \alpha=\mu\alpha',\,\, \beta=\mu\beta'\}
$$ 
such that $b=\sum_{(\alpha, \beta) \in F}\lambda_{\alpha, \beta} s_{\alpha}s_{\beta}^*$.  Let $(\alpha_0,\beta_0)\in F$ be such that   $(\alpha,\beta)\in F$ implies  $|\beta|\leq |\beta_0|$. If the vertex $s(\alpha_0)=s(\beta_0)$ does not lie on a cycle (of length $n$) then 
$$
s_{\beta_0}^* s_\alpha  =0 \qquad \textrm{for all }(\alpha, \beta) \in F.
$$
 Thus $ab=0$ with  $a=s_{\beta_0}s_{\beta_0}^*\in A$. If $s(\alpha_0)=s(\beta_0)$ lies on a cycle, then using the fact that this cycle has an entrance one can construct a path $\eta=\eta_1....\eta_{|\eta|}$ with $|\eta|\geq |\beta_0|$ such that   $\eta_1...\eta_{|\beta_0|}=\beta_0$ and   
\begin{equation}\label{equation not equality}
\eta_{|\beta_0|+1}...\eta_{|\eta|}\neq  \eta_{|\alpha_0|+1}...\eta_{|\eta|-n}.
\end{equation}
Then for all $(\alpha, \beta) \in F$ we see that  $s_{\eta}^* (s_\alpha s_\beta^*) s_{\eta}$ is either zero, when $\alpha\neq \eta_{1}...\eta_{|\alpha|}$ or  $\beta\neq \eta_{1}...\eta_{|\beta|}$, or it is equal to  $s_{\eta_{|\alpha|+1}...\eta_{|\eta|}}^*  s_{\eta_{|\beta|+1}...\eta_{|\eta|}}$, which is also  zero by \eqref{equation not equality}.
Thus $aba=0$ with  $a=s_{\eta}s_{\eta}^*\in A$.

If  $E$ is finite then we have (iv)$\Leftrightarrow$ (ii) by \cite[Theorem 3.19(I)]{kwa-interact} modulo \cite[Propositions 2.19 and 3.2]{kwa-interact} and the fact that the  partial dynamical system $(\{\widehat{D}_n\}_{n\in Z}, \{\h_n\}_{n\in Z})$ is generated by the single partial homeomorphism $h_1$, cf. \cite{kwa}.
\end{proof}
\begin{rem} By Theorem \ref{uniqueness theorem for fell bundles}, the implication   (iv)$\Rightarrow$ [(i)$\Leftrightarrow$(ii) $\Leftrightarrow$(iii)] in the above proposition is valid for an arbitrary graph $E$. We suspect that the converse implication is also true and  the proof   would in essence require generalizing \cite[Theorem 3.19]{kwa-interact}, however this is beyond the scope of the present paper.\end{rem}
\begin{cor}\label{residual aperiodicity corollary} 
 The following statements are equivalent:
\begin{itemize}
\item[(i)] $\B$ is residually aperiodic.
\item[(ii)]  $\B$ has the residual intersection property.
\item[(iii)] $E$ satisfies Condition (K). 
\end{itemize}
If $E$ is finite, that is both $E^0$ and $E^1$ are finite, then  the above conditions are equivalent to
\begin{itemize}
\item[(iv)] the  partial dynamical system $(\{\widehat{D}_n\}_{n\in Z}, \{\h_n\}_{n\in Z})$  dual to $\B$ is residually topologically free.
\end{itemize}
\end{cor}
\begin{proof} For any $(H,B)\in \mathcal{H}_E$ let $J_{H,B}$ denote the ideal in $C^*(E^*)$ generated by $I_{H,B}$. By Proposition \ref{description of B-invariant ideals for graph algebras}, every graded ideal in $C^*(E^*)$ is of this form. By \cite[Corollary  3.6]{bhrs} the quotient $C^*(E^*)/J_{H,B}$ is naturally isomorphic to the graph $C^*$-algebra $C^*((E/H)\setminus \beta(B))$ of a certain graph $(E/H)\setminus \beta(B)$, see \cite{bhrs}. Moreover, it is well known, and follows, for instance, from the proof of \cite[Corollary 3.8]{bhrs}, that $E$ satisfies Condition (K) if and only if every graph $(E/H)\setminus \beta(B)$ satisfies Condition (L). Now  the assertion  follows from Proposition \ref{aperiodicity for graph algebras}.
\end{proof}

By Lemma \ref{lemma on orbit spaces2}, we have $\Prime^\B(B_0)=\Prim^\B(B_0)$. We now turn to a description of this set. This will allow us, in the presence of condition (K), to deduce from Theorem \ref{Primitive ideal space description} description of $\Prim(C^*(E)) $  originally obtained using a different approach in \cite[Corollary  4.8]{bhrs}. 
\begin{lem}\label{maximal tails lemma}
Suppose that $I_{H,B}\in \I^\B(B_0)$,  $(H,B)\in \mathcal{H}_E$, belongs to $\Prime^\B(B_0)$. Then $M:=E^0\setminus H$ satisfies:
\begin{itemize}
\item[(a)] If $v\in E^0$, $w\in M$, and $v\geq w$ in $E$, then  $v\in M$,
\item[(b)] If $v\in M$ and $0<|r^{-1}(v)|<\infty$, then there exists  $e\in E^1$ with $r(e)=v$ and  $s(e)\in M$,
\item[(c)] For every  $v,w\in M$ there exists $y\in M$ such that   $v \geq y$ and  $w \geq y$.
\end{itemize}
\end{lem}
\begin{proof} Conditions (a), (b) say that   $H$ is  hereditary and saturated. In the proof of \cite[Lemma  4.1]{bhrs} it was shown that if condition (c) is not satisfied then $I_{H,B}\notin \Prime^\B(B_0)$. 
\end{proof}
A non-empty  subset $M$ of $E^0$ satisfying conditions (a), (b), (c) of Lemma \ref{maximal tails lemma} is called a \emph{maximal tail} in $E$. For any non-empty subset $X$ of $E^0$ we write 
$$
\Omega(X):=\{w\in E^0\setminus X: w \ngeq v \textrm{ for all }v\in X\}.
$$
We also put $\Omega(v):=\Omega(\{v\})$.
Note that  $\Omega(M)=E^0\setminus M$ for any maximal tail $M$. Moreover, for any $v\in E^0$ that receives infinitely many edges, the set $\Omega(v)$ is hereditary and saturated; actually  $\Omega(v)$ is a complement of a maximal tail. Any such vertex with the property that $v\in \Omega(v)_\infty^{\fin}$ is called a \emph{breaking vertex}. In other words, $v\in E^0$ is a breaking vertex if and only if $|r^{-1}(v)|=\infty$ and $0<|r^{-1}(v)\cap  s^{-1}(\Omega(v))|<\infty$.  
\begin{prop}\label{Prime description for graphs}
The set  $\Prime^\B(B_0)=\Prim^\B(B_0)$ consists of  ideals $I_{\Omega(M),\Omega(M)_\infty^{\fin}}$ associated to maximal tails $M$, and  ideals  $I_{\Omega(v),\Omega(v)_\infty^{\fin}\setminus{\{v\}}}$ associated to breaking vertices $v$.
\end{prop}
\begin{proof}
Suppose that $I_{H,B}$ is in  $\Prime^\B(B_0)$. By Lemma \ref{maximal tails lemma},  we have  $H=\Omega(M)$ where  $M$ is a maximal tail. Note that the set  $\Omega(M)_\infty^{\fin}\setminus B$ can not contain more than one vertex. Indeed, if we had two distinct vertices $v_1,v_2  \in \Omega(M)_\infty^{\fin}\setminus B$, then we would get two $\B$-invariant ideals $I_i:=I_{\Omega(M), B\cup \{v_i\}}$, $i=1,2$,  such that $I_1\cap I_2\subseteq I_{\Omega(M),B}$ and $I_i \nsubseteq I_{\Omega(M),B}$ for $i=1,2$. Suppose then that $B=\Omega(M)_\infty^{\fin}\setminus\{v\}$ for a vertex $v\in \Omega(M)_\infty^{\fin}$. If $v$ is not a breaking vertex then we get two $\B$-invariant ideals $I_1:=I_{\Omega(v),\Omega(v)_\infty^{\fin}}$ and  $I_2:=I_{\Omega(M),\Omega(M)_\infty^{\fin}}$ such that $I_1\cap I_2\subseteq I_{\Omega(M),B}$ and $I_i \nsubseteq I_{\Omega(M),B}$ for $i=1,2$, a contradiction. Hence $v$ must be a breaking vertex.

Let us now consider an ideal $I_{\Omega(M),B}$ where $M$ is a  maximal tail and $B=\Omega(M)_\infty^{\fin}$ or $B=\Omega(M)_\infty^{\fin}\setminus\{v\}$  where $v$ is a breaking vertex such that $M=\Omega(v)$. We need to show that  $I_{\Omega(M),B}$ is  $\B$-prime. 
Suppose that   $I_{H_1,B_1}$, $I_{H_2,B_2}\in \I^\B(B_0)$ are such that $I_{H_1,B_1}\cap I_{H_2,B_2}\subseteq I_{\Omega(M),B}$. By Proposition \ref{description of B-invariant ideals for graph algebras} and Remark \ref{meet for J-pairs} this is equivalent to saying that  $H_1\cap H_2 \subseteq \Omega(M)=E^0\setminus M$ and 
\begin{equation}\label{inclusion leading to primeness} 
H_1\cap B_2 \cup B_1\cap H_2\cup B_1\cap B_2  \subseteq  \Omega(M)\cup  B.
\end{equation} 
We claim that  either $H_1\subseteq \Omega(M)$ or $H_2\subseteq \Omega(M)$. Indeed, if we assume on the contrary that there is $v\in H_1\cap M$ and $w\in H_2\cap M$, then  taking $y\in M$ such that $v \geq y$ and  $w \geq y$ we get $y\in H_1\cap H_2\cap M$, a contradiction.  

Thus we may assume that $H_1\subseteq \Omega(M)$.  If $(H_1,B_1) \leq (\Omega(M),B)$, then $I_{H_1,B_1}\subseteq I_{\Omega(M),B}$.  Suppose  then that $(H_1,B_1) \nleq (\Omega(M),B)$, and consider two cases:

1). Let $B=\Omega(M)_\infty^{\fin}$. Since  $H_1\subseteq \Omega(M)$, $(H_1,B_1) \nleq (\Omega(M),B)$ implies that there is $v\in B_1\cap M$ such that $|r^{-1}(v)\cap s^{-1}(M)|=0$. By  properties (a) and (c) in Lemma \ref{maximal tails lemma}, we see that $M=\{w\in E^0: w \geq v \}$ (since $v$ is a `source' in $M$, $v$ is a unique vertex with this property). In particular, if $H_2\cap M\neq \emptyset$ then $v\in H_2$.  However, in view of \eqref{inclusion leading to primeness} we see that $v$ cannot belong neither to $H_2$ nor to $B_2$. Hence $H_2\subseteq \Omega(M)$ and since $v\notin B_2$ we  must have  $(H_2,B_2) \leq (\Omega(M),B)$, that is   $I_{H_2,B_2}\subseteq I_{\Omega(M),B}$.

2). Let $B=\Omega(M)_\infty^{\fin}\setminus\{v\}$,  where $v$ is a breaking vertex and $\Omega(M)=\Omega(v)$. The last relation implies that $M=\{w\in E^0: w \geq v \}$. Note that $B_1$ must contain $v$. Indeed, if $v\notin B_1$ then arguing as in  case 1) we see that there is a unique $v'\in B_1\cap M$ with the property that $M=\{w\in E^0: w \geq v' \}$. This leads to a contradiction as $v\neq v'$. Thus $v\in B_1$. It follows from \eqref{inclusion leading to primeness}  that $v\notin H_2$ and $v\notin B_2$. Since $M=\{w\in E^0: w \geq v' \}$ and $H_2$ is hereditary, we conclude that $H_2\subseteq \Omega(M)$. Combining the last inclusion with  $v\notin B_2$ we get that  $(H_2,B_2) \leq (\Omega(M),B)$. Hence   $I_{H_2,B_2}\subseteq I_{\Omega(M),B}$.
\end{proof}
\begin{cor}\label{primitives for graphs}
If the equivalent conditions in Corollary \ref{residual aperiodicity corollary} hold,  then the elements in the primitive spectrum $\Prim(C^*(E))$ of the graph algebra $C^*(E)$ are in a bijective correspondence with maximal tails and breaking vertices in the graph $E$.
\end{cor}
\begin{proof}
Combine  Proposition \ref{Prime description for graphs} with Theorem \ref{Primitive ideal space description}. 
\end{proof}
Now, we show that Theorem \ref{pure infiniteness for paradoxical Fell bundles} is optimal in the sense that when applied to graph $C^*$-algebras our conditions implying pure infiniteness are not only sufficient but also necessary.
\begin{thm}\label{pure infiniteness of graph algebras} 
The $C^*$-algebra $C^*(E)$ of a directed graph $E$ is purely infinite if and only if   the associated Fell bundle $\B=\{B_n\}_{n\in \Z}$  is residually aperiodic  and every non-zero projection in   is residually strictly $\B$-infinite.
\end{thm}
\begin{proof}
The `if' part follows from  Theorem \ref{pure infiniteness for paradoxical Fell bundles}(ii). To show the `only if' part, suppose that $C^*(E)$  is purely infinite.  Every projection in   $B_0$ is Murray-von Neumann equivalent   to a projection of the form $\sum_{\alpha, \beta \in F}\lambda_{\alpha, \beta} s_{\alpha}s_{\beta}^*$ where $F\subseteq E^*$ is finite. The latter is in turn a finite sum of mutually orthogonal projections each of which is Murray-von Neumann equivalent to a projection of the form $s_\mu s_\mu^*$ with $\mu \in E^*$. Thus, by Lemma \ref{sums and equivalence for infiniteness}, it suffices to show that $a:=s_\mu s_\mu^*$ is residually strictly $\B$-infinite. Let then $ I_{H,B}$, $(H,B)\in \mathcal{H}_E$, be a  $\B$-invariant ideal in $B_0$.  Suppose  that $a\notin I_{H,B}$. Then   $v:=s(\mu)\in E^0\setminus H$. Since  $I_{H,B}$ is an intersection of $\B$-prime ideals, see Lemma \ref{intersection of prime ideals lemma}, we see that $E^0\setminus H$ is a sum of maximal tails, cf. Remark \ref{meet for J-pairs}. Thus   $v\in M$ for a certain maximal tail $M$ contained in $E^0\setminus H$. By \cite[Theorem 2.3]{hs}, there is a path $\alpha$ in $M$ which connects a loop $\gamma$ in $M$ to  $v$   and $\gamma$ has  an entrance in $M$. This implies that  $s_{\mu\alpha}  s_{\mu\alpha}^* -  s_{\mu\alpha \gamma}  s_{\mu\alpha\gamma}^*\notin  I_{H,B}$. Moreover, putting 
$$
a_1=s_{\mu\alpha \gamma}  s_{\mu\alpha}^*,\qquad a_2=s_\mu s_\mu^*- s_{\mu\alpha} s_{\mu\alpha}^*, \qquad b=s_{\mu\alpha}  s_{\mu\alpha}^* -  s_{\mu\alpha \gamma}  s_{\mu\alpha\gamma}^*, 
$$ 
one readily sees that  $b\in aB_0 a$,  $a_1\in aB_{-|\gamma|}$,   $a_2\in aB_0$ and 
$$
s_\mu s_\mu^*=a_1^*a_1+ a_2^*a_2, \qquad a_1^*a_2=a_1^*b=a_2^*b=0.
$$
Hence the image of $a$ in the quotient $B_0/I_{H,B}$ is $\B/\J$-infinite where $\J:=\{B_n I_{H,B}\}_{n\in \Z}$.
 \end{proof}
 \begin{rem}\label{Remark below}
 One could conjecture that  if $C^*(E)$ is purely infinite then every projection in  $B_e$ is $\B$-paradoxical (or at least is equivalent to a $\B$-paradoxical one). However, even for finite graphs this conjecture is very hard to verify. In particular, the above theorem illustrates the practical advantage of  residual $\B$-infiniteness over $\B$-paradoxicality.
 \end{rem}

\section{Crossed products by a semigroup of corner systems}\label{Crossed products by a semigroup of corner systems}
In this section, we  consider systems  $(A,G^+,\alpha,L)$ studied  in \cite{Kwa-Leb} when $A$ is a unital $C^*$-algebra. We will generalize concepts of \cite{Kwa-Leb} to non-unital case. Combining them with the findings of \cite{kwa-exel} we reveal their internal structure and  connections with other construction. Then  we apply the results of the present paper to $C^*$-algebras associated to $(A,G^+,\alpha,L)$.

\subsection{Various pictures of semigroup corner systems} Throughout this subsection,   we let
$G^+$ be the positive cone of a totally ordered abelian group
$G$ with the identity $0$, that is we have:
$$
G^+\cap (-G^+)=\{0\}, \qquad G=G^+ \cup (-G^+),
$$
and we write $g \geq h$ if $g-h\in G^+$, $g,h\in G$.
We  fix a $C^*$-algebra  $ A$. We denote by $\End(A)$ the set of endomorphisms of $A$ and by $\Pos(A)$ the space of positive maps on $A$. Composition of mappings yields a structure  of a unital semigroup on $\End(A)$  and $\Pos(A)$, where the unit is the identity map.  We always assume that  a homomorphism between two semigroups with unit preserves the units.  We recall  that a \emph{multiplicative domain} of $\morp\in \Pos(A)$ is a $C^*$-algebra given by 
$$
MD(\morp)=\{a\in A: \morp(ba)=\morp(b)\morp(a), \quad  \morp(ab)=\morp(a)\morp(b), \quad \textrm{for all }b\in A\},
$$
cf.,  for instance, \cite{kwa-exel} and references therein.
We say that $\morp:A\to A$ is \emph{extendible} if it extends to a strictly continuous map $\overline{\morp}:M(A)\to M(A)$.

An \emph{Exel system}, originally defined in \cite{exel3},      is a  triple  $(A,\alpha,L)$ where $\alpha\in  \End(A)$ and $L\in \Pos(A)$ are such that $
L(\alpha(a)b)=aL(b)$,  for all $a,b\in A$. Then $L$ is called a \emph{transfer operator} for $\alpha$, and as shown in \cite[Proposition 4.2]{kwa-exel}, $L$ is automatically extendible.  In \cite[Definition 4.19]{kwa-exel}, an Exel system $(A,\alpha,L)$ is called a \emph{corner  system} if $E:=\alpha \circ L$ is a conditional expectation onto a hereditary  $C^*$-subalgebra $\alpha(A)$ of $A$. By  \cite[Lemma 4.20]{kwa-exel}, an Exel system $(A,\alpha,L)$ is a corner system if and only if $\alpha$ is extendible and 
\begin{equation}\label{relation for complete operators}
\alpha(L(a))=\overline{\alpha}(1) a \overline{\alpha}(1), \qquad\textrm{ for all }  a\in A.
\end{equation}
Note that then both  maps $\alpha$ and $L$ are extendible  and $(M(A),\overline{\alpha},\overline{L})$ is again a corner system. In unital case, systems $(A,\alpha,L)$ satisfying \eqref{relation for complete operators} were called \emph{complete} in \cite{Ant-Bakht-Leb}. One readily sees that an iteration  $(A,\alpha^n,L^n)$,  $n\in\N$, of  an Exel system $(A,\alpha,L)$ is an Exel system, and  if  $(A,\alpha,L)$ is a corner system then $(A,\alpha^n,L^n)$ is a corner system (one can check \eqref{relation for complete operators} inductively). Thus these systems can be treated as semigroup dynamical systems with the underlying semigroup $\N$. We will  consider  systems over the more general semigroup  $G^+$. We will use \cite[Lemma 4.20]{kwa-exel} to show that the following three `corner systems' are actually  different sides of the same `coin'. They form a subclass of Exel-Larsen systems introduced  in \cite{Larsen}, and they may be viewed as generalizations to the non-unital case   of (finely representable) $C^*$-dynamical systems considered in \cite{Kwa-Leb}.

 Recall that a $C^*$-subalgebra $B$ of a $C^*$-algebra $A$ is a corner in $A$ if there is a projection $p\in M(A)$ such that $B=pAp$. In particular, an ideal $I$ in a $C^*$-algebra is a corner in $A$ if and only if it is complemented in $A$, that is if $A$ is a direct sum of $I$ and $I^\bot$.
\begin{defn} Consider two semigroup homomorphisms $\al:G^+\rightarrow \End(A)$ and $L:G^+\rightarrow \Pos(A)$. We say that:
\begin{itemize}
\item[(i)] $(A,G^+,\alpha,L)$ is a \emph{semigroup corner (Exel) system} if for each $ t\in G^+$, $(A,\alpha_t,L_t)$ is a corner system.
\item[(ii)] $\al$ is a \emph{semigroup of corner endomorphisms} if each $\alpha_t$, $t\in G^+$,  has a corner range and a complemented kernel (note that then $\alpha_t$ is necessarily extendible and $\alpha_t(A)=\overline{\alpha_t}(1)A\overline{\alpha_t}(1)$).
\item[(iii)]  $L$ is a \emph{semigroup of corner retractions} if for every $t\in G^+$, $L_t(A)$ is a complemented ideal in $A$ and the annihilator $(\ker L_t|_{MD(L_t)})^{\bot}$ of the  kernel of  $L_t:MD(L_t)\to L_t(A)$ is mapped by $L_t$ onto $L_t(A)$. 
\end{itemize}
\end{defn}
\begin{rem} Suppose that $(A,G^+,\alpha,L)$ is a  semigroup corner system. Clearly $t \mapsto \overline{\alpha}_t$  and $t\mapsto \overline{L}_t$ define  semigroup homomorphisms $\overline{\al}:G^+\rightarrow \End(M(A))$ and $\overline{L}:G^+\rightarrow \Pos(M(A))$.   Thus $(M(A),G^+,\overline{\alpha},\overline{L})$ is a semigroup corner system.  In particular, $(A,G^+,\alpha,L)$ is a $C^*$-dynamical system in the sense of \cite{Larsen} and  $(M(A),G^+,\overline{\alpha},\overline{L})$  is a dynamical system  in the sense of \cite{Kwa-Leb}. Note also, cf. \cite[2.2]{Kwa-Leb}, that 
\begin{equation}\label{families of projections}
\{\overline{\alpha}_t(1)\}_{t\in G^+}\subseteq M(A)\qquad  \textrm{ and }\qquad \{\overline{L}_t(1)\}_{t\in G^+}\subseteq Z(M(A))
\end{equation}
are  non-increasing families of projections (the latter are central).
\end{rem}

\begin{prop}  Let $\al:G^+\rightarrow \End(A)$ and $L:G^+\rightarrow \Pos(A)$ be two maps. 
The following statements are equivalent:
\begin{itemize}
\item[(i)] $(A,G^+,\alpha,L)$ is a  semigroup corner system;
\item[(ii)] $\al$ is a semigroup of corner endomorphisms, and for every  $t\in G^+$
\begin{equation}\label{corner transfer formula}
 L_t(a)=\alpha_t^{-1}(\overline{\alpha}_t(1) a \overline{\alpha}_t(1)), \qquad a\in A,
\end{equation}
where $\alpha_t^{-1}$ is the inverse to the isomorphism $\alpha_t:(\ker\alpha_t)^{\bot}\to \overline{\alpha_t}(1)A\overline{\alpha_t}(1)$;
\item[(iii)]  $L$ is semigroup of corner retractions, and for every $t\in G^+$
\begin{equation}\label{corner endomorphism formula}
\alpha|_{L_t(A)^{\bot}}\equiv 0   \quad \textrm{ and  }\quad \alpha_t(a)=L_t^{-1}(a), \,\, \textrm{ for } a\in L(A),
\end{equation}
where $L_t^{-1}$ is the inverse to the isomorphism $L_t:(\ker L_t|_{MD(L_t)})^{\bot}\to L_t(A)$.
\end{itemize}
In particular, if the above equivalent conditions hold, then  $\alpha$ and $L$ determine each other uniquely, and
\begin{equation}\label{relation for domain predomain}
\alpha_t(A)=\overline{\alpha_t}(1)A\overline{\alpha_t}(1)=(\ker L_t|_{MD(L_t)})^{\bot}, \qquad L_t(A)=\overline{L}_t(1)A=(\ker\alpha_t)^{\bot},
\end{equation}
for each $t\in G^+$.
\end{prop}
\begin{proof} By \cite[Lemma 4.20]{kwa-exel}, for each $t\in G^+$ the following statements  are equivalent:
\begin{itemize}
\item $(A,\alpha_t,L_t)$ is a corner system; 
\item $\alpha_t$ has a corner range,   complemented kernel and $L_t$ is given by \eqref{corner transfer formula};
\item  $L_t(A)$ is a complemented ideal in $A$,   $L_t((\ker L_t|_{MD(L_t)})^{\bot})=L_t(A)$ and $\alpha_t$ is given by \eqref{corner endomorphism formula};
\end{itemize}
and if these conditions are satisfied then \eqref{relation for domain predomain} holds.  Thus item (i) implies (ii) and (iii). To show  (ii)$\Rightarrow$ (i), we need to check that $L_t$'s given by \eqref{corner transfer formula} satisfy the semigroup law. To this end, note that $\{\overline{\alpha}_t\}_{t\in G^+}$ form a semigroup of endomorphisms of $M(A)$, and in particular $\{\overline{\alpha}_t(1)\}_{t\in G^+}$ is  a  non-increasing family of projections, cf. \cite[Page 405]{Kwa-Leb}. Thus for any $t,s\in G^+$ and $a\in A$
we have
\begin{align*}
L_s(L_t(a))&=\alpha_s^{-1}\left(\overline{\alpha}_s(1) L_t(a) \overline{\alpha}_s(1)\right)=\alpha_s^{-1}\left( L_t(\overline{\alpha}_t(\overline{\alpha}_s(1))a\overline{\alpha}_t(\overline{\alpha}_s(1))) \right)\\
&=\alpha_s^{-1}\left(\alpha_t^{-1}(\overline{\alpha}_{s+t}(1)a\overline{\alpha}_{s+t}(1)) \right)=\alpha_{s+t}^{-1}\left(\overline{\alpha}_{s+t}(1)a\overline{\alpha}_{s+t}(1) \right)\\
&=L_{t+s}(a).
\end{align*}
Similarly, to show  (iii)$\Rightarrow$ (i) we need to check that $\al_t$'s given by \eqref{corner endomorphism formula} satisfy the semigroup law. To this end, note that  $\overline{L}_t(1)$ is a central projection in $M(A)$ such that $\overline{L}_t(1)A=L_t(A)$. In particular,  $\alpha_t$ is  given by the formula 
$$
\alpha_t(a)=L_t^{-1}(\overline{L}_t(1)a), \qquad a\in A.
$$
Moreover, since $L_{s+t}(A)=L_{t}(L_s(A))\subseteq L_{t}(A)$ we conclude that  the family  $\{\overline{L}_t(1)\}_{t\in G^+}$ is non-increasing.
Now,  using \eqref{relation for complete operators}, for any $t,s\in G^+$ and $a\in A$ we get
$$
\overline{\alpha}_t(\overline{L}_{s+t}(1))=\overline{\alpha}_t(\overline{L}_{t}(\overline{L}_{s}(1)))=\overline{\alpha}_t(1)\overline{L}_{s}(1)\overline{\alpha}_t(1)=\overline{\alpha}_t(1)\overline{L}_{s}(1).
$$
Hence 
\begin{align*}
\al_s(\al_t(a))&=L_s^{-1}(\overline{L}_{s}(1)\al_t(a))=L_s^{-1}(\overline{L}_{s}(1)\overline{\al}_t(1)\al_t(a))=L_s^{-1}(\overline{\alpha}_t(\overline{L}_{s+t}(1))\al_t(a))
\\
&=L_s^{-1}(\overline{\alpha}_t(\overline{L}_{s+t}(1)a))=L_s^{-1}(L_t^{-1}(\overline{L}_{s+t}(1)a))=L_{s+t}^{-1}(\overline{L}_{s+t}(1)a)
\\
&=\al_{s+t}(a).
\end{align*}
\end{proof}
We also reveal a connection between   semigroup corner systems $(A,G^+,\alpha,L)$ and the concept of an interaction group introduced in \cite{exel4}. We emphasize that Exel in \cite{exel4} considered arbitrary discrete groups but he assumed that  the algebra $A$ and the maps involved in an interaction are unital. Since the latter  requirements  applied to $(A,G^+,\alpha,L)$ would force $\alpha$ and $L$ to act by automorphisms, we consider a version of \cite[Definition 3.1]{exel4} where we drop  the assumptions on  units. We formulate it using our abelian group $G$.

\begin{defn}[cf. Definition 3.1 of \cite{exel4}]
An \emph{interaction group} is a triple $(A,G,\VV)$ where $\VV:G\to \Pos(A)$ is a partial representation, that is, $\VV_0=id$,  and
\begin{equation}\label{partial representation}
 \VV_g \VV_h \VV_{-h}=\VV_{g+h}\VV_{-h},  \qquad \VV_{-g} \VV_{g} \VV_{h}=\VV_{-g}\VV_{g+h}, \qquad g,h \in G; 
\end{equation}
and for each $g\in G$ the pair $(\VV_g,\VV_{-g})$ is an interaction in the sense of \cite{exel-inter}, that is in view of \eqref{partial representation} it suffices to require that 
$$
\VV_g(A)\subseteq MD(\VV_{-g}), \qquad\textrm{and}\qquad  \VV_{-g}(A)\subseteq MD(\VV_{g}),
$$
cf. \cite[Remark 3.15]{kwa-exel}.
\end{defn}
\begin{prop}
Suppose that 
$(A,G^+,\alpha,L)$ is a  semigroup corner system. Then putting
\begin{equation}\label{corner interaction}
\VV_g:=
\begin{cases}\alpha_g  & \textrm{if } g \geq 0,\\
L_{-g}  & \textrm{if } g \leq 0,
\end{cases} \qquad g \in G,
\end{equation}
we get an interaction group $\VV$.
\end{prop}
\begin{proof} It is tempting to apply Proposition 13.4 from \cite{exel4}. However, its proof relies  in an essential way on the the assumption that the maps preserve units, which we dropped in our setting. Thus we need to prove it ad hoc. 
In particular, by \cite[Proposition 4.13]{kwa-exel} we know that for each $g\in G$, $(\VV_g,\VV_{-g})$ is an interaction. Hence it suffices to check \eqref{partial representation}. To this end,  note that for $t\in G^+$ and $a\in A$ we have
$$
(\VV_t\VV_{-t})(a)=\overline{\alpha}_t(1)a\overline{\alpha}_t(1),\qquad (\VV_{-t}\VV_t)(a)=\overline{L}_t(1)a.
$$
We recall that  the families of projections \eqref{families of projections} are non-increasing. 
We proceed with a case by case proof:

1) If $g,h\geq 0$ or $g,h \leq 0$ then relations \eqref{partial representation} hold by semigroup laws for $\alpha$ and $L$.

2) If  $g \leq 0\leq h$ and  $ h\leq -g$, then 
$$
(\VV_g \VV_h \VV_{-h})(a)=L_{-g}( \overline{\alpha}_h(1)a\overline{\alpha}_h(1))=L_{-g}(a)=L_{-g-h}(L_h(a))=(\VV_{g+h} \VV_{h})(a),
$$
\begin{align*}
(\VV_{-g} \VV_g \VV_{h})(a)&= \overline{\alpha}_{-g}(1)\alpha_h(a)\overline{\alpha}_{-g}(1)=\alpha_h(\overline{\alpha}_{-g-h}(1)a\overline{\alpha}_{-g-h}(1))
\\
&= \alpha_h(\alpha_{-g-h}(L_{-g-h}(a)))=\alpha_{-g}(L_{-g-h}(a))=(\VV_{-g} \VV_{g+h})(a).
\end{align*}

3) If   $g \leq 0\leq h$ and  $ h\geq -g$, then 
\begin{align*}
(\VV_g \VV_h \VV_{-h})(a)&=L_{-g}( \overline{\alpha}_h(1)a\overline{\alpha}_h(1))=\overline{\alpha}_{g+h}(1)L_{-g}(a)\overline{\alpha}_{g+h}(1)
\\
&=\alpha_{g+h}(L_{g+h}(L_{-g}(a)))=\alpha_{g+h}(L_{h}(a))=(\VV_{g+h} \VV_{h})(a),
\end{align*}
\begin{align*}
(\VV_{-g} \VV_g \VV_{h})(a)&= \overline{\alpha}_{-g}(1)\alpha_h(a)\overline{\alpha}_{-g}(1)=\alpha_h(a)=\alpha_{-g}(\alpha_{g+h}(a))
=(\VV_{-g} \VV_{g+h})(a).
\end{align*}

4) If  $h \leq 0\leq g$ and  $ g\geq -h$, then 
$$
(\VV_g \VV_h \VV_{-h})(a)=\alpha_{g}( \overline{L}_{-h}(1)a)=\al_{g}(a)=\alpha_{g+h}(\alpha_{-h}(a))=(\VV_{g+h} \VV_{h})(a),
$$
\begin{align*}
(\VV_{-g} \VV_g \VV_{h})(a)&= \overline{L}_{g}(1)L_{-h}(a)=L_{-h}(\overline{\alpha}_{-h}(\overline{L}_{g}(1))a)=L_{-h}\big(\overline{\alpha}_{-h}(\overline{L}_{-h}(\overline{L}_{g+h}(1)))a\big)
\\
&= L_{-h}(\overline{\alpha}_{-h}(1) \overline{L}_{g+h}(1) \overline{\alpha}_{-h}(1)a)=L_{-h}( \overline{L}_{g+h}(1)a)
\\
&=L_{-h}( L_{g+h}(\alpha_{g+h}(a)))=L_{g}(\alpha_{g+h}(a))=(\VV_{-g} \VV_{g+h})(a).
\end{align*}

5) If   $h \leq 0\leq g$ and  $ g\leq -h$, then 
\begin{align*}
(\VV_g \VV_h \VV_{-h})(a)&= \alpha_{g}( \overline{L}_{-h}(1)a)=\overline{\alpha}_{g}(\overline{L}_{g}( \overline{L}_{-g-h}(1))\alpha_{g}(a)
\\
&=\overline{\alpha}_{g}(1) \overline{L}_{-g-h}(1)\overline{\alpha}_{g}(1) \alpha_{g}(a)=\overline{L}_{-g-h}(1) \alpha_{g}(a)
\\
&=L_{-g-h}( \alpha_{-h}(a))= (\VV_{g+h} \VV_{h})(a),
\end{align*}
\begin{align*}
(\VV_{-g} \VV_g \VV_{h})(a)&= \overline{L}_{g}(1) L_{-h}(a)=L_{-h}(a)=L_{g}(L_{-g-h}(a))
=(\VV_{-g} \VV_{g+h})(a).
\end{align*}

\end{proof}
\subsection{Fell bundles and $C^*$-algebras associated to semigroup corner systems}
Let $(A,G^+,\alpha,L)$ be a  semigroup corner system. 
The authors of  \cite{Kwa-Leb} (who considered the case with $A$ unital) associated  to $(A,G^+,\alpha,L)$ a Banach $*$-algebra $\ell^1(G^+,\alpha,A)$ and then constructed a faithful representation of $\ell^1(G^+,\alpha,A)$ on a Hilbert space. This regular representation induces a pre-$C^*$-norm on $\ell^1(G^+,\alpha,A)$. Completion of $\ell^1(G^+,\alpha,A)$ in this norm yields the crossed product $A\rtimes_\alpha G^+$, see  \cite[Subsection 6.5]{Kwa-Leb}. In this subsection we generalize this construction to non-unital case, by making  explicit the Fell bundle structure underlying the $*$-Banach algebra $\ell^1(G^+,\alpha,A)$.

We associate to $(A,G^+,\alpha,L)$ a Fell bundle $\B=\{B_g\}_{g\in G}$ defined as follows. For any $t\in G^+$, we denote by $\delta_t$ and $\delta_{-t}$  abstract markers and  consider Banach spaces
\begin{equation}\label{banach spaces for Fell bundel}
B_t:=\{a\delta_t:  a \in A\overline{\al}_t(1)\}, \qquad\quad   B_{-t}:=\{a\delta_{-t}:  a \in \overline{\al}_t(1)A\} 
\end{equation}
naturally isomorphic to $A\overline{\al}_t(1)$ and $\overline{\al}_t(1)A$, respectively.
We define the `star' and the `multiplication' operations by the formulas: 
\begin{equation}\label{invol}
 (a\delta_{g})^*:=a^*\delta_{-g},
\end{equation}

\begin{equation}\label{mul1}
a\delta_{g}\cdot b\delta_{h}
:=\begin{cases} 
a\al_{g}(b)\delta_{g+h} & g,h \geq 0, \\
 L_{-g}(ab)\delta_{g+h}, &   g<0,\,\,h \geq 0,  \,\, g+h \geq 0, \\
 L_{h}(ab)\delta_{g+h}, &  g<0,\,\, h \geq 0,\,\, \, g+h <0,\\
 a\alpha_{g+h}(b)\delta_{g+h}, &   h<0,\,\,g \geq 0,  \,\, g+h \geq 0, \\
\alpha_{-g-h}(a)b\delta_{g+h}, & h<0,\,\, g \geq 0,\,\, \, g+h <0,\\
 \al_{h}(a)b\delta_{g+h} & g,h < 0  .
 \end{cases}
\end{equation}
 To understand where these relation come from, see Remark \ref{dictionary remark} below. The well-definiteness of \eqref{invol} is clear, and with a little more effort, cf. the first part of  proof of \cite[Proposition 4.2]{Kwa-Leb},   it can  also be seen for \eqref{mul1}.
\begin{prop}\label{Fell bundle of corner system}
The family $\B=\{B_g\}_{g\in G}$   of Banach spaces \eqref{banach spaces for Fell bundel} with operations  \eqref{invol}, \eqref{mul1} forms a Fell bundle.
\end{prop}
\begin{proof} The only not obvious axioms of a Fell bundle that we need to check, see, for instance \cite[Definition 16.1]{exel-book}, are:
\begin{equation}\label{relations to be checked to get a Fell bundle}
(a\delta_{g}\cdot b\delta_{h})^*=( b\delta_{h})^*\cdot(a\delta_{g})^*, \qquad a\delta_{g} \cdot (b\delta_{h}\cdot c\delta_{f}) =(a\delta_{g}\cdot b\delta_{h}) \cdot c\delta_{f},
\end{equation}
for all $a\delta_{g}\in B_g$,  $b\delta_{h}\in B_h$,  $c\delta_{f}\in B_f$, $g,h,f \in G^+$. The first  relation in \eqref{relations to be checked to get a Fell bundle} follows from the first part of the proof of \cite[Theorem 4.3]{Kwa-Leb}, but can  also be  inferred directly from \eqref{mul1}. The second relation in \eqref{relations to be checked to get a Fell bundle} was checked in the second  part of the proof of \cite[Theorem 4.3]{Kwa-Leb} in the case when $g+h+f \geq 0$. The case $g+h+f \leq 0$ can be  covered by passing to adjoints.
\end{proof}
\begin{defn}\label{definition of semigroup crossed product}
We call $\B=\{B_g\}_{g\in G}$ constructed above the \emph{Fell bundle associated to the corner system}   $(A,G^+,\alpha,L)$. We define the \emph{crossed product} for $(A,G^+,\alpha,L)$ to be  
$$
A\rtimes_{\alpha,L}G^+:=C^*(\B),
$$
the cross sectional $C^*$-algebra of $\B$. We will identify $A$ with $B_0 \subseteq A\rtimes_{\alpha,L}G^+$.
\end{defn}
The crossed product $A\rtimes_{\alpha,L}G^+$ can be viewed as a  crossed product for the semigroup $\alpha$, for the  semigroup $L$, or as a crossed product for the group interaction $\VV$ given by   \eqref{corner interaction}. To show this we  use the following lemma, which is of its own interest. It is related to the problem of extension of a representation of a semigroup  to a partial representation of a group, studied for instance in \cite{exel-book}, cf. \cite[Definition 31.19]{exel-book}.
\begin{lem}\label{extensions of semigroup homomorphisms}
Any semigroup homomorphism $U:G^+\to B$ to  a multiplicative subsemigroup of a $C^*$-algebra $B$ consisting of  partial isometries  extends to a $*$-partial representation  $V:G\to B$ of $G$.
\end{lem}
\begin{proof}
Assume that $U:G^+\to B$ is a semigroup homomorphism attaining values in partial isometries.  We only need to show, cf. \cite[9.2]{exel-book}, that putting $V_t:=U_t$ and  $V_{-t}:=U_t^*$, for $t\in G^+$, we have
$$
V_g V_h V_{-h}=V_{g+h}V_{-h},   \qquad g,h \in G.
$$ 
Since  the product of two partial isometries is a partial isometry if and only if the corresponding initial and final projections  commute, see \cite[Proposition 12.8]{exel-book},  we conclude that the projections $U_tU_t^*$, $U^*_s U_s$  commute for all $s,t \in G^+$. We have the following cases:

1) If $g,h\geq 0$ or $g,h \leq 0$ then $V_g V_h V_{-h}=V_{g+h}V_{-h}$ holds because  $V$ is a semigroup  homomorphism when restricted to  $G^+$ or to $-G^+$.

2) If  $g \leq 0\leq h$ and  $ h\leq -g$, then 
$$
V_g V_h V_{-h}=U_{-g}^* U_h U_{h}^*= (U_{-g-h}^* U_h^*)U_h U_{h}^*=U_{-g-h}^* U_h^*=V_{g+h}V_{-h}.
$$

3) If   $g \leq 0\leq h$ and  $ h\geq -g$, then 
\begin{align*}
V_g V_h V_{-h}&=U_{-g}^* U_h U_{h}^*=U_{-g}^*(U_{-g} U_{g+h})(U_{g+h}^* U_{-g})=(U_{-g}^*U_{-g}) (U_{g+h}U_{g+h}^*) U_{-g}
\\
&= (U_{g+h}U_{g+h}^*)(U_{-g}^*U_{-g}) U_{-g}=U_{g+h}U_{g+h}^* U_{-g}=U_{g+h}U_{h}^*=V_{g+h}V_{-h}.
\end{align*}

4) If  $h \leq 0\leq g$ and  $ g\geq -h$, then 
$$
V_g V_h V_{-h}=U_{g} U_{-h}^* U_{-h}=U_{g+h}U_{-h} U_{-h}^* U_{-h}=U_{g+h}U_{-h}=V_{g+h}V_{-h}.
$$

5) If   $h \leq 0\leq g$ and  $ g\leq -h$, then 
\begin{align*}
V_g V_h V_{-h}&=U_{g} U_{-h}^* U_{-h}=U_{g} (U_{g}^* U_{-g-h}^*) (U_{-g-h}  U_{g})=(U_{g} U_{g}^*) (U_{-g-h}^* U_{-g-h})  U_{g})
\\
&= (U_{-g-h}^* U_{-g-h}) (U_{g} U_{g}^*) U_{g})=U_{-g-h}^* U_{-g-h} U_{g}=U_{-g-h}^* U_{-h}=V_{g+h}V_{-h}.
\end{align*}

\end{proof}
\begin{prop}\label{Proposition for representations} Suppose $(A,G^+,\alpha,L)$ is a corner system. Let $\pi:A\to B(H)$ be a non-degenerate representation and let $U:G^+\to B(H)$ be a mapping. The following statements are equivalent:
\begin{itemize}
\item[(i)] $U$ is a semigroup homomorphism, and 
$$
U_t\pi(a)U_t^*=\pi(\alpha_t(a)), \qquad  U_t^*\pi(a)U_t=\pi(L_t(a)), \qquad \textrm{for all }a\in A,\,  t\in G^+.
$$
\item[(ii)] $U$ is a semigroup homomorphism, and 
$$
U_t\pi(a)U_t^*=\pi(\alpha_t(a)), \qquad  \pi((\ker\alpha_t)^\bot)\subseteq U_t^*U_t\pi(A),\quad \textrm{for all }a\in A,\,  t\in G^+.
$$

\item[(iii)]  $U$ is a semigroup homomorphism, and 
$$
U_t^*\pi(a)U_t=\pi(L_t(a)), \quad  \pi((\ker L_t|_{MD(L_t)})^{\bot})\subseteq U_t\pi(A)U_t^*,\quad \textrm{for all }a\in A,\,  t\in G^+.
$$
\item[(iv)]  $U$ extends to a $*$-partial representation  $V:G\to B(H)$ of $G$ such that  
$$
V_g\pi(a)V_{g^{-1}}=\pi(\VV_g(a)), \qquad \textrm{for all }a\in A,\,  g\in G,
$$
where $\VV$ is the group interaction given by   \eqref{corner interaction}.
\end{itemize}
\end{prop}
\begin{proof} 

(i)$\Leftrightarrow$(ii). This follows from \cite[Proposition 4.2]{kwa-rever}.

(i)$\Rightarrow$(iii). This is clear because $(\ker L_t|_{MD(L_t)})^{\bot}=\alpha_t(A)$, see \eqref{relation for domain predomain}.

(iii)$\Rightarrow$(i). Let $a\in A$ and  $t\in G^+$. Denote by $\overline{\pi}:M(A)\to B(H)$ the unique extension of $\pi$ to a representation of $M(A)$. Note that $\overline{\pi}(\overline{L}_t(1))=U^*_tU_t$, so in particular $U_t$ is a partial isometry. Thus in view of \eqref{corner endomorphism formula} and \eqref{relation for domain predomain} we get 
\begin{align*}
\pi(\alpha_t(a))&=U_tU_t^*\pi(\alpha_t(a))U_tU_t^*=\pi(L_t(\alpha_t(a)))=U_t\pi(\overline{L_t}(1)a) U_t^*
\\
&=U_t\overline{\pi}(\overline{L}_t(1))\pi(a) U_t^*=U_t\pi(a) U_t^*.
\end{align*}

(i)$\Rightarrow$(iv). This follows from Lemma \ref{extensions of semigroup homomorphisms} since  $\overline{\pi}(\overline{L}_t(1))=U^*_tU_t$ is a projection and hence $U_t$ is a partial isometry, for every $t\in G^+$.

(iv)$\Rightarrow$(i).  For $t\in G^+$, $U_t$ is a partial isometry and we necessarily have $V_{-t}=U_t^*$. Moreover,  $\overline{\pi}(\overline{L}_t(1))=U^*_tU_t$ and hence $\{U^*_tU_t\}_{t\in G^+}$ forms a non-decreasing family of projections. Using this, for any $t,s\in G^+$,  we get
$$
U_t U_s =  U_t U_s U_s^* U_s = U_{t+s} U_{s}^* U_s=  U_{t+s}.
$$ 

This proves the equivalence of conditions (i)-(iv). 

\end{proof}
\begin{thm}\label{Description  by generators and relations} For any corner system $(A,G^+,\alpha,L)$  there is a semigroup homomorphism $u: G^+\to M(A\rtimes_{\alpha,L}G^+)$ taking  values in partial isometries such that we have
\begin{equation}\label{relations for corner system}
u_t a u_t^*=\alpha_t(a), \qquad u_t^* a u_t=L_t(a), \qquad a\in A, \, t\in G^+,
\end{equation}
and  the elements of the form
$$
a=\sum_{t\in F} u_t^*a_{-t} +a_0+  \sum_{x\in F} a_tu_t , \quad \ F\subseteq G^+\setminus\{0\}, \ \ \vert F\vert < \infty,
$$
where $a_{t}\in A \overline{\alpha}_t(1)$, $a_{-t}\in \overline{\alpha}_t(1)A$, form a dense $*$-subalgebra of $A\rtimes_{\alpha,L}G^+$. 

The crossed product $A\rtimes_{\alpha,L}G^+$ is a universal $C^*$-algebra for pairs $(\pi,U)$ satisfying the equivalent conditions (i)-(iv) in Proposition \ref{Proposition for representations}, in the sense that for any such pair the assignments 
\begin{equation}\label{maps to be extended to the Fell bundle}
(\pi\rtimes U)(a_t u_t) \longmapsto \pi(a_t)U_t, \qquad (\pi\rtimes U)(u_t^*a_{-t}) \longmapsto U_t^*\pi(a_{-t}), \qquad t\in G^+,
\end{equation}
extend to a non-degenerate representation $\pi\rtimes U$ of $A\rtimes_{\alpha,L}G^+$  and every non-degenerate representation of $A\rtimes_{\alpha,L}G^+$ arises this way.
\end{thm}
\begin{proof} Recall that we identify $B_0=A\delta_0$ with $A$. In particular, $A$  
 is a non-degenerate $C^*$-subalgebra of $A\rtimes_{\alpha,L}G^+$, that is $A(A\rtimes_{\alpha,L}G^+)=A\rtimes_{\alpha,L}G^+$. Moreover, if we let $t\in G^+$ and $\{\mu_\lambda\}$ 
be  an approximate unit in $A$, then using \eqref{mul1}  we get
$$
\lim_{\lambda} \mu_\lambda\delta_t a\delta_0=\mu_\lambda \alpha_t (a) \delta_t=\alpha_t (a) \delta_t\,\, \textrm{ and }\,\,\lim_{\lambda} (\mu_\lambda\delta_t)^* a\delta_0=\overline{\alpha_t}(1) \mu_\lambda  a \delta_{-t}=\overline{\alpha_t}(1)   a \delta_{-t}. 
$$
Therefore, we conclude that  $\lim_{\lambda} \mu_\lambda \delta_t$ converges strictly to an element of $M(A\rtimes_{\alpha,L}G^+)$. Let us denote it by $u_t$, and note that we have $u_ta\delta_0=\alpha_t (a) \delta_t$ and $u_t^*a\delta_0=\overline{\alpha}_t(1)   a \delta_{-t}$. For any $t,s\in G^+$ we have
\begin{align*}
u_t u_s  (a\delta_0)&=u_t   \alpha_s(a)\delta_s=\lim_{\lambda} \mu_\lambda\delta_t \alpha_s(a)\delta_s=\lim_{\lambda} \mu_\lambda\alpha_{t}(\alpha_s(a))\delta_{t+s}
\\
&=\alpha_{t+s}(a)\delta_{t+s}= u_{t +s}  (a\delta_0).
\end{align*}
It follows that $G^+\ni t \stackrel{u}{\longmapsto} u_t \in M(A\rtimes_{\alpha,L}G^+)$ is a semigroup homomorphism. We have
$$
u_t a\delta_0 u_t^*=\alpha_t(a)\delta_0, \qquad u_t^* a\delta_0 u_t=L_t(a)\delta_0, \qquad a\in A, \, t\in G^+,
$$
which follows from the following calculations, where $b\in A$ is arbitrary:
$$
u_t (a\delta_0 )u_t^*  (b\delta_0)=\alpha_t(a)\delta_t \cdot \overline{\alpha}_t(1)b\delta_{-t}=\alpha_t(a)\overline{\alpha}_t(1)b \delta_0=\alpha_t(a)\delta_0 \cdot b \delta_0,
$$
$$
u_t^* (a\delta_0 )u_t  (b\delta_0)=\overline{\alpha}_t(1)a\delta_{-t} \cdot \alpha_t(b)\delta_{t}=L_t(\overline{\alpha}_t(1)a \alpha_t(b)) \delta_0
=L_t(a )b \delta_0
=L_t(a )\delta_0 \cdot b \delta_0.
$$
This proves the first part of the assertion, because we clearly have 
$a_t\delta_t=a_t u_t$ and $a_{-t}\delta_{-t}=u_{-t}^*a_{-t} $ for any $a_t\delta_t\in B_t$, $a_{-t}\delta_{-t}\in B_{-t}$, $t\in G^+$.

 Assume now that the pair $(\pi,U)$ satisfies   equivalent conditions (i)-(iv). In view of Definition \ref{definition of semigroup crossed product}, to see that the maps in \eqref{maps to be extended to the Fell bundle} extend to a representation of  $A\rtimes_{\alpha,L}G^+$ it suffices to check that 
$$
(\pi\rtimes U)(a\delta_g)^*=(\pi\rtimes U)\big((a \delta_g)^*\big),\quad  (\pi\rtimes U)\big(a \delta_g \cdot b \delta_h\big) =(\pi\rtimes U)(a \delta_g) (\pi\rtimes U)(a \delta_h),
$$
for all $a\delta_g\in B_g$ and $b\delta_h\in B_h$, $g,h\in G$. The first of the above relations is clear and the second one is shown in the proof of \cite[Proposition 5.3]{Kwa-Leb}.

If  $\sigma:A\rtimes_{\alpha,L}G^+\to B(H)$ is an arbitrary  non-degenerate representation it extends uniquely to a representation $\overline{\sigma}:M (A\rtimes_{\alpha,L}G^+)\to B(H)$.  Clearly, putting  
$$
\pi(a):=\sigma(a\delta_0),\,\,\,   a\in A, \qquad U_t:=\overline{\sigma}(u_t),\,\,\,  t\in G^+,
$$
we get a pair $(\pi,U)$ satisfying  equivalent conditions (i)-(iv) and  such that $\sigma=\pi\rtimes U$.
\end{proof}
\begin{rem}\label{dictionary remark}
Theorem \ref{Description  by generators and relations} allow us to assume the  following `dictionary' determining the structure of $A\rtimes_{\alpha,L}G^+$: we have
$$
a\overline{\alpha}_t(1)\delta_t= a u_t , \qquad\quad  (\overline{\alpha}_t(1)a)\delta_{-t} =u_t^*a, \qquad t\in G^+, a\in A,
$$
and then multiplication of these spanning elements is determined by relations \eqref{relations for corner system}. 
In particular, these relations imply the following commutation relations:
$$
au_{t}^*=u^{*}_t\alpha_{t}(a), \qquad u_{t} a=\alpha_{t}(a)u_{t}, \qquad  a\in A, \, t\in G^+.
$$
Note also that the above description makes explicit the fact that the $C^*$-algebra $A\rtimes_{\alpha,L}G^+$ coincides with the crossed product studied in \cite{Kwa-Leb} when $A$ is untial, see \cite[Theorem 5.4]{Kwa-Leb}. 
\end{rem}
\begin{cor}\label{Exel-Larsen crossed product} Let $(A,G^+,\alpha,L)$ be a corner system. The crossed product $A\rtimes_{\alpha,L}G^+$ is naturally isomorphic to Exel-Larsen crossed product introduced in \cite[Definition 2.2]{Larsen}.
\end{cor}
\begin{proof} By \cite[Proposition 4.3]{Larsen}, we may identify  Exel-Larsen crossed product for $(A,G^+,\alpha,L)$ with the quotient $C^*$-algebra $\TT_X/\I_K$ where $\TT_X$ is the Toeplitz algebra of a  product system $X=\bigsqcup_{t\in G^{+}} X_t$ naturally associated to $(A,G^+,\alpha,L)$, see \cite[Proposition 3.1]{Larsen}, and $\I_K$ is an ideal in $\TT_X$  generated by differences 
$$
i_X(a) - i_X^{(t)}(\phi_t(a)), \qquad \textrm{for }a\in K_t:=A\alpha_t(A)A\cap \phi_t^{-1}(\K(X_t)) \textrm{ and }  t \in G^+, 
$$
where $i_X:X\to \TT_X$ is a universal representation of $X$ in $\TT_X$, $i^{(t)}_X:\K(X_t)\to \TT_X$ is the associated homomorphism, and $\phi_t:A\to \LL(X_t)$ is the left action of $A$ on $X_t$, $t\in G^+$. For more details see \cite{Larsen}. 
As shown in \cite[Proposition 4.3]{Larsen}, there exists a semigroup homomorphism $i_{G^+}:G^+\to M(\TT_X)$ such that putting  $i_A:=i_X|_{X_0}=i_X|_{A}$ we get that 
$$
\TT_X=C^*(\bigcup_{t\in G^{+}} i_A(A)i_{G^+}(t)),\quad\textrm{and}\quad 
i_{G^+}(t)^*i_A(a)i_{G^+}(t)=i_A(L_t(a)), \qquad a\in A,\,\, t\in G^+.
$$
Moreover, the latter picture of $\TT_X$ is universal (note that the author of \cite{Larsen} considers also additional relations $i_{G^+}(t)i_A(a)=i_A(\alpha_t(a))i_{G^+}(t)$, $a\in A$, $t\in G^+$, but by \cite[Proposition 4.3]{kwa-exel} they are automatic). In particular, the assignments 
$$
i_A(a)i_{G^+}(t) \longmapsto a u_t, \qquad a\in A,\, t\in G^+,
$$
determine a well-defined surjective homomorphism $\Phi:\TT_X\to A\rtimes_{\alpha,L}G^+$. In the proof of \cite[Theorem 4.22]{kwa-exel} it is shown that, under our assumptions, we have $A\alpha_t(A)A\subseteq  \phi_t^{-1}(\K(X_t))$. Hence $K_t=A\alpha_t(A)A$. Another argument in the proof of \cite[Theorem 4.22]{kwa-exel}, cf. also \cite[Lemma 3.12]{kwa-exel}, shows that 
$$
i_X^{(t)}(\phi_t(\alpha_t(a)))=i_{G^+}(t) i_A(a) i_{G^+}(t)^*, \qquad a\in A,\, t\in G^+.
$$
This allows us to conclude that $\I_K$ is generated by the differences 
$$
i_A(\alpha_t(a)) - i_{G^+}(t) i_A(a) i_{G^+}(t)^*, \qquad \textrm{for }a\in A \textrm{ and }  t \in G^+.
$$
Accordingly, $\I_K\subseteq \ker\Phi$ and $\Phi$ factors through to a surjective homomorphism $\Psi:\TT_X/\I_K\to A\rtimes_{\alpha,L}G^+$. By the universality of $A\rtimes_{\alpha,L}G^+$, $\Psi$ admits an inverse. Hence $\TT_X/\I_K\cong A\rtimes_{\alpha,L}G^+$. 
\end{proof}

\subsection{Partial dynamical systems dual to  semigroup corner systems}
Let $(A,G^+,\alpha,L)$ be a  semigroup corner system. For each  $ t\in G^+$, $L_t(A)$ is an ideal in $A$ and  $\alpha_t(A)$ is a hereditary subalgebra of $A$. Thus we  may treat $\widehat{L_t(A)}$ and $\widehat{\alpha_t(A)}$ as open subsets of $\SA$.  Then the mutually inverse isomorphisms $\alpha_t:L_t(A)\to \alpha_t(A)$ and  $L_t:\alpha_t(A)\to L_t(A)$ give rise to partial homeomorphisms of  $\SA$:
$$
\widehat{\alpha}_t([\pi]):=[\pi\circ \alpha_t], \qquad \widehat{L}_{t}([\pi]):=[\pi\circ L_t],
$$ 
cf. \cite[Section 4.5]{kwa-rever} for a detailed description of these maps. Using  the group interaction \eqref{corner interaction} we can express it in a more symmetric way. Namely, we have
homeomorphisms
$$
\widehat{\VV}_g:\widehat{\VV_g(A)}\to \widehat{\VV_{-g}(A)} \qquad\textrm{where}\qquad \widehat{\VV}_g([\pi])=[\pi\circ \VV_g], \qquad g \in G,
$$
and we assume the identification $\widehat{\VV_g(A)}=\{[\pi]\in \SA: \pi(\VV_g(A))\neq 0\}$.
\begin{lem}\label{lemma for duals to interactions} The family  $(\{\widehat{\VV_g(A)}\}_{g\in G}, \{\widehat{\VV}_g\}_{g\in G})$ is a partial action on $\SA$ which  coincides with the opposite to the partial action dual to the Fell bundle $\B=\{B_g\}_{g\in G}$ associated to  $(A,G^+,\alpha,L)$.

In particular,   $(\{\widehat{\VV_g(A)}\}_{g\in G}, \{\widehat{\VV}_g\}_{g\in G})$ is a lift of a  partial action  $(\{\Prim (\VV_g(A))\}_{g\in G}, \{\widecheck{\VV}_g\}_{g\in G})$ on $\Prim(A)$ given by 
$$
\widecheck{\VV}_g(P)=A\VV_{-g}(P)A, \qquad g\in G;
$$
in other words, $\widecheck{\VV}_{-t}(P)=A\alpha_t(P)A$ and $\widecheck{\VV}_t(P)=L_t(P)$, for $t\in G^+$.
\end{lem}
\begin{proof} For any $t\in G^+$, we have $A\alpha_t(A)A=A\overline{\alpha}_t(1)A=B_t\cdot B_{-t}=D_t$ and $L_t(A)=L_t(\overline{\alpha}_t(1)A\overline{\alpha}_t(1))=B_{-t}\cdot B_t=D_{-t} $. Thus, with our identifications,  we have $\widehat{L_t(A)}=\widehat{D}_{-t}$ and
$$
\widehat{\alpha_t(A)}=\{[\pi]\in \SA: \pi(\alpha_t(A))\neq 0\}=\{[\pi]\in \SA: \pi(A\alpha_t(A)A)\neq 0\}=\widehat{D}_t.
$$
Let now $\pi:A\to \B(H)$ be   an irreducible representation   such that  $\pi(\alpha_t(A))\neq 0$.  Then $\widehat{\alpha}_t([\pi])$ is the equivalence class of the representation $\pi\circ\alpha_t:A\to B(\pi(\alpha_t(A))H)$.  For any $a_i\delta_{-t} \in B_{-t}=\overline{\alpha}_t(1)A\delta_{-t}$, $h_i\in H$,  $i=1,...,n$,  we have 
$$
\|\sum_{i} a_i \delta_{-t}  \otimes_\pi h_i\|^2=\|\sum_{i,j} \langle h_i, \pi( a_i^*a_j)h_j\rangle_p\|= \|\sum_{i} \pi(a_i)h_i\|^2.
$$
Since $\pi(\alpha_t(A))H=\pi(\overline{\alpha}_t(1)A)H$, we see  that $a\delta_{-t}\otimes_\pi h \mapsto \pi(a)h $  yields a unitary operator $U:B_{-t}\otimes_{\pi} H \to \pi(\alpha_t(A))H$. Furthermore, for $a\in A$, $b\in \alpha_t(A)$ and $h\in H$ we have
\begin{align*}
 \Big( B_{-t}\dashind^{D_t}_{D_{-t}}(\pi)(a) U^*\Big) \pi(b)h&= B_{-t}\dashind^{D_t}_{D_{-t}}(\pi)(a)\,   (b\delta_{-t} \otimes_\pi h)= (\alpha_t(a)b)\otimes_\pi h
\\
&=(\alpha_t(a)b\delta_{-t})\otimes_\pi h =  \Big(U^* (\pi\circ \alpha_t)(a)\Big) \pi(b)h.
\end{align*}
Hence   $U$  intertwines $B_{-t}\dashind^{D_t}_{D_{-t}}$ and  $\pi\circ\alpha_t$. Accordingly,  $\h_{-t}=\widehat{\alpha}_t$ and  $\h_{t}=\widehat{\alpha}_t^{-1}=\widehat{L}_{t}$. This proves the first part of the assertion. To show the second part, we use the `dictionary' from Remark \ref{dictionary remark}. Then it is immediate that the corresponding Rieffel homeomorphisms, cf. \eqref{Rieffel homeomorphism},  are given by 
$$
h_t(I)=(A u_t)  I (u_t^*A) = A\al_t(I)A, 
$$
$$
h_{-t}(I)=(u_t^*A)   I (A u_t)  = L_t(I)=AL_t(I)A,
$$
for any $I\in \I(A)$ and $t\in G^+$.
\end{proof}

\begin{defn}
We call  $(\{\widehat{\VV_g(A)}\}_{g\in G}, \{\widehat{\VV}_g\}_{g\in G})$ and  $(\{\Prim (\VV_g(A))\}_{g\in G}, \{\widecheck{\VV}_g\}_{g\in G})$  
described above  \emph{partial dynamical systems dual to the interaction} $\VV$. 
\end{defn}
Before we state the main result of this subsection we need a lemma and a definition. 
\begin{lem}\label{lemma for invariant ideals}
If $I\in \I(A)$, then the following conditions are equivalent:
\begin{itemize}
\item[(i)] $\alpha_t(I)=\overline{\alpha}_{t}(1)I\overline{\alpha}_{t}(1)$ for every $t\in G^+$,
\item[(ii)] $L_t(I)=\overline{L}_{t}(1)I$  for every $t\in G^+$,
\item[(iii)] $\VV_g(I)\subseteq I$  for every $g\in G$, where $\VV$ is the group interaction
\item[(iv)] $\widehat{I}$ is invariant under the partial action $(\{\widehat{\VV_g(A)}\}_{g\in G}, \{\widehat{\VV}_g\}_{g\in G})$.
\end{itemize}
In particular, if the above equivalent conditions hold, then we have a quotient semigroup corner system $(A/I,G^+,\alpha^I,L^I)$ where $\alpha_t^I(a)=a +I$ and   $L^I_t(a)=a +I$, and its associated group interaction  is given by  $\VV_g^I(a)=a + I$, for $a\in A$, $t\in G^+$, $g\in G$.
\end{lem}
\begin{proof} The equivalence of conditions (i)-(iv)  follows from \cite[Lemma 2.22]{kwa-interact}, which was proved for unital  $A$  but the proof carries over to the general case. The second part of the assertion is now clear.
\end{proof}
\begin{defn}\label{invariant ideals for corner systems definition} If $I\in \I(A)$ satisfies the equivalent conditions   (i)-(iv) in Lemma \ref{lemma for invariant ideals} we  call $I$  an \emph{invariant ideal} for $(A,G^+,\alpha,L)$.

\end{defn}
\begin{thm}\label{first of main results of the section}
Let $(A,G^+,\alpha,L)$ be a semigroup corner system, $\{\VV_g\}_{g\in G}$ its associated group interaction, and $\widehat{\VV}=(\{\widehat{\VV_g(A)}\}_{g\in G}, \{\widehat{\VV}_g\}_{g\in G})$  and $\widecheck{\VV}=(\{\Prim (\VV_g(A))\}_{g\in G}, \{\widecheck{\VV}_g\}_{g\in G})$ the dual partial actions. Then 
\begin{itemize}
\item[(i)] If  $\widehat{\VV}$ is topologically free, then every pair $(\pi,U)$ satisfying  equivalent conditions (i)-(iv) in Proposition \ref{Proposition for representations} such that  $\pi$ is faithful gives rise to a faithful representation $\pi\rtimes U$ of $A\rtimes_{\alpha,L}G^+$.
\item[(ii)] If  $\widehat{\VV}$ is residually topologically free, then the map
$
 J\mapsto J\cap A
$
is a homeomorphism from $\I(A\rtimes_{\alpha,L}G^+)$ onto the subspace of $\I(A)$ consisting of  invariant ideals for $(A,G^+,\alpha,L)$.
\item[(iii)] If  $\widehat{\VV}$ is residually topologically free, $A$ is separable and $G^+$ is countable then we have a homeomorphism
$$
\Prim(A\rtimes_{\alpha,L}G^+)\cong \OO(\Prim A), 
$$
where $\OO(\Prim A)$ is the quasi-orbit space associated to  $\widecheck{\VV}$.
\end{itemize}
\end{thm}
\begin{proof} Since $G$ is amenable we have $A\rtimes_{\alpha,L}G^+=C^*(\B)=C^*_r(\B)$ for the associated Fell bundle.

(i) By Lemma \ref{lemma for duals to interactions} and Theorem \ref{uniqueness theorem for fell bundles}, $A\rtimes_{\alpha,L}G^+=C^*_r(\B)$ has the intersection property. Since $\ker(\pi\rtimes U)\cap A=\ker\pi=\{0\}$, we conclude that $\ker(\pi\rtimes U)=\{0\}$.

(ii) Apply Lemma \ref{lemma for duals to interactions} and Corollary  \ref{residual topological freeness corollary}.

(iii) Apply Lemma \ref{lemma for duals to interactions} and Theorem \ref{Primitive ideal space description}.
 
 \end{proof}

\subsection{Purely infinite crossed products  for semigroup corner systems}

We fix a semigroup corner system $(A,G^+,\alpha,L)$. Let $\B=\{B_g\}_{g\in G}$ be its associated Fell bundle, and $\VV=\{\VV_g\}_{g\in G}$ be its associated group interaction.
\begin{lem}\label{aperiodicity for corner systems lemma}
The following conditions are equivalent
\begin{itemize}
\item[(i)]  $\B=\{B_g\}_{g\in G}$ is aperiodic,
\item[(ii)] for  each $t\in G^+\setminus\{0\}$, each $a\in A$ and every hereditary
subalgebra $D$ of $A$ 
$$\inf \{\|d a \alpha_t(d)\| : d\in D,\,d \geq 0,\,\, \|d\|=1\}=0.$$
\item[(iii)] for  each $t\in G^+\setminus\{0\}$, each $a\in A$ and every hereditary
subalgebra $D$ of $A$ 
$$\inf \{\|d L_t((da)^*da)d\| : d\in D,\,d \geq 0,\,\, \|d\|=1\}=0.$$

\item[(iv)] for  each $g\in G\setminus\{0\}$, each $a\in A$ and every hereditary
subalgebra $D$ of $A$ 
$$\inf \{\|d \VV_g((da)^*da)d\| : d\in D,\,d \geq 0,\,\, \|d\|=1\}=0.$$

\end{itemize}
\end{lem}  
\begin{proof} Let  $t\in G^+\setminus\{0\}$ and  $a,d\in A$ where  $d \geq 0$.  Since 
$
\|d  \big(a\overline{\alpha}_t(1)\delta_t\big) d  \|=\|d a\alpha_t(d)\delta_t \|=\|d a\alpha_t(d)\|$ and  $\|d  \big(\overline{\alpha}_t(1) a\delta_{-t}\big) d  \|=\|\alpha_t(d) d a\delta_{-t} \|=\|d a^*\alpha_t(d)\|$ we see  that (i)$\Leftrightarrow$(ii).
Since 
$$
\|d a \alpha_t(d)\|^2=\| \alpha_t(d) a^*d d a \alpha_t(d)\|= \|L_t (\alpha_t(d) (da)^*da \alpha_t(d))\|=\|d L_t((da)^*da)d\|
$$
we get  (ii)$\Leftrightarrow$(iii). The implication  (iv)$\Rightarrow$(iii) is clear. Moreover, $\big($(ii)$\Leftrightarrow$(iii)$\big)$ $\Rightarrow$(iv) because  
$$
\|d \alpha_t((da)^*da)d\|=\|d \alpha_t(a^*) \alpha_t(d)\|^2.
$$
\end{proof}
\begin{defn}\label{aperiodicity for corner systems definition} We  say that a semigroup corner system $(A,G^+,\alpha,L)$ is \emph{aperiodic} if the equivalent conditions in Lemma \ref{aperiodicity for corner systems lemma} are satisfied. We say $(A,G^+,\alpha,L)$  is \emph{residually aperiodic} if the quotient system $(A/I,G^+,\alpha^I,L^I)$ is aperiodic  for every  invariant  ideal $I$ for $(A,G^+,\alpha,L)$.

\end{defn}
Now, we formulate notions of  residually infinitess and paradoxicality for corner systems. 
\begin{defn}\label{paradoxical definition} Let $a \in A^+\setminus\{0\}$. We  say that   $a$ is \emph{infinite} for  $(A,G^+,\alpha,L)$ if there is $b\in A^+\setminus\{0\}$ such that for any $\varepsilon >0$ there are elements  $t_1,...,t_{n+m}\in G^+$, and  $a_{\pm k} \in aA$,  $k=  1,  2,..., n+m$
such that
\begin{equation}\label{relation 2}
 a\approx_\varepsilon\sum_{k=1}^{n} \alpha_{t_{k}}(a_{-k}^*a_{-k}) + L_{t_{k}}(a_{k}^*a_{k}), \quad b\approx_\varepsilon\sum_{k= n+1}^{n+m} \alpha_{t_k}(a_{-k}^*a_{-k}) + L_{t_{k}}(a_{k}^*a_{k}),
\end{equation}
\begin{align}
\|a_{k}^* a_{l}\|< \frac{\varepsilon}{\max\{n^2, m^2\}} \quad\textrm{ for all }\,\,l,k =\pm 1,..., \pm(n+m), \,\, k\neq l \label{relation 3}.
\end{align}
If the above conditions hold for $b$ equal to $a$ then we  say that   $a$  is \emph{paradoxical} for $(A,G^+,\alpha,L)$. We say that $a$ is \emph{residually infinite} for  $(A,G^+,\alpha,L)$ if for every invariant ideal $I$ for $(A,G^+,\alpha,L)$ either $a\in I$ or the image of $a$ in $A/I$ is infinite for $(A/I,G^+,\alpha^I,L^I)$. \end{defn}
\begin{prop}\label{justification of paradoxicality}
If  $a \in A^+\setminus\{0\}$  is residually infinite (resp.  paradoxical) for $(A,G^+,\alpha,L)$ then it is residually infinite (resp. paradoxical) for the associated Fell bundle. 
\end{prop}
\begin{proof} Suppose that $a\in A^+\setminus\{0\}$ is  infinite for  $(A,G^+,\alpha,L)$. Let $\varepsilon >0$ and choose $t_1,...,t_{n+m}\in G^+$, and  $a_{\pm k} \in aA$,  $k=  1,  2,..., n+m$
 as in Definition \ref{paradoxical definition} but for $\varepsilon/4$. We use the `dictionary' from Remark \ref{dictionary remark} and put
$$
b_k:=a_k u_{t_k}, \qquad b_{n+k}:=a_{-k}u_{t_k}^*, \qquad \textrm{for }k=1,...,n;
$$
$$
b_{n+k}:=a_ku_{t_k}, \qquad b_{n+m+k}:=a_{-k}u_{t_k}^*, \qquad \textrm{for }k=n+1,...,m.
$$
Clearly, $b_i \in a B_{s_i}$, for $i=1,...,2(n+m)$, where  $s_i:=t_i$,  $s_{n+i} :=-t_i$ for $i=1,...,n$; and $s_ {n+i}:=t_i$,  $s_{n+m+i}:=-t_i$ for $i=n+1,...,m$. Let $i,j =1,...,2(n+m)$. Assume that $i\neq j$. Then  $b_i= a_k u_{t_k}$ or $a_k u_{t_{-k}}^*$ and $b_j= a_l u_{t_l}$ or $b_j= a_l u_{t_{-l}}^*$. In any case $k\neq l$, and thus  by \eqref{relation 3} we get
$$
\|b_i^* b_j\| \leq \|a_k^* a_l\|  < \frac{\varepsilon}{\max\{(2n)^2, (2m)^2\}}.
$$
On the other hand, we have 
$$
\sum_{i=1}^{2n} b_i^* b_i=  \sum_{k=1}^{n} (a_k u_{t_k})^* a_k u_{t_k} +\sum_{k=1}^{n} (a_{-k}u_{t_k}^*)^* a_{-k}u_{t_k}^*=\sum_{k=1}^{n} \alpha_{t_{k}}(a_{-k}^*a_{-k}) + L_{t_{k}}(a_{k}^*a_{k}),
$$
and similarly  $\sum_{i=2n+1}^{2m} b_i^* b_i=\sum_{k=n+1}^{m} \alpha_{t_{k}}(a_{-k}^*a_{-k}) + L_{t_{k}}(a_{k}^*a_{k})$. Thus using \eqref{relation 2} we get
$$
a\approx_\varepsilon\sum_{i=1}^{2n} b_i^* b_i, \qquad b\approx_\varepsilon \sum_{i=2n+1}^{2m} b_i^* b_i.
$$
Hence $a$ is infinite for $\B=\{\B_g\}_{g\in G}$. 

Replacing, in the above argument, $b$ with $a$ one obtains that if  $a$ is  paradoxical for $(A,G^+,\alpha,L)$ then it paradoxical for $\B$. Using the fact that invariant ideals for $(A,G^+,\alpha,L)$ and $\B$-invariant ideals coincide, see Lemma \ref{lemma for duals to interactions}, one gets that if  $a$ is residually infinite for $(A,G^+,\alpha,L)$ then $a$ is residually $\B$-infinite.
\end{proof}
Now we are ready to state the conditions implying pure infiniteness of $A\rtimes_{\alpha,L}G^+$. 
\begin{thm}\label{pure infiniteness for paradoxical corner systems} Suppose that  $(A,G^+,\alpha,L)$  is a residually aperiodic semigroup corner system  and one of the following two conditions holds 
\begin{itemize} 
\item[(i)] $A$ contains finitely many invariant ideals for $(A,G^+,\alpha,L)$ and every element in $A^+\setminus\{0\}$ is  Cuntz equivalent to a residually infinite element for $(A,G^+,\alpha,L)$,
\item[(ii)] $A$ has the ideal property and every element in $A^+\setminus\{0\}$ is  Cuntz equivalent to a residually infinite element for $(A,G^+,\alpha,L)$,
\item[(iii)] $A$ is of real rank zero  and every non-zero projection in $A$ is Cuntz equivalent to a residually infinite element for $(A,G^+,\alpha,L)$.
\end{itemize}
Then  $A\rtimes_{\alpha,L}G^+$ has the ideal property and is purely infinite. 
\end{thm}
\begin{proof}
Apply Proposition \ref{justification of paradoxicality} and Theorem \ref{pure infiniteness for paradoxical Fell bundles}.
\end{proof}
\begin{rem}\label{Eduard remark}  The main result of \cite{Ortega-Pardo} is a criterion of pure infiniteness of Stacey's crossed product $A\rtimes_\alpha \N$ of  a unital separable $C^*$-algebra  $A$ of real rank zero by an injective endomorphism $\alpha:A\to A$. This result can be deduced from Theorem \ref{pure infiniteness for paradoxical corner systems}. Indeed, the authors of \cite{Ortega-Pardo} used the fact that   $A\rtimes_\alpha \N$ is Morita equivalent to the crossed product $B\rtimes_\beta \Z$ by an automorphism $\beta:B\to B$, where $B$ is a separable $C^*$-algebra of real rank zero. They  assumed that $\beta$ satisfies the residual Rokhlin$^*$ property, cf. \cite[Definition 2.1]{s} and $\alpha$ (and therefore also $\beta$) residually contracts projections, cf. \cite[Definition 3.2]{Ortega-Pardo}. By \cite[Corollary 2.22]{s} and \cite[Theorem 10.4]{OlPe},  the former property is equivalent to residual aperiodicity of $(B, \N, \{\beta^n\}_{n\in \N}, \{\beta^{-n}\}_{n\in \N})$. The latter property readily implies that every projection in $B$ is residually infinite for $(B, \N, \{\beta^n\}_{n\in \N}, \{\beta^{-n}\}_{n\in \N})$.  Hence Theorem \ref{pure infiniteness for paradoxical corner systems}(iii) applies to $(B, \N, \{\beta^n\}_{n\in \N}, \{\beta^{-n}\}_{n\in \N})$.

\end{rem}

\end{document}